\documentclass[11pt]{article}

\usepackage[margin=1.3in]{geometry}
\usepackage{mathtools}
\usepackage{fullpage}
\usepackage[dvips]{graphics}
\usepackage{epsfig}
\usepackage{multirow}
\usepackage{array}
\usepackage{extarrows}
\usepackage{graphicx}
\mathtoolsset{showonlyrefs}

\usepackage[symbol]{footmisc}

\usepackage{comment}
\usepackage[titletoc,title]{appendix}
\numberwithin{equation}{section}
\usepackage{amsthm}
\usepackage{enumerate}
\usepackage{amsmath}

\usepackage{amssymb,amsmath}
\usepackage{latexsym}
\usepackage[usenames]{xcolor}
\usepackage{pgf}
\usepackage[numbers]{natbib}

\usepackage{hyperref}
\hypersetup{
    colorlinks,%
    citecolor=black,%
    filecolor=black,%
    linkcolor=black,%
    urlcolor=black
}
\usepackage{subfig}
\numberwithin{equation}{section}

\newtheorem{definition}{ \noindent D{\footnotesize EFINITION}}[section]
\newtheorem{theorem}{ \noindent T{\footnotesize HEOREM}}
\newtheorem{prop}{ \noindent P{\footnotesize ROPOSITION}}[section]
\newtheorem{lemma}{ \noindent L{\footnotesize EMMA}}[section]
\newtheorem{coro}{ \noindent C{\footnotesize OROLLARY}}
\newtheorem{remark}{ \noindent R{\footnotesize EMARK}}[section]

\newtheorem{assumption}{ \noindent A{\footnotesize SSUMPTION}}

\newcommand{\E}{\mathbb{E}}
\newcommand{\Prob}{\mathbb{P}}
\renewcommand{\Pr}[1]{\mathbb{P}\left(#1\right)}
\newcommand{\R}{\mathbb{R}}

\renewcommand\Im{\operatorname{Im}}

\newcommand{\Var}{\mathrm{Var}}
\newcommand{\ii}{\mathrm{i}}

\DeclareMathOperator{\tr}{tr}

\begin{document}

\title{Sparse Hanson-Wright Inequalities with Applications}
\author{Yiyun He$^*$, Ke Wang$^\dagger$ and Yizhe Zhu$^\S$}

\date{}
\maketitle

\begin{abstract}
We derive new Hanson-Wright-type inequalities tailored to the quadratic forms of random vectors with sparse independent components. Specifically, we consider cases where the components of the random vector are sparse $\alpha$-subexponential random variables with $\alpha>0$. When $\alpha=\infty$, these inequalities can be seen as quadratic generalizations of the classical Bernstein and Bennett inequalities for sparse bounded random vectors. To establish this quadratic generalization, we also develop new Bernstein-type and Bennett-type inequalities for linear forms of sparse  $\alpha$-subexponential random variables that go beyond the bounded case $(\alpha=\infty)$. Our proof relies on a novel combinatorial method for estimating the moments of both random linear forms and quadratic forms.

We present two key applications of these new sparse Hanson-Wright inequalities: (1) A local law and complete eigenvector delocalization for sparse $\alpha$-subexponential Hermitian random matrices, generalizing the result of He et al. (2019) beyond sparse Bernoulli random matrices. To the best of our knowledge, this is the first local law and complete delocalization result for sparse $\alpha$-subexponential random matrices down to the near-optimal sparsity $p\geq \frac{\mathrm{polylog}(n)}{n}$ when $\alpha\in (0,2)$ as well as for unbounded sparse sub-gaussian random matrices down to the optimal sparsity $p\gtrsim \frac{\log n}{n}.$  (2) Concentration of the Euclidean norm for the linear transformation of a sparse $\alpha$-subexponential random vector, improving on the results of G{\"o}tze et al. (2021) for sparse sub-exponential random vectors.
\end{abstract}

\footnotetext[1]{Department of Mathematics, University of California San Diego, La Jolla, CA 92093, USA,yih130@ucsd.edu}
\footnotetext[2]{Department of Mathematics, Hong Kong University of Science and Technology, Clear Water Bay, Kowloon, Hong Kong, kewang@ust.hk. }
\footnotetext[4]{Department of Mathematics, University of Southern California, 3620 Vermont Avenue
Los Angeles, CA 90089, USA, yizhezhu@usc.edu. }


\section{Introduction}\label{sec:intro}

The Hanson-Wright-type inequalities concern the concentration of the quadratic forms of random vectors. More specifically, consider a random vector $X=(X_1,\cdots,X_n)^T$ with independent components and a deterministic matrix $A=(a_{ij})_{1\le i,j\le n}$. The primary focus is on understanding the concentration of the quadratic form $$S_n:=\sum_{i,j} a_{ij} X_i X_j = X^T A X$$ around its expectation $\mathbb{E}(S_n)=\sum_{i,j} a_{ij} \mathbb{E}(X_i X_j)=\E(X^T A X)$.

Hanson and Wright \cite{HW71,Wright73} consider the case where $X_i$'s are centered sub-gaussian random variables. A random variable $X$ is \emph{sub-gaussian} if $\Prob(|X-\E(X)|>t )\le 2\exp(-t^2/K^2)$ for any $t> 0$. Its \emph{sub-gaussian norm} is defined as \[\|X\|_{\Psi_2}:= \inf \{ t>0 : \E(\exp(|X|/t)^2) \le 2 \} <\infty.\]  Rudelson and Vershynin \cite{RV13} provide a modern proof of the Hanson-Wright inequality, improving the dependence on the matrix operator norm in the tail bound from \cite{HW71,Wright73}. Specifically, they show that if $\max_{1\le i \le n} \|X_i\|_{\Psi_2} \le K$, then for any $t\ge 0$, 
\begin{align}\label{hw}
\Prob\left(|S_n-\mathbb E (S_n)|\geq t \right)\leq 2 \exp\left(-c\min\left\{\frac{t^2}{K^4 \|A\|_F^2}, \frac{t}{K^2 \|A\|}\right\} \right).
\end{align}

The Hanson-Wright inequality has become a fundamental tool in the study of high-dimensional probability \cite{vershynin2018high}, random matrix theory \cite{VW15}, and signal processing \cite{krahmer2011new}. However, the inequality~\eqref{hw} is limited to sub-gaussian vectors, and it does not capture the variance dependence of $X$. If we consider  a random vector with centered Ber$(p)$ entries where $p\to 0$ as $n\to \infty$, the inequality \eqref{hw} holds with $K\asymp |\log p|^{-1/2}$ \cite{ostrovsky2014exact}.
Given a Bernoulli random vector $X$, the dependence on $p$ of a linear form $\sum_{i=1}^n a_iX_i$ is well-captured by Bennett's inequality  (or Chernoff bound), which plays a fundamental role in various applications in random graph theory and theoretical computer science \cite{alon2016probabilistic}.
These observations naturally lead to the following question:

\begin{center}
 \textit{Is there a generalization of Bennett's inequality for the quadratic forms of sparse random vectors?}   
\end{center}

In this paper, we aim to derive Hanson-Wright inequalities for random variables with a sparse structure, and our new concentration inequalities capture the additional log factors analogous to the improvement from Bernstein's inequality to Bennett's inequality for linear forms of random vectors. 

A canonical example takes the form $X_i= x_i \xi_i$, where $x_i\sim \mathrm{Ber}(p_i)$ and all $x_i,\xi_i$ being independent. The $\xi_i$'s are  \emph{$\alpha$-subexponential} random variables (also known as sub-Weibull of order $\alpha$), characterized by a parameter $\alpha > 0$ such that for any $t > 0$, $\Prob(|\xi_i| \ge t) \le 2\exp(-t^{\alpha}/K_1^{\alpha})$. These random variables can be equivalently defined using either the Orlicz norm $\|\xi_i\|_{\Psi\alpha}<\infty$ or the $L_r$ norms: $\|\xi_i\|_{L_r} \le K_2 r^{1/\alpha}$. The constants $K_1$, $K_2$, and $\|\xi_i\|_{\Psi_\alpha}$ differ only by factors depending on $\alpha$. It can be readily verified that the moments of $X_i$ satisfy $\E(|X_i|^r) \le p_i (K_2r^{\frac{1}{\alpha}})^r$.

The class of $\alpha$-subexponential random variables includes sub-gaussian ($\alpha = 2$), subexponential ($\alpha = 1$), and heavier-tailed distributions ($\alpha \in (0,1)$). We also allow $\alpha = \infty$, which corresponds to bounded random variables satisfying $|\xi_i| \leq K$ almost surely (or equivalently, $\|\xi_i\|_{L_r} \leq K$ for all $r \geq 1$). The concentration behavior of $\alpha$-subexponential random variables exhibits a transition at $\alpha=1$, with $\alpha\ge 1$ corresponding to log-concave tails and $\alpha\le 1$ to log-convex tails. Our primary interest lies in the sparse regime where $p_i \ll 1$.

More generally, we study \emph{sparse $\alpha$-subexponential random variables}, which we characterize through their moments:
\begin{definition}[Sparse $\alpha$-subexponential random variables] A random variable $X$ is said to be a sparse $\alpha$-subexponential random variable for some $\alpha\in(0,\infty]$ if
\begin{align}\label{def:rv}
\E(|X|^r) \le p \left(Kr^{\frac{1}{\alpha}}\right)^r
\end{align}
for all $r\ge 1$ and some $K>0$, where $p\in(0,1]$ is the sparsity parameter.
\end{definition}
We note that one could potentially derive concentration inequalities by first reducing to the special case where $x_i\sim \mathrm{Ber}(p_i)$ and  $\xi_i$ is a symmetric
$\alpha$-Weibull random variable using standard symmetrization and contraction principles (see \cite[Theorem 1]{kwapien1987decoupling} for instance). While this approach might lead to simpler proofs, our current approach works directly with the general definition of sparse $\alpha$-subexponential random variables. This generality makes it easier to apply our results in practice.

Random vectors with sparse $\alpha$-subexponential entries arise in various areas, including large deviation theory \cite{ganguly2024spectral}, random matrix theory \cite{auffinger2016extreme}, and random weighted networks \cite{xu2020optimal}, where their concentration properties are of critical importance.

Throughout this paper, $C, C', c$ denote absolute constants that may vary in
different equations, and $C_\alpha, c_\alpha, c_\alpha'$ are absolute constants depending only on $\alpha$.

\medskip

\noindent{\bf Notations:} Let $A=(a_{ij})$ be an $n\times n$ real matrix. The Frobenius norm of $A$ is defined as $\|A\|_F =\sqrt{\sum_{i,j} a_{ij}^2}$, and its operator norm is denoted by $\|A\|$. The maximum norm is given by $\|A\|_{\max} = \max_{i,j} |a_{ij}|$. We also denote $\|A\|_{1,\infty} = \max_i \sum_{j} |a_{ij}|$ and $\|A\|_{2,\infty} = \max_i (\sum_{j} |a_{ij}|^2)^{1/2}$. For a vector $a\in \mathbb R^n$, $\|a\|$ denotes its Euclidean norm and $\|a\|_{\infty}$ denotes the maximum absolute value of its entries.  
For a real random variable $X$, its $L_r$ ($r\ge 1$) norm is denoted by $\|X\|_{L_r} = (\E |X|^r)^{1/r}$. The $\alpha$-Orlicz norm\footnote{This is a quasi-norm for $\alpha \in (0,1).$} of a random variable $X$ is defined by \[\|X\|_{\Psi_\alpha}:= \inf \{ t>0 : \E(\exp(|X|/t)^\alpha) \le 2 \} <\infty.\] For two functions $f(n),g(n)>0$, we use the asymptotic notations $ g(n)\ll f(n)$ and $g(n)=o(f(n))$ if $f(n)/g(n)\to \infty$ as $n\to \infty$. Additionally, we use $f(n) \lesssim g(n)$ or $g(n) \gtrsim f(n)$ to denote that $f(n)\le C g(n)$ for some $C>0$. Let $[\log(x)]_+$ denote the function equal to $\log(x)$ if $x>1$ and $0$ otherwise. In the latter case, we interpret $\frac{1}{[\log(x)]_+} = \infty$.
For a set $S$, let $\mathbf{1}(S)$ denote the indicator function of $S$.

\medskip

Below we present our new results on the concentration of the quadratic form for sparse random vectors. We first state 
Bernstein-type inequalities for random quadratic forms and linear forms in Section~\ref{sec:bernstein}, which is easier to state and serves as an important immediate result towards a Bennett-type inequality for the quadratic form, which we will introduce in Section~\ref{sec:bennett}.

\subsection{Bernstein-type inequalities for sparse random vectors}\label{sec:bernstein}

To derive the concentration inequalities, we decompose the quadratic form $S_n$ into two parts: the diagonal part $S_{\mathrm{diag}} = \sum_{i=1}^n a_{ii} X_i^2$ and the off-diagonal part $S_{\mathrm{off}} = \sum_{i\neq j} a_{ij} X_i X_j$. We first establish concentration inequalities for $S_{\mathrm{off}}$, which presents the main technical challenge due to the dependencies among the terms in the sum. We then derive new concentration results for $S_{\mathrm{diag}}$, extending the existing theory for sums of independent random variables to the sparse setting.

\begin{theorem}
    \label{thm:sparse_alpha}
    Let $X = (X_1, X_2,\ldots, X_n)^T$ be a random vector with independent centered sparse $\alpha$-subexponential components, where each $X_i$ satisfies $\E(|X_i|^r) \le p_i (K r^{\frac{1}{\alpha}})^r$ for all $r\ge 1$ and some $\alpha \in (0,\infty]$. Let $A$ be a diagonal-free $n\times n$ matrix. Then for any $t>0$,
\begin{align}\label{eq:main2}
\Prob(|X^TAX|\geq t) \le e^2\exp\!\left(-c_{\alpha} \min\left\{\frac{t^2}{K^4 \sum_{i,j} a_{ij}^2 p_ip_j} , \frac{t}{K^2 \gamma_{1,\infty}},\left(\frac{t}{K^2 \|A\|_{\max}}\right)^{\min\{\frac{\alpha}{2}, \frac{1}{2}\}} \right\} \!\right),
\end{align}
where $c_\alpha$ is an absolute constant depending on $\alpha$ and
\begin{align}\label{eq:def_gamma_free}
\gamma_{1,\infty}:=\max_i\Big\{\sum_{j\neq i}|a_{ij}|p_j, \sum_{j\neq i}|a_{ji}|p_j\Big\}.
\end{align}
\end{theorem}

The tail bound in our theorem consists of three terms. The sub-gaussian tail is expected from the central limit theorem since $$\Var(X^T A X) = \sum_{i\neq j}(a_{ij}^2 + a_{ij} a_{ji} ) \E(X_i^2)\E(X_j^2) \le 2^{1+\frac{4}{\alpha}} K^4 \sum_{i\neq j} a_{ij}^2 p_i p_j.$$
The inequality above follows from $a_{ij}a_{ji} \le (a_{ij}^2 + a_{ji}^2)/2.$ 
The subexponential tail dominates when $A$ has a prominent row or column. For instance, consider a case where only the first row of $A$ is non-zero, with non-negative entries for simplicity, and $X_i = x_i-p$ for all $i$ where $x_i\sim\mathrm{Ber}(p)$.  Then $X^T A X = X_1 \cdot (\sum_{i=2}^n a_{1i} X_i)$ is the product of two independent random variables. If $p$ is chosen such that $p(\sum_{i=2}^n a_{1i})=p\|A\|_{1,\infty}=O(1)$, then  $\sum_{i=2}^n a_{1i} x_i$ is approximately Poisson with parameter $p\|A\|_{1,\infty}$. The tail behavior of $X^T A X$ is controlled by $\exp(-ct/p\|A\|_{1,\infty})$. 
For $\alpha\in(0,1]$, the $\alpha/2$-subexponential tail emerges due to individual terms in the summation $X^T A X = \sum_{i\neq j} a_{ij} X_i X_j$. Conditioning on $x_i$'s, each $a_{ij} X_i X_j$ is a product of two independent $\alpha$-subexponential random variables and has a tail of magnitude $\exp(-(t/K^2 |a_{ij}|)^{\alpha/2})$. For $\alpha> 1$, while the $1/2$-subexponential tail holds (as these random variables are subexponential), this tail bound is sub-optimal. Subsequently, we derive a sharper tail bound that improves this result (given in Theorem \ref{thm:simple}). 

For a general matrix $A$, to control the diagonal sum $S_{\mathrm{diag}}$, we introduce the following bound for the concentration of the sum of independent sparse $\alpha$-subexponential random variables:
\begin{theorem}\label{thm:linear}
Let $X = (X_1, X_2,\ldots, X_n)$ be a random vector with independent centered sparse $\alpha$-subexponential components, where each $X_i$ satisfies $\E(|X_i|^r) \le p_i (K r^{\frac{1}{\alpha}})^r$ for all $r\ge 1$ and some $\alpha \in (0,\infty]$. Let $a=(a_1,\ldots,a_n)$ be a deterministic vector in $\mathbb R^n$. Then for any $t>0$,
 \begin{align}\label{eq:thm_linear}
\Prob\left(\left|\sum_{i=1}^n a_i X_i \right| \ge t \right) \le e^2\exp\left( -c_{\alpha}' \min \left\{ \frac{t^2}{K^2 \sum_i a_{i}^2 p_i}, \left(\frac{t}{K\|a\|_{\infty}}  \right)^{\min\{\alpha,1\}}\right\}\right).
 \end{align}
\end{theorem}

Theorem~\ref{thm:linear} can be seen as a generalization of Bernstein's inequality for $\alpha$-subexponential random variables with a dependence on the variance of the linear form. Our proof technique relies on estimating all moments of the linear form $\langle a, X\rangle$. A discussion of related results to Theorem \ref{thm:linear} is given in Section \ref{sec:old}. 

\subsection{Bennett-type inequalities for sparse random vectors}\label{sec:bennett}

It is worth noting that for $\alpha>1$, the tail bounds in both Theorem \ref{thm:sparse_alpha} and Theorem \ref{thm:linear} are sub-optimal as they do not depend on $\alpha$. To address this limitation, we develop new tail bounds that improve upon both theorems in the large deviation regime for all $\alpha>1$. Below, we present a simplified version of these improved bounds, which have better tail probability estimates compared to Theorems~\ref{thm:sparse_alpha} and \ref{thm:linear}  in the large deviation regime.  {More precise statements of the improved linear form (Theorem~\ref{thm:linearbetter}) and quadratic form  (Theorem~\ref{thm:quadbetter})  are deferred to  Section~\ref{sec:improved}}.

\begin{theorem}[Simplified tail bounds: $\alpha>1$]\label{thm:simple} Let $X = (X_1, X_2,\ldots, X_n)^T$ be a random vector with independent centered sparse $\alpha$-subexponential components, where each $X_i$ satisfies $\E(|X_i|^r) \le p_i (K r^{\frac{1}{\alpha}})^r$ for all $r\ge 1$ and some $\alpha \in (1,\infty]$.
\begin{itemize}
    \item \text{(Linear form) } Let $a=(a_1,\ldots,a_n)$ be a deterministic vector in $\mathbb R^n$. Denote $\lambda^2={\sum_{i=1}^n a_i^2p_i}$. 
    For any $t \gtrsim K\lambda^2/\|a\|_{\infty}$, 
\begin{align}\label{eq:simple-tail-linear}
\Prob\left(\left|\sum_{i=1}^n a_iX_i \right|>t \right) \le e^2 \exp\left( - \frac{C_{\alpha}t}{ K \|a\|_{\infty}} \left[\log\Big( \frac{t\|a\|_{\infty}}{c_\alpha K \lambda^2}\Big) \right]^{1-\frac{1}{\alpha}} \right).
\end{align}

    \item \text{(Quadratic form) }Let $A$ be a diagonal-free $n\times n$ matrix. Denote $\sigma^2=\sum_{i\neq j} a_{ij}^2 p_i p_j$ and $\gamma_{1,\infty}$ as in \eqref{eq:def_gamma_free}. Set $\Lambda:=\max\{\gamma_{1,\infty},(\|A\|_{\max}\sigma^2)^{\frac{1}{3}}\}$.
    For any $t\gtrsim K^2 {\Lambda^2}/{\|A\|_{\max}}$, 
    \begin{align*}
    \Prob\left( \left|X^T A X \right|>t \right) \le e^2 \exp\left(-C_{\alpha}\sqrt{\frac{t}{K^2 \|A\|_{\max}}} \left[\log\left(\frac{t\|A\|_{\max}}{c_{\alpha}K^2 \Lambda^2} \right) \right]^{1-\frac{1}{\alpha}} \right).
    \end{align*}
\end{itemize}
\end{theorem}
Compared to Theorems~\ref{thm:sparse_alpha} and \ref{thm:linear},  when $\alpha>1$, for sufficiently large $t$, the tail probability gains an additional $\log^{1-1/\alpha}(t)$ factor, analogous to the improvement from Bernstein's inequality to Bennett's inequality.
Notably, for the sparse bounded random variables ($\alpha=\infty$), \eqref{eq:simple-tail-linear} matches Bennett's inequality in the large deviation regime (up to constants). A detailed comparison is presented in Section \ref{sec:old}. The tail behavior described by Bennett's inequality in the large deviation regime for bounded random variables is known to be sharp (up to universal constants); see \cite[Example 2.4]{major2005tail}. To the best of our knowledge, these Bennett-type inequalities are new for both the linear form when $\alpha\in (1,\infty)$ and the quadratic form when $\alpha\in (1,\infty]$ of sparse $\alpha$-subexponential random variables.

\medskip

The proof of  Hanson-Wright inequalities for sub-gaussian random vectors in \cite{RV13} relies on bounding the exponential generating function with a Gaussian comparison argument. Due to the sparsity parameters, a direct adaptation of the proof in \cite{RV13} will yield a dependence on the sub-gaussian norm of $X_i$, which does not have the optimal dependence on $p_i$. Moreover, this technique does not apply to $\alpha$-subexponential random variables. Instead, our proof is based on the moment method, which captures the correct variance term in the sub-gaussian tail part and the subexponential tail part also depends on the sparsity parameters. The moment method approach was used in \cite{HKM19,schudy2011bernstein,schudy2012concentration}. Our inequality improves the ones in \cite{HKM19} (corresponding to $\alpha=\infty$) to have Bennett-type inequalities for both linear and quadratic forms given in Theorems \ref{thm:linearbetter} and \ref{thm:quadbetter}. We also generalize and improve the inequalities from \cite{schudy2011bernstein} to sparse $\alpha$-subexponential random vectors for all the regime $\alpha \in(1,\infty]$. A detailed discussion and comparison with the existing literature is presented in Section \ref{sec:old}.

In Section~\ref{sec:application}, we present two applications of the new sparse Hanson-Wright inequalities. The first application establishes a new local law and eigenvector delocalization for Hermitian random matrices with independent sparse $\alpha$-subexponential entries. Notably, we derive the sharpest known upper bound for the infinity norm of eigenvectors in Wigner matrices with sparse $\alpha$-subexponential entries for any $\alpha\in (0,\infty]$ when the sparsity parameter $p\gtrsim \frac{\log^{2\vee (1+2/\alpha)} (n)}{n}$. 
For $\alpha\ge 2$, we extend the local law and eigenvector delocalization results to the super-critical regime $p\gtrsim \frac{\log (n)}{n}$, matching the bounds in \cite{HKM19} ($\alpha=\infty$) and generalizing them to any $\alpha\ge 2$ beyond the bounded support case. 
The second application addresses the concentration of the norm of linear transformations of sparse $\alpha$-subexponential random vectors. We anticipate that our new concentration inequalities will find further applications in the analysis of sketching algorithms \cite{derezinski2021sparse} and for the large deviation principle for weighted random graphs \cite{augeri2025large}.

\medskip 

\noindent{\bf Organization:} The rest of the paper is organized as follows.  Section \ref{sec:old} reviews existing literature on sparse Hanson-Wright inequalities and concentration of sums of sparse random variables, along with a discussion and comparison of our findings with previous work. We present two applications of our sparse Hanson-Wright inequality in Section~\ref{sec:application}. Generalization of our main results to general matrix and non-centered random vectors are given in Section~\ref{sec:generalization}.
The proofs of our main theoretical results are provided in Sections \ref{sec:proofmain} and \ref{sec:improved}, while Section \ref{sec:proof-apply} contains the proofs of the applications. 

Auxiliary results are presented in Appendices~\ref{sec:appendix}.  Appendices \ref{app:delo-proof} and \ref{app:local} provide additional proofs for the applications.  Finally, the proof of corollaries are presented in  Appendix~\ref{App:B}. 

\section{Previous results and comparison}\label{sec:old}

Recent progress in the Hanson-Wright-type inequalities and more generally, the polynomials of independent random variables, can be found in \cite{latala99, VW15, Adamczak15, Zhou19, HKM19, Zajkowski20, KZ20,Bellec19, Zhou20, LC21, CY21,gotze2021concentration, sambale2023some, augeri2025large,kim2000concentration, schudy2012concentration,dadush2018fast,buterus2023some,wang2024deformed} and references therein. For the following discussion in this section, let $A$ denote a diagonal-free, symmetric matrix.   We mainly focus on the results of concentration for quadratic and linear forms of $\alpha$-subexponential random vectors and sparse random vectors.

\paragraph{Quadratic forms for $\alpha$-subexponential random vectors}
Specifically, an extension from sub-gaussian to $\alpha$-subexponential random vectors was obtained in \cite[Corollary 1.4 (2)]{gotze2021concentration}:
	for any $\alpha\in(0,1]$, let $X_1,\dots, X_n$ be independent centered random variables with $\|X_i\|_{\Psi_{\alpha}}\leq K$. For any $t\geq 0$,
\begin{align}\label{eq:Gotze}
  \mathbb P&\left(|X^TAX-\mathbb EX^TAX|\geq t \right) \nonumber\\
  &\leq 2\exp \left(-{C_{\alpha}}\min \left \{ \frac{t^2}{K^4\|A\|_{F}^2}, \frac{t}{K^2 \|A\|},\left(\frac{t}{K^2\|A\|_{2,\infty}}\right)^{\frac{2\alpha}{2+\alpha}}, \left(\frac{t}{K^2 \|A\|_{\max}} \right)^{\alpha/2} \right\} \right).
\end{align}
The above tail bound encompasses up to four different regimes. It represents a refined version of the analog to the Hanson-Wright inequality (for large $t$), as presented in Proposition 1.1 of \cite{gotze2021concentration}:
\begin{align}\label{eq:gotze}
\mathbb P&\left(|X^TAX-\mathbb EX^TAX|\geq t \right) \leq 2\exp \left(-{C_{\alpha}}\min \big\{ {t^2}/{K^4\|A\|_{F}^2}, \left({t}/{K^2\|A\|}\right)^{\alpha/2} \big\} \right).
\end{align}
The extension of \eqref{eq:gotze} to any $\alpha \in (0,2]$ is discussed in Sambale's work \cite{sambale2023some}.  
Note that for the results mentioned above, their sub-gaussian tail depends on the $\Psi_{\alpha}$-norm of the $X_i$, which does not capture the sharp dependence on the sparsity parameter $p$ when $X_i$ is sparse. For related work on moment estimates, two-sided bounds for the symmetric log-convex case ($0<\alpha<1$) have been derived in \cite{KL15} for general multilinear forms, while the log-concave case ($\alpha\ge 1$) has been studied under various conditions in \cite{latala1996tail, latala99, AL12}. The proofs in \cite{gotze2021concentration,sambale2023some} proceed by converting a quadratic chaos
into a bilinear form in two independent copies, after which one typically invokes general
multi-level concentration inequalities expressed in terms of $\Psi_\alpha$-norms and matrix norms. Such tools are very robust and are well aligned with the regime $0<\alpha\le 1$. Indeed, in subsequent work \cite{DHWZ25} we develop a decoupling-based
approach that yields sharper sparse Hanson--Wright inequalities for $\alpha\in(0,1].$

For the present paper, our main improvement for $\alpha>1$ is of Bennett type and includes an
additional logarithmic gain. This refinement relies on retaining the explicit $p_i$-factors in
high-moment bounds and optimizing over the moment order; such a gain is typically not captured when sparsity is first summarized through global $\Psi_\alpha$-norm control. It would be interesting to understand whether a decoupling framework can be combined with comparably sharp moment bookkeeping to recover the same Bennett-type improvement for $\alpha>1$. We refer to \cite{gotze2021concentration} for a detailed discussion of the related results.

\paragraph{Quadratic forms for sparse random vectors}
Now we turn to the results that have a better dependence on $p$ compared to \eqref{eq:gotze} for sparse random vectors. Zhou \cite{Zhou19} initiated the study of sparse Hanson-Wright inequality that captures the dependence on $p_i$ for sparse sub-Gaussian random vectors $(\alpha=2)$, incorporating $p_i$'s in the sub-gaussian tail of $\eqref{eq:gotze}$. Notably, when $p_i$ are constant,  Zhou's sparse Hanson-Wright inequality recovers \eqref{hw}.
Some generalizations of Zhou's inequality can be found in \cite{PWL23,wu2023precise}. In a follow-up work, Zhou \cite{Zhou20} examined quadratic forms of non-identical independent Bernoulli variables, bounding the moment-generating function of the off-diagonal term.  Augeri \cite{augeri2025large} studied sparse bounded random variables with uniform sparsity $p$ and $\|A\| \le 1$, obtaining sharp tail bounds in the large deviation regime for $1 \gg p \gg \log n/n$ (Lemma 2.8 from \cite{augeri2025large}). He, Knowles, and Marcozzi \cite{HKM19} established an upper bound on the $L_r$ norms $\|X^T A X\|_{L_r}$ for $X_i$'s centered $\mathrm{Ber}(p)$, which proved crucial for analyzing sparse Erd\H os-R\'enyi random graphs in the supercritical regime. Their result has been generalized to sparse sub-gaussian random vectors in \cite[Lemma 7.4]{augeri2023large}.
 While our Theorems \ref{thm:linearbetter} and \ref{thm:quadbetter} draw inspiration from their work, they ultimately surpass the bounds established in \cite{HKM19}.

The study of the concentration of U-statistics also offers insights that are useful for deriving tail estimates of random quadratic forms. In Gin\'e, Lata{\l}a, and Zinn's study \cite{GLZ00}, by applying decoupling results and performing straightforward calculations, one can obtain the following tail estimate as a special case of their Theorem 3.3 in \cite{GLZ00}: consider $x=(x_1,\cdots,x_n)^T$ with independent components $x_i$ being centered Bernoulli random variables with parameter $p_i$. For any $t>0$, 
\begin{align}\label{eq:GLZ}
&\mathbb P \left( \left| x^TAx\right|\ge t \right) \le  2\exp \left (-c \min\left\{ \frac{t^2}{\sum_{i\neq j}a_{ij}^2 p_i p_j},\frac{t }{B}, \left( \frac{t}{D}\right)^{2/3}, \left( \frac{t}{\|A\|_{\max}}\right)^{1/2}\right\} \right),
\end{align} 
where $B = (\max_i p_i) \|A\|$ and $D=\max_j \sqrt{\sum_i a_{ij}^2 p_i}$.

Recently, Chakrabortty and Kuchibhotla \cite{CK18} studied a class of unbounded U-statistics. Their results (Theorem 1 in \cite{CK18}) yield moment estimates of the quadratic forms for sparse $\alpha$-subexponential random variables with each $X_i = x_i \xi_i$, where $x_i \sim \mathrm{Ber}(p)$ and $\xi_i$ are independent, centered $\alpha$-subexponential random variables. Notably, for $\alpha \in (0,1]$, the tail bound implied by these moment estimates is consistent with \eqref{eq:Gotze} for constant $p$, up to the $\log n$ factors. For $\alpha > 1$, the tail bound is equivalent to that of $\alpha = 1$, again up to $\log n$ factors. Recent work by \cite{Bak23} has extended the results in \cite{CK18}, generalizing their findings to a broader class of unbounded, heavy-tailed U-statistics.

 Schudy and Sviridenko \cite{schudy2011bernstein} derived concentration inequalities for polynomials of independent \emph{central moment bounded} random variables, improving upon their earlier work \cite{schudy2012concentration}. The central moment bounded random variables encompass bounded and continuous or discrete log-concave random variables. 
    When $X_1,\ldots,X_n$ are independent sparse exponential random variables, where $X_i = x_i \eta_i$ with independent $x_i \sim \mathrm{Ber}(p_i)$ and centered $\eta_i \sim \mathrm{Exp}(K^{-1})$,  their results yield: 
\begin{align}\label{eq:Bernoulli_polynomial}
\mathbb P \left( \left| X^T A X \right|\ge t \right)
\leq  2\exp \left (-C \min\left\{ \frac{t^2}{K^4 \sum_{i,j}a_{ij}^2p_ip_j},\frac{t }{K\gamma_{1,\infty}}, \left( \frac{t}{K^2\|A\|_{\max}}\right)^{1/2}\right\} \right).
\end{align} 
This result \eqref{eq:Bernoulli_polynomial} can be generalized to sparse $\alpha$-subexponential random variables for all $\alpha \in [1,\infty]$, using a contraction principle (see Theorem 1 in \cite{kwapien1987decoupling}) alongside the decoupling technique. 

Addressing the asymptotic behavior, \cite{bhattacharya2021asymptotic} established the limiting distributions for the random quadratic form $S_n=\sum_{i,j} a_{ij} x_i x_j$ with symmetric, $\{0,1\}$-valued, diagonal-free $A$ and i.i.d. Bernoulli $x_i$, focusing on the sparse regime where $\E(S_n) = O(1)$.

\paragraph{Comparison to Theorem~\ref{thm:sparse_alpha} and Theorem~\ref{thm:quadbetter}} 
We now compare our results with these key findings from the literature discussed above. For $\alpha \in [1,\infty]$, the bound \eqref{eq:Bernoulli_polynomial} aligns with our Theorem \ref{thm:sparse_alpha}, while our Theorem \ref{thm:quadbetter} provides an improvement over the term $\left( \frac{t}{K^2\|A\|_{\max}}\right)^{1/2}$ in the large deviation regime. The result in \eqref{eq:Gotze} represents the special case where all sparse parameters $p_i=1$. Notably, \eqref{eq:GLZ}, which represents the sparse bounded case, is consistent with \eqref{eq:Gotze} when $\alpha=1$, while incorporating the sparse parameters $p_i$. 

For $\alpha \in (0,1)$, our Theorem \ref{thm:sparse_alpha}, when all $p_i=1$, yields a weaker bound than \eqref{eq:Gotze}. This can be seen from the corresponding moment inequalities:
\begin{align}\label{eq:compareG}
r^{\frac{2+\alpha}{2\alpha}} K^2 \|A\|_{2,\infty}   \le r^{\frac{1}{\alpha}+ \frac{1}{2}} K^2 \|A\|_{1,\infty}\|A\|_{\max} \le r^{\frac{2}{\alpha}}K^2 \|A\|_{\max} + rK^2 \|A\|_{1,\infty}.
\end{align}
Using a similar argument, when $X_i$'s are centered Bernoulli random variables, \eqref{eq:GLZ} outperforms our Theorem \ref{thm:sparse_alpha} when $\alpha=\infty$. For $\alpha=2$, unlike \eqref{eq:gotze}, taking $p_i=1$ in either Theorem \ref{thm:sparse_alpha} or Theorem \ref{thm:quadbetter} does not recover the Hanson-Wright inequality for sub-gaussian random vectors in \cite{RV13}. However, in Theorem \ref{thm:quadbetter} (see also Theorem~\ref{thm:simple}), we improve the large deviation term to $\left(\frac{t}{\|A\|_{\max}}\right)^{1/2}\log^{1-\frac{1}{\alpha}}\left(\frac{t\|A\|_{\max}}{C_{\alpha}\Lambda^{2}}\right)$ for $t$ sufficiently large. These additional log factors are essential for our application for local law and eigenvector delocalization of sparse random matrices in Section~\ref{sec:local_law}.

A natural direction for future work is to derive a refined concentration inequality that unifies our results (Theorem \ref{thm:sparse_alpha} for $\alpha \in (0,1]$ and Theorem \ref{thm:quadbetter} for $\alpha\in (1,\infty]$) and the above results. Our techniques naturally extend to polynomials of independent random variables, as studied in \cite{gotze2021concentration, GLZ00, schudy2011bernstein, sambale2023some}; these extensions will appear in forthcoming work.

\paragraph{Linear forms of random vectors}
The concentration of the sum of independent random variables $\sum_{i=1}^n a_i X_i$ has been an area of intense research focus, with a rich body of literature. Our main contribution in this direction is the establishment of new Bennett-type inequalities in Theorem \ref{thm:linearbetter} (see also Theorem~\ref{thm:simple}) for the sum of sparse $\alpha$-subexponential random variables when $\alpha>1$. We will briefly review some key results.

For bounded random variables, the classical Bennett's inequality \cite{bennett1962probability} (see \cite[Theorem 2.9]{BLMbook}) provides the following bound: For $\max_i X_i \le K$,
 \begin{align}\label{eq:Bennett}
\Prob\left(\sum_{i=1}^n a_i (X_i-\E(X_i)) \ge t \right) \le \exp\left( - \frac{\nu^2}{K^2 \|a\|_{\infty}^2}h\left(\frac{K\|a\|_{\infty} t}{\nu^2} \right)\right),
 \end{align}
where $\nu^2 = \sum_{i=1}^n a_i^2\E(X_i^2)$ and $h(u)=(1+u)\log(1+u)-u$ for $u>0$. 
In particular, when $t\gtrsim \frac{\nu^2}{K\|a\|_{\infty}}$, the tail bound on the RHS of \eqref{eq:Bennett} is approximately $\exp\left( - \frac{t}{C K\|a\|_{\infty}} \log\left(\frac{K\|a\|_{\infty} t}{\nu^2} \right) \right).$ Compared to our Theorem \ref{thm:simple} where $\alpha=\infty$, since $\nu^2 \le K^2 \sum_i a_i^2 p_i = K^2 \lambda^2$, replacing $\nu^2$ with $K^2 \lambda^2$ in this tail bound yields exactly \eqref{eq:simple-tail-linear} up to constants. As shown in \cite[Example 2.4]{major2005tail}, the large deviation behavior captured by Bennett's inequality achieves optimal sharpness (up to universal constants) for bounded random variables.
Improved and refined versions of Bennett's inequality for the sum of independent bounded random variables have been developed, incorporating the Lambert W function; see, for instance, \cite{Zheng18}, \cite{Jebara18}, and \cite{Light20}.

For sums of independent symmetric $\alpha$-Weibull random variables, two-sided bounds have been established in the literature: Gluskin and Kwapie\'n \cite{gluskin1995tail} derived both moment and tail bounds for the sum of $\alpha\ge 1$, while Hitczenko, Montgomery-Smith, and Oleszkiewicz \cite{hitczenko1997moment} obtained moment inequalities for $\alpha< 1$. These results, combined with the symmetrization and contraction principles, enable one to obtain moment estimates for the $\alpha$-subexponential case. Although an exact analogue of our Theorem \ref{thm:linear} for sparse $\alpha$-subexponential random variables has not been found, a similar conclusion might be achievable through existing moment estimates combined with appropriate conditioning techniques. In this direction, we highlight two relevant results from \cite{KC22}. In \cite[Theorem 3.1]{KC22}, the authors leverage results from \cite{Latala97} to establish
 \begin{align*}
\Prob\left(\left|\sum_{i=1}^n a_i X_i \right| \ge t \right) \le 2\exp\left( -C_{\alpha} \min \left\{ \frac{t^2}{K^2 \|a\|_2^2}, \left(\frac{t}{K L_\alpha(a)} \right)^{\alpha} \right\}\right),
 \end{align*}
where $L_\alpha(a)=\|a\|_\infty$ if $\alpha \le 1$ and $L_\alpha(a)=\|a\|_{\frac{\alpha}{\alpha-1}}$ if $\alpha > 1$. The bounds above are sharp with respect to their dependence on the vector $a$. However, these results do not capture the dependence on the sparsity parameters $p_i$. The tail bounds that scale with variance are presented in \cite[Theorem 3.2]{KC22}, utilizing proof techniques from \cite{adamczak2008tail}: 
 \begin{align}
2\exp\left( - \min \left\{ C\frac{t^2}{\sum_i a_{i}^2 p_i\E\xi_i^2}, C_\alpha\left(\frac{t}{K\|a\|_{\infty}\log^{\frac{1}{\alpha}}(n+1)}  \right)^{\min\{\alpha,1\}}\right\}\right).
 \end{align}
The factor $(\log(n + 1))^{1/\alpha}$ is shown to be sharp by \cite[Section 2.2]{adamczak2008tail}, assuming only $\|X_i\|_{\psi_\alpha} < \infty$ and variance-based norm bounds. In contrast, our Theorem \ref{thm:linear} eliminates this logarithmic term in the subexponential tail by employing a variance proxy $K^2\sum_{i} a_i^2 p_i$ in the sub-gaussian tail.

\paragraph{Technical novelty and comparison of approaches.}
Although our proofs are based on the moment method, the improvement in the final tail bounds
is not merely a consequence of applying Markov's inequality to high moments.
The key novelty is a graph-structured moment expansion together with a refined counting
argument tailored to sparse inputs.
Concretely, when expanding $\E\big|\sum_{i\neq j} a_{ij}X_iX_j\big|^r$,
each monomial corresponds to a labeled multigraph whose vertices are the indices that appear,
and whose edge-multiplicities encode how often each factor $X_iX_j$ is used.
Our estimates exploit this structure by organizing the contribution according to graph features so that each time a new index is activated one incurs an explicit multiplicative penalty through the
sparsity parameters $(p_i)$.
This approach is in the spirit of the moment-method analyses in \cite{HKM19,schudy2011bernstein,schudy2012concentration},
but we leverage the graph structure more systematically and use a sharper enumeration of the relevant combinatorial configurations than \cite{schudy2011bernstein,schudy2012concentration}. These refinements are what yield the correct variance proxy $\sum_{i\neq j} a_{ij}^2 p_ip_j$ and the sparsity-adaptive parameters (e.g.\ $\gamma_{1,\infty}$ and $\|A\|_{\max}$) in the large-deviation regime.

In contrast to approaches that summarize the inputs through global Orlicz norms and then apply general multi-level polynomial concentration (as discussed above), our moment bounds retain the explicit $p_i$-weights throughout the optimization over the moment order.
This is essential for our Bennett-type improvements for $\alpha>1$, including the additional
logarithmic gain in Theorems~\ref{thm:linearbetter} and~\ref{thm:quadbetter}.

\medskip

We conclude the literature review with two open questions. It is unclear to us if a sparsity-adaptive Hanson-Wright inequality exists for Bernoulli random vectors which is optimal for all sparsity regimes. Taking $p_i$ to be a constant independent of $n$ in our inequalities does not recover the Hanson-Wright inequality for sub-gaussian random vectors in \cite{RV13}. For sparse sub-gaussian random vectors, it would be interesting to unify our Theorem~\ref{thm:quadbetter} and Zhou's inequality in \cite{Zhou19}. 
After completing this manuscript, we posted a related preprint \cite{DHWZ25} that further refines
sparse Hanson-Wright inequalities for quadratic forms for $\alpha\in(0,2]$. In particular, for
$0<\alpha\le 1$ it yields a sharpened bound that is optimal in several canonical examples, and it
also recovers (up to constants) the best-known non-sparse inequalities when $p_i\equiv 1$.
These results are complementary to the present paper: \cite{DHWZ25} focuses on $0<\alpha\le 1$, whereas here we emphasize Bennett-type improvements for $\alpha>1$ (for both linear and quadratic forms), including an additional logarithmic gain in the large-deviation regime obtained via our graph-based moment method.


\section{Applications}\label{sec:application}

\subsection{Sparse Wigner matrices: local law and eigenvector delocalization}\label{sec:local_law}

In this section, we apply concentration inequalities developed in this paper to study the spectral statistics of \emph{sparse Wigner matrices with $\alpha$-subexponential entries} for $\alpha>0$. Specifically, we consider a symmetric matrix $H \in \mathbb R^{n\times n}$ with sparse $\alpha$-subexponential entries defined as
\begin{align}\label{eq:H}
H=\frac{1}{\sqrt{np}} \widetilde{H} = \frac{1}{\sqrt{np}} (\widetilde{H}_{ij}),
\end{align}
where $\widetilde{H}_{ij}$ ($1\le i \le j \le n$) are independent centered random variables with the following assumptions:

\begin{assumption}\label{assumption:tildeH}
Each entry $\widetilde{H}_{ij}$ has the form $\widetilde{H}_{ij} = x_{ij} \xi_{ij}$, where $x_{ij}$ and $\xi_{ij}$ are independent, with $x_{ij} \sim \mathrm{Ber}(p)$ and $\xi_{ij}$ having mean 0, variance 1. Here $p=p(n)\in (0,1)$.
    We assume that for any $r\ge 1$, $\max_{i,j} \|\xi_{ij}\|_{L_r}\le K r^{\frac{1}{\alpha}}$, where $K\geq 1$ is a constant. When $\alpha=\infty$, this is interpreted as the $\xi_{ij}$'s being bounded by $K$. 
\end{assumption}

Denote the eigenvalues of $H$ as $\lambda_1(H),\cdots, \lambda_n(H)$ and their corresponding $\ell_2$-normalized eigenvectors as $u_1,\cdots,u_n$. When $np\to \infty$, it is known that the empirical spectral distribution (ESD) of $H$, defined as $\mu_n:=\frac{1}{n}\sum_{i=1}^n \delta_{\lambda_i(H)}$, converges weakly to the \emph{semicircle} law in probability \cite{wigner1993characteristic,TVW13}. When $\xi_{ij}=1$, $H$ is the adjacency matrix of Erd\H{o}s-R\'{e}nyi graph, and its spectral properties, including the local law and eigenvector delocalization \cite{TVW13,HKM19}, universality \cite{EKYY12,EKYY13,huang2015bulk}, and extreme eigenvalues \cite{BBK20,alt2021extremal}, have received considerable attention in recent years. When $\xi_{ij}$ is a bounded random variable, several extensions have been obtained in \cite{alt2021extremal,tikhomirov2021outliers,augeri2025large,DZ19}. When $p=\frac{c}{n}$, the large deviation principle of the spectral edge was studied in \cite{ganguly2024spectral,ganguly2022large}.

We establish the local law and complete eigenvector delocalization for $H$, analogous to He, Knowles, and Marcozzi's work \cite{HKM19} that studied the (normalized) adjacency matrix of the Erd\H{o}s-R\'enyi random graph $G(n,p)$ in the supercritical regime $np\ge C \log n$. As shown in \cite{alt2022completely}, this is the optimal sparsity level (up to a constant factor) for complete eigenvector delocalization. To the best of our knowledge, this is the first local law and complete delocalization result for sparse $\alpha$-subexponential random matrices down to the near-optimal sparsity $p\geq \frac{\mathrm{polylog}(n)}{n}$ and for unbounded sparse sub-gaussian random matrices down to the optimal sparsity $p\geq \frac{C\log n}{n}.$

Our methodology follows the framework of \cite{HKM19}, a method that was developed through a series of works \cite{EKYY13, EKYY12} (see references therein) and is comprehensively surveyed in \cite{BGK16}. The key innovation in our approach lies in the application of the large deviation inequalities developed in this paper to handle the sparse $\alpha$-subexponential setting.

The main tool to study the spectral statistics of $H$ is via its \emph{Green function} 
\[ G(z):= (H-z)^{-1}\]
for $z\in \mathbb{C}\setminus\{\lambda_1(H),\ldots,\lambda_n(H)\}$. Here we abbreviate $(H-zI_n)^{-1}$ as $(H-z)^{-1}$. Define the spectral domain  
\begin{align}\label{eq:S}
\mathbf{S}:=\{ E + \ii \eta \in \mathbb{C} : |E|\le 10, n^{-1} < \eta \le 1 \}.
\end{align}

For $\Im z\neq 0$, the Stieltjes transform $m(z)$ of the semicircle law is defined as
\[ m(z):= \int \frac{\rho_{sc}(dx)}{x-z}, \quad \rho_{sc}(dx):= \frac{1}{2\pi} \sqrt{(4-x^2)_+} dx.\]
Furthermore, $m(z)$ is characterized as the unique solution of the quadratic equation 
\begin{align}\label{eq:m}
m(z) + \frac{1}{m(z)} + z =0
\end{align}
such that $\Im m(z)>0$ for $\Im z >0$. We derive the following local law that estimates the matrix elements of $G(z)$ for $z\in \mathbf S$. Assume $n\ge 10$.
\begin{theorem}[Local law for sparse Wigner matrix with $\alpha$-subexponential entries]\label{thm:local} Consider the matrix $H$ in \eqref{eq:H} under Assumption~\ref{assumption:tildeH} and its Green function $G(z)$. For any $D>0$ and $\delta \in (0,1)$, there exists $L \equiv L(\delta, D)$ such that the following holds.  Let $  p\ge L C_\alpha^2 K^4 \frac{(\log n)^{\max\{2,\frac{2}{\alpha} \}}}{n}$. For $z=E+\ii \eta \in \mathbf S$ with $\eta \ge L C_\alpha^2 K^4 \frac{\log^2 n}{n}$, we have
\begin{align}\label{eq:entryG}
\Prob\left( \max_{i,j} \left| G_{ij}(z) - m(z) \delta_{ij} \right| \le \delta \right) \ge 1- n^{-D}.
\end{align}
Here, $C_\alpha\ge 1$ is an absolute constant depending only on $\alpha$.
\end{theorem}

A standard consequence of the local law is the semicircle law on small scales, as given below. Recall that $\mu_n$ is the ESD of $H$.
\begin{theorem}[Semicircle law on small scales]\label{thm:localsmall} Under the assumption of Theorem \ref{thm:local}, for any $D>0$ and interval $I$ satisfying $|I| \ge L C_\alpha^2 K^4 \frac{\log^2 n}{n}$, we have
\[ \Prob\left( |\mu_n(I) - \rho_{sc}(I)| \le \delta |I| \right) \ge 1- n^{-D}.\]
\end{theorem}

The proof of Theorem \ref{thm:localsmall} can be derived by either following the proof detailed in Section 8 from \cite{BGK16} or applying Lemma 64 from \cite{TV11} together with Theorem \ref{thm:local}. We omit the details. 

Another consequence of the local law is the complete delocalization of the (normalized) eigenvectors $\{u_i\}_{i=1}^n$ of $H$.  Note that compared to Theorem ~\ref{thm:local}, we require a slightly stronger lower bound on $p$ due to a technical condition on the spectral norm concentration of $H$ (see Lemma~\ref{lem:bdH}).

\begin{theorem}[Eigenvector delocalization for sparse Wigner matrix with $\alpha$-subexponential entries]\label{thm:delo} Consider the matrix $H$ in \eqref{eq:H} under Assumption~\ref{assumption:tildeH}. For any $D>0$, there exist $L=L(D)$ and $n_0 = n_0(D)$ such that for $p \ge L C_\alpha^2 K^4 \frac{(\log n)^{\max\{ 1+\frac{2}{\alpha},2 \}} }{n} $ and $n \ge n_0$, we have
\[\Prob\left(\forall i : \|u_i\|_{\infty} \le \sqrt{2L} C_{\alpha} K^2\frac{\log n}{\sqrt{n}} \right) \ge 1-2n^{-D}. \]
Here, $C_\alpha\ge 1$ is an absolute constant depending only on $\alpha$.
\end{theorem}

The study of the complete eigenvector delocalization has a rich history (see survey \cite{OVW16}). For the adjacency matrix of the Erd{\H o}s-R{\'e}nyi random graph $G(n,p)$, \cite{HKM19} showed that all eigenvectors satisfy $\|u_i\|_{\infty} \le \frac{n^{(\log n)^{-1/2}}}{\sqrt n}$ with high probability when $p \ge C\frac{\log n}{n}$, improving upon the bounds established in \cite{TVW13}. The conjectured optimal bound is $\|u_i\|_{\infty} \lesssim \frac{\sqrt{\log n}}{\sqrt{n}}$, which has been proven for bulk eigenvectors of Wigner matrices with bounded entries \cite{VW15}. For all eigenvectors, the optimal delocalization for Wigner matrices with subexponential entries was established in \cite{benigni2022optimal}. Our bound in Theorem \ref{thm:delo} appears to be $\sqrt{\log n}$ away from this optimal bound in the dense case, yet represents the current best result for Wigner matrices with sparse $\alpha$-subexponential entries. The proof of Theorem \ref{thm:delo} follows by combining two key ingredients: the local law and an estimate of the spectral norm of $H$. The complete details are presented in Appendix \ref{app:delo}.

The conclusion of Theorem \ref{thm:delo} can be extended to non-first eigenvectors of the adjacency matrix of the Erd{\H o}s-R\'enyi random graph $G(n,p)$ with minor modifications to the proof, following the approach of \cite{HKM19}. When $\alpha=\infty$, our result improves upon \cite{HKM19} by tightening the upper bound on the infinity norms from $\frac{n^{(\log n)^{-1/2}}}{\sqrt n}$ to $\frac{\log n}{\sqrt{n}}$. However, Theorem \ref{thm:delo} requires a more restrictive condition of $p\geq \frac{C\log^2 n}{n}$, compared to their super-critical regime of $p\geq \frac{C\log n}{n}$.

In the proof of Theorem \ref{thm:delo}, we have adapted Theorem \ref{thm:linear} and Theorem \ref{thm:sparse_alpha}. By incorporating the improved bounds from Theorem \ref{thm:linearbetter} and Theorem \ref{thm:quadbetter}, we can establish results analogous to \cite{HKM19} that hold up to the super-critical regime $p\gtrsim \frac{\log n}{n}$. Moreover, we demonstrate that these results extend to random matrices with sparse $\alpha$-subexponential entries, where $\alpha \in [2, \infty]$, in the super-critical regime. Note that $\beta = 1-\frac{1}{\alpha} \in [\frac{1}{2},1].$ And when $\alpha\geq 2$, $H$ is a sparse sub-Gaussian random matrix.

\begin{theorem}[Eigenvector delocalization for sparse sub-Gaussian random matrices]\label{thm:delo-new} Consider the matrix $H$ in \eqref{eq:H} with a parameter $\alpha\in [2,\infty]$. For any $D>0$ and  $\varepsilon \in (0,1)$, there exist $L=L(D,\varepsilon)$ and $n_0=n_0(D, \alpha,K,\varepsilon)$ such that for $p\ge LC_\alpha^2 K^4 \frac{\log n}{n}$ and $n\ge n_0$, we have 
\[\Prob\left(\forall i : \|u_i\|_{\infty} \le n^{-\frac{1}{2}+\varepsilon(\log n)^{-1+\frac{1}{2\beta}}} \right) \ge 1-2n^{-D}. \]
\end{theorem}
To the best of our knowledge, this is the first complete eigenvector delocalization result for sparse sub-gaussian random matrices in the supercritical regime $p\geq \frac{C\log n}{n}$ beyond the bounded random variable case \cite{TVW13,EKYY13,HKM19,DZ19}. A local law for $H$ when $\alpha\geq 2$ is also proved and we provide the detailed statement in Theorem~\ref{thm:local-new}.
Sparse Wigner matrices correspond to sparse random graphs with random weights, which has gained significant interest recently in the large deviation theory of random matrices  \cite{ganguly2022large,ganguly2024spectral,augeri2023large,augeri2025large}.
For $\alpha=\infty$, corresponding to sparse bounded entries in $H$, the upper bound $\|u_i\|_{\infty} \le n^{-\frac{1}{2}+(\log n)^{-\frac{1}{2}}}$ aligns with the result in \cite{HKM19}. The proof of Theorem \ref{thm:delo-new} follows the same approach as Theorem \ref{thm:delo}, but incorporates the improved concentration inequalities for $\alpha\ge 2$, namely Theorem \ref{thm:linearbetter} and Theorem \ref{thm:quadbetter}.
A sketch of this proof is provided in Appendix \ref{app:delo-new}.

\begin{remark}
The results presented in this section can be extended in several directions. First, the matrix $H$ can be complex-valued, requiring only cosmetic modifications to our proofs. Moreover, the diagonal entries can have a common variance that differs from the variance of the off-diagonal entries. A more substantial generalization involves sparse Wigner-type matrices with heterogeneous sparsity parameters $p_{ij}$, as studied in \cite{DZ19}. Such extensions can be achieved using the quadratic vector equation developed in \cite{AEK19}. While these generalizations are natural and tractable, we do not pursue their details in this work.
\end{remark}

\subsection{Concentration of the Euclidean norm of sparse random vectors}

Another application of our sparse Hanson-Wright inequality concerns the estimation of the Euclidean norm of random vectors, an application previously considered in \cite{RV13, gotze2021concentration} in the non-sparse setting. For a sparse $\alpha$-subexponential random vector $X$ and a deterministic matrix $B$, we establish the following concentration result for the squared norm $\|BX\|^2$ as a corollary of Theorem~\ref{thm:sparse_alpha}.

\begin{theorem}
    \label{thm:vec_norm}
    Consider $X=(X_1, \dots, X_n)^T$ a random vector with independent centered sparse $\alpha$-subexponential components, where each $X_i$ satisfies $\E(|X_i|^r) \le p_i (K r^{\frac{1}{\alpha}})^r$ for all $r\ge 1$ and some $\alpha \in (0,1]$. Let $B\in \R^{m\times n}$ be a deterministic matrix with columns $B_1,\dots, B_n$. For any $s>0$, with probability at least $1-C'\exp(-cs)$, the following holds:
    \begin{align}
        \left|\|BX\|^2 - p \|B\|_F^2 \right|\le &\sqrt{s}  K^2 \sqrt{p^2\|B^T B\|_F^2 + p\sum_{i}\|B_i\|^2}\\
        &+ s K^2 \max_{1\leq i\leq n}\Big\{ p\sum_{j\neq i}|\langle B_i,B_j \rangle|,\|B_i\|^2 \Big\} + s^{\frac{2}{\alpha}}  K^2 \|B^T\|_{2,\infty}^2. \label{eq:vec_norm}
    \end{align}
\end{theorem}
Note that in the above result, $\|B^T\|_{2,\infty}=\max_i \|B_i\|$ denotes the maximum Euclidean norm among all columns of $B$. Theorem \ref{thm:vec_norm} is a direct consequence of Theorem \ref{thm:sparse_alpha}. For $\alpha>1$, analogous results can be derived using Theorem \ref{thm:quadbetter}; we omit the details here. 

Let us compare our result with the existing bound in \cite{gotze2021concentration} for the special case $\alpha=1$, although both their result and ours apply to $\alpha\in(0,1]$. Consider a sparse subexponential random vector $X$ with i.i.d. entries $X_i=x_i\xi_i$ where all $x_i,\xi_i$ are independent, $x_i\sim\mathrm{Ber}(p)$ and $\xi_i$ is a mean-zero, unit-variance random variable with $\|\xi_i\|_{L_r}\leq Kr$ for all $r\ge 1$. These conditions imply $\|X_i\|_{\Psi_1}\leq CK$ for a universal constant $C>0$. According to \cite[Corollary 2.2]{gotze2021concentration}, with probability at least $1-C'\exp(-cs)$, 
\begin{align}
    \left|\|BX\|^2 - p \|B\|_F^2 \right|\le & \sqrt{s} K^2  \|B^TB\|_F + s K^2 \|B\|^2\\
     &+ s^{3/2} K^2\max_{i} \sqrt{\sum_{j} \langle B_i, B_j\rangle^2 } + s^2K^2   \|B^T\|_{2,\infty}^2. \label{eq:Gotze_norm}
\end{align}
Compared to \eqref{eq:Gotze_norm}, our bound \eqref{eq:vec_norm} does not have the $s^{3/2}$ term, while the $s^2$ terms in both bounds are identical. For small $p$, the $\sqrt{s}$ term in \eqref{eq:vec_norm} improves upon the corresponding term in \eqref{eq:Gotze_norm} since $\sqrt{\sum_{i}\|B_i\|^2} \leq \|B^TB\|_F$. When $p$ is sufficiently small, the $s$ term in \eqref{eq:vec_norm} also improves the corresponding term in \eqref{eq:Gotze_norm} due to the fact that $\max_{i} \|B_i\| \leq \|B\|$.

\section{Generalization to general matrices and non-centered random vectors}\label{sec:generalization}
Finally, we present two generalizations that combine the tail bounds for linear and quadratic forms. The first generalization is the Hanson-Wright inequality for a general matrix $A$. The second generalization is the Hanson-Wright inequality for sparse random variables with a diagonal-free matrix $A$, where we relax the assumption of centered random variables. For brevity, we only consider the special case of $\alpha=1$, though these results can be extended for arbitrary $\alpha>0$ using the  results we present in Section~\ref{sec:intro}. 

\begin{coro} \label{cor:main}
Let $X=(X_1,\cdots,X_n)^T$ be a random vector whose components are independent centered sparse subexponential random variables, i.e., $\E(|X_i|^r) \le p_i (K r)^r$ for all $r\ge 1$.  Let $A$ be an $n\times n$ matrix.  Then for any $t>0$,
\begin{align*}
 &\Prob\left( \left|X^T A X - \E(X^T A X) \right|\ge t\right) \\
 \le & e^2\exp\left(-C\min\left\{\frac{t^2}{K^4 (\sum_{i}a_{ii}^2p_i + \sum_{i\neq j}a_{ij}^2p_ip_j)} , \frac{t}{K^2 \gamma_{1,\infty}},\left(\frac{t}{K^2 \|A\|_{\max}}\right)^{\frac{1}{2}} \right\} \right),
\end{align*}
where 
\begin{align}\label{eq:def_gamma_A}
\gamma_{1,\infty}:=\max_i\left\{\sum_{j\neq i} |a_{ij}|p_j, \sum_{j\neq i} |a_{ji}|p_j, |a_{ii}|\right\}.
\end{align}
\end{coro}
Note that the definition of $\gamma_{1,\infty}$ in \eqref{eq:def_gamma_A} is consistent with the diagonal-free case in  \eqref{eq:def_gamma_free}. 
Next, we provide the Hanson-Wright inequality for non-centered random variables. 
\begin{coro}\label{coro:general-HW}
Let $Y=(Y_1,\ldots,Y_n)^T$ be a random vector with independent sparse subexponential entries, where each $Y_i$ satisfies $\E(|Y_i|^r) \le p_i (K r)^r$ for all $r\ge 1$. Let $A$ be an $n\times n$ diagonal-free matrix. Then there exists an absolute constant $C>0$ such that for any $t>0$,
\begin{align}
 &\Prob\left( \left|Y^T A Y - \E(Y^T A Y)\right| \ge t \right)\\
 & \le e^2\exp\left(-C\min\left\{\frac{t^2}{K^4 \sum_{i\neq j} a_{ij}^2 p_ip_j+K^2\sum_{i} b_i^2 p_i} , \frac{t}{K^2 \gamma_{1,\infty}},\left(\frac{t}{K^2 \|A\|_{\max}}\right)^{\frac{1}{2}} \right\} \right),
\end{align}
where $b_i :=\max\{|\sum_{j\neq i} a_{ij} \E(Y_j)|,|\sum_{j\neq i} a_{ji} \E(Y_j)|\}$ and $\gamma_{1,\infty}$ is defined in \eqref{eq:def_gamma_free}.
\end{coro}


\section{Proofs of Theorem \ref{thm:sparse_alpha} and Theorem~\ref{thm:linear}}\label{sec:proofmain}

The key input of the proof of Theorem \ref{thm:sparse_alpha} is the following moment estimates, whose proof represents the primary technical contribution of this work.

\begin{lemma}\label{lem:momentHW}Under the assumptions of Theorem \ref{thm:sparse_alpha},
    let $r$ be any even integer. We have 
\begin{align}\label{eq:momentbd}
\|X^T A X\|_{L_r} \le C_\alpha K^2\max\left\{r^{\max\{2,\frac{2}{\alpha}\}}\|A\|_{\max}, r\gamma_{1,\infty},  \sqrt{r}\sigma\right\},
\end{align}
where $\gamma_{1,\infty}$ is defined in \eqref{eq:def_gamma_A} and $\sigma:=\sqrt{\sum_{i\neq j} a_{ij}^2 p_ip_j}.$ Here, $C_\alpha=48 e^{5+2/e}\cdot 8^{\max\{1,1/\alpha\}}$.
\end{lemma}

To enhance the clarity of our presentation, we structure our proofs as follows. In Section \ref{sec:lemma}, we provide the proof of Lemma \ref{lem:momentHW} for the special case where $\alpha=1$. Subsequently, in Section \ref{sec:lemma-alpha}, we outline the necessary modifications to extend the conclusion to the general case of $\alpha \in (0,\infty]$. Finally, in Section \ref{sec:proofmainthm}, we demonstrate that Theorem \ref{thm:sparse_alpha} follows as a direct consequence of Lemma \ref{lem:momentHW}.


\subsection{Proof of Lemma \ref{lem:momentHW} when $\alpha=1$}\label{sec:lemma}

Without loss of generality, we can assume $a_{ij} \ge 0$, as the right-hand side in Lemma~\ref{lem:momentHW} remains the same if we replace every $a_{ij}$ with $|a_{ij}|$.  Moreover, note that \[X^T A X=X^T \frac{1}{2}(A+A^T) X :=  X^T \widetilde A X,\] where $\widetilde A =(\widetilde a_{ij})$ is symmetric and diagonal-free. It is straightforward to verify that $\sum_{i,j}\widetilde a_{ij}^2 p_i p_j \le \sum_{i,j} a_{ij}^2 p_i p_j $, $\max_i \sum_{j\neq i} |\widetilde a_{ij}| p_j \le \gamma_{1,\infty}$, and $\|\widetilde A\|_{\max} \le \|A\|_{\max}$. Therefore, it suffices to assume $A$ is symmetric in the subsequent proof.

\medskip

\noindent{\bf Step 1. (Rewrite the moment).} We start with the expansion
    \begin{align}\label{eq:quadratic}
        \mathbb E \left| \sum_{i,j=1}^n a_{ij} X_i X_j\right|^r=\sum_{i_1,\dots,i_{2r}} a_{i_1i_2}\cdots a_{i_{2r-1}i_{2r}} \E[X_{i_1}\cdots X_{i_{2r}}].
    \end{align}

 Since $X_i$   are centered random variables, for the expectations to be nonzero, every index $i_l$   in the summation on the right-hand side of \eqref{eq:quadratic} must appear at least twice. We introduce the following notation.

\begin{itemize}
    \item Let $\mathcal P(2r)$ be the set of partitions of $[2r]$, and let $\mathcal P_{\geq 2}(2r)$ be the set of partitions of $[2r]$ whose blocks are of size at least $2$. For each $\Pi\in\mathcal P_{\geq 2}(2r)$, let $s\in [n]^{|\Pi|}$ assign indices $[n]$ to the blocks in $\Pi$.  Denote $\mathcal P_{\geq 2}(2r,k) \subset \mathcal P_{\geq 2}(2r)$ to be the set of partitions of $[2r]$ with exactly $k$ blocks, each of size at least 2. 
    \item Encode each  $\Pi\in \mathcal P_{\geq 2}(2r,k)$ as a multigraph $G(\Pi)$ whose vertex set is $\Pi$. More precisely, consider $\mathcal G=([2r], E(\mathcal G))$ to be the graph with vertex set $[2r]$ and edge set 
    $$\{\{1,2\}, \{3,4\},\cdots,\{2r-1,2r\}\}.$$ 
    A multigraph $G(\Pi)$ is obtained by 
     defining $V(G(\Pi)) = \Pi$ and 
    \[E(G(\Pi)) = \{\{\tau_{2\ell-1},\tau_{2\ell}\}\mid \ell\in[r];\; i\in \tau_i\in \Pi \;\forall i\in [2r]\}.\]
    The multigraph $G(\Pi)$ has $r$ edges and $k$ vertices, and each vertex has a degree of at least 2. 
    For each vertex $\tau\in \Pi$, let $\deg_{G(\Pi)}(\tau)$ be the degree of  $\tau$ in $G(\Pi)$.
    \item We further denote $\mathcal P_{\geq 2}(2r,k,t)\subset \mathcal P_{\geq 2}(2r,k)$ to be the set of partitions of $[2r]$ with exactly $k$ blocks and the corresponding graph $G(\Pi)$ has $t$ connected components. 
\end{itemize}

Using the moment bounds
 \[\E|X_i|^l \leq p_i (K{l})^l \leq p_i (eK)^l l!,\]
we can bound \eqref{eq:quadratic} by 
\begin{align}\label{eq:quadratic2}
(eK)^{2r}\sum_{k=2}^{r} \sum_{t=1}^{k/2} \sum_{\Pi\in \mathcal P_{\geq 2}(2r,k,t)} &\Bigg(\prod_{\tau\in \Pi}\deg_{G(\Pi)}(\tau)! \nonumber\\
&\qquad \times \sum_{s\in [n]^{\Pi}}\left(\prod_{\{\pi_1,\pi_2\}\in E(G(\Pi))}a_{s_{\pi_1}s_{\pi_2} }\prod_{\pi \in \Pi}p_{s_\pi}\right)\Bigg).
\end{align}
Here in the sum, the variable $\Pi$ is a partition of $[2r]$, which indicates what group of $X_{i_1},\dots, X_{i_{2r}}$ are the copies of the same random variable. And $s:\Pi \to[n]$ indicates the index of $X_1,\dots, X_n$ that a certain group in $\Pi$ corresponds to. Note that $s$ should be an injection: if $i=s(\tau_1)=s(\tau_2)$ for some $\tau_1,\tau_2\in \Pi$, by definition $X_{i_{j_1}}$ and $X_{i_{j_2}}$ are the copy of the same random variable $X_i$, $\forall j_1\in \tau_1, j_2\in \tau_2$, and hence $\tau_1=\tau_2$ are the same group. We bound \eqref{eq:quadratic} with \eqref{eq:quadratic2} by changing the sum over $s$ from all possible injections $\Pi \to [n]$ to all maps $[n]^\Pi$.

Also, note that $A$ is diagonal-free. If there is some loop edge $\{\pi_1, \pi_2\}\in E(G(\Pi))$ with ${\pi_1}={\pi_2}$, there will be $a_{s_{\pi_1}s_{\pi_2}}=0$, annihilating the product and also the whole sum. Therefore, in the later discussion, we will focus on the subset of $\{\Pi\in\mathcal P_{\geq 2}(2r,k)\mid G(\Pi) \text{ is loop-free}\}$. Note that it also implies $2\leq k\leq r$ as it is impossible to be loop-free when $k=1$.

For simplicity, for a fixed multigraph $G$, we define
\begin{align}
\mathcal S_{G}:=\sum_{s\in [n]^{V(G)}}\left(\prod_{\{\pi_1,\pi_2\}\in E(G)} a_{s_{\pi_1}s_{\pi_2}}\prod_{\pi \in V(G)}p_{s_\pi}\right). 
\end{align}
Hence, as $\Pi = V(G(\Pi))$, \eqref{eq:quadratic2} becomes
\begin{align}\label{eq:quadratic3}
\mathbb E \left| \sum_{i\neq j} a_{ij} X_i X_j\right|^r &\le  (eK)^{2r}\sum_{k=2}^{r}  \sum_{t=1}^{k/2} \sum_{\Pi\in \mathcal P_{\geq 2}(2r,k,t) } \Big( \mathcal S_{G(\Pi)} \cdot\prod_{\tau\in \Pi}\deg_{G(\Pi)}(\tau)! \Big). 
\end{align}

\noindent{\bf Step 2. (Fix $\Pi$ and bound $\mathcal S_{G(\Pi)}$ uniformly).}
The next lemma provides estimates for $\mathcal S_{G(\Pi)}$. Recall that here $\gamma_{1,\infty}:=\max_i\left\{\sum_{j\neq i} |a_{ij}|p_j, \sum_{j\neq i} |a_{ji}|p_j\right\}$ for a diagonal-free matrix $A$.
\begin{lemma}\label{lem:main} For each fixed $\Pi\in \mathcal P_{\geq 2}(2r,k,t)$, with $\sigma^2 \coloneqq \sum_{i\neq j} a_{ij}^2p_ip_j$, there is
\begin{align}\label{eq:estS}
\mathcal S_{G(\Pi)} 
&\le \|A\|_{\max}^{r-k} \cdot \gamma_{1,\infty}^{k-2t} \cdot \sigma^{2t}. 
\end{align}
\end{lemma}
The proof of Lemma \ref{lem:main} is deferred to Section \ref{subsec:lem}. It is worth noting that the bound on $\mathcal S_{G(\Pi)} $ in Lemma \ref{lem:main} is uniform and independent of $G(\Pi)$.  To proceed from \eqref{eq:quadratic3}, we claim that the following bound holds: 
\begin{align}\label{eq:counting}
    \sum_{\Pi\in \mathcal P_{\geq 2}(2r,k,t)} \prod_{\tau\in \Pi}\deg_{G(\Pi)}(\tau)! \leq 2^{6r} e^{(3+2/e)r} r^{2r-k-t}.
\end{align}
We establish \eqref{eq:counting} in Step 3 below.

\noindent {\bf Step 3. (Consider all possible $\Pi$).}
For each $G(\Pi)$ with $t$ connected components, we can partition their $r$ edges and $k$ vertices into $t$ components and denote $k_i, r_i$ to be the number of vertices and edges in the $i$-th component correspondingly. More precisely, such a partition $\Pi\in \mathcal P_{\geq 2}(2r,k,t)$ can be attained in the following steps:

\begin{enumerate}
    \item First we consider all the partitions that partition $[2r]$ into $k$ \emph{labeled} subsets, i.e. each partition is indeed a surjective map $\phi:[2r]\to [k]$. Let $\Phi(\Pi)$ denote the set of such surjection $\phi$ that corresponds to partition $\Pi$. Therefore, as the parts in partition $\Pi\in \mathcal P_{\geq 2}(2r,k,t)$ are not ordered, $|\Phi(\Pi)| = k!$.
    Hence, 
    \begin{align}
        \sum_{\Pi\in \mathcal P_{\geq 2}(2r,k,t)} \prod_{\tau\in \Pi}\deg_{G(\Pi)}(\tau)! &\leq \frac{1}{k!}\sum_{\Pi\in \mathcal P_{\geq 2}(2r,k,t)} \sum_{\phi\in \Phi(\Pi)} \prod_{\tau=1}^k\deg_{G(\phi)}(\tau)!.
    \end{align}
    
    \item To compute the sum over $\phi$, first consider all possible positive integer series of $(r_1,\dots, r_t)$ and $(k_1,\dots, k_t)$, such that $\sum_{i=1}^t r_i=r$ and $\sum_{i=1}^t k_i=k$. There are $\binom{r-1}{t-1}\binom{k-1}{t-1}\leq 2^{r+k}$ such pairs of series in total.
    
    \item For fixed series $(r_1,\dots, r_t)$ and $(k_1,\dots, k_t)$, consider the connected components in $G$ are ordered in lexicographic order on the edge partition and the $i$-th component has $r_i$ edges and $k_i$ vertices. As the order is determined by the edge partition, the edge with the smallest index in the $i$-th component is determined after fixing the edges in the first $i-1$ components. Therefore, there are at most
    \[\binom{r-1}{r_1-1}\binom{r-r_1-1}{r_2-1}\cdots \binom{r-r_1-\dots-r_{t-1}-1}{r_t -1}\leq \
    \prod_{i=1}^t \binom{r}{r_i-1}\]
    many ways to partition $r$ edges $\{\{1,2\},\dots,\{2r-1,2r\}\}$ into $t$ connected components such that the $i$-th component has $r_i$ edges. Similarly, as we have already ordered the connected components, there are
    \[\binom{k}{k_1}\binom{k-k_1}{k_2}\cdots \binom{k-k_1-\dots-k_{t-1}}{k_t}\leq \
    \prod_{i=1}^t \binom{k}{k_i}\]
    many ways to partition the vertices. Therefore,
    \begin{align}
        &\frac{1}{k!}\sum_{\phi} \prod_{\tau=1}^k\deg_{G(\phi)}(\tau)! \\
        &\leq \frac{2^{r+k}}{k!}\max_{\text{component partition}}  \left(\prod_{i=1}^t \binom{r}{r_i-1}\binom{k}{k_i} \right) \sum_{\phi:\text{consistent with } r_i,k_i} \prod_{\tau=1}^k\deg_{G(\phi)}(\tau)! \\
        & \leq \frac{2^{r+k}}{k!}\max_{\text{component partition}} \left(\prod_{i=1}^t \binom{r}{r_i-1}\binom{k}{k_i} \right)  \left(\prod_{i=1}^t \sum_{\phi_i: [2r_i]\to [k_i]} \prod_{\tau=1}^{k_i}\deg_{G(\phi_i)}(\tau)!\right).
    \end{align}
    
    \item Finally, inside $i$-th connected component, there are $r_i$ edges and $k_i$ labeled vertices $\{j_1,\dots, j_{k_{i}}\}$. Let us further fix the degree $d_1,\dots,d_{k_{i}}$ of each vertex and then sum over all possible degree sequences. As the vertices are labeled and the degree of each vertex is fixed, there are at most $\binom{2r_i}{d_1,\dots,d_{k_{i}}}$ many different partitions $\phi_i$ of $[2r_i]$. Therefore, with $d_1+\dots+d_{k_i}=2r_i$, we have
    \begingroup
\allowdisplaybreaks
    \begin{align}
        \sum_{\phi_i: [2r_i]\to [k_i]} \prod_{\tau=1}^k\deg_{G(\phi_i)}(\tau)! &\leq \sum_{d_1,\dots, d_{k_{i}}} \binom{2r_i}{d_1,\dots,d_{k_{i}}} \cdot \prod_{\ell = 1}^{k_i} d_{\ell}! \\ 
        &\leq \binom{2r_i-1}{k_i - 1} (2r_i)!\\
        &\leq \binom{2r_i}{r_i} (2r_i)!\\
        &= \left(\frac{(2r_i)!}{r_i!}\right)^2\\
        &\leq (2r_i)^{2r_i}.
    \end{align}
    \endgroup
\end{enumerate}

Combining all the analysis of $\Pi$ above, with $r_i\geq 2$ we have
\begingroup
\allowdisplaybreaks
\begin{align}\label{eq:degree}
    \sum_{\Pi\in \mathcal P_{\geq 2}(2r,k,t)} \prod_{\tau\in \Pi}\deg_{G(\Pi)}(\tau)! 
    &\leq \frac{2^{r+k}}{k!}\max_{\text{component partition}} \left(\prod_{i=1}^t \binom{r}{r_i-1}\binom{k}{k_i} \right)  \left(\prod_{i=1}^t 2^{2r_i}r_i^{2r_i}\right) \nonumber\\
    &\leq \frac{2^{r+k}}{k!} \max_{\sum r_i=r,\sum k_i=k} \prod_{i=1}^t \left(\frac{er}{r_i-1}\right)^{r_i-1}\left(\frac{ek}{k_i}\right)^{k_i}2^{2r_i} r_i^{2r_i}\nonumber\\
    &\leq \frac{2^{5r}e^{2r}}{k!} \max_{\sum r_i=r,\sum k_i=k} \prod_{i=1}^t \left(\frac{r}{r_i}\right)^{r_i-1}\left(\frac{k}{k_i}\right)^{k_i}r_i^{2r_i}\nonumber\\
    &\leq \frac{2^{5r}e^{2r}r^{2r}}{k!} \max_{\sum r_i=r,\sum k_i=k} \prod_{i=1}^t \left(\frac{r}{r_i}\right)^{r_i-1}\left(\frac{k}{k_i}\right)^{k_i}\left(\frac{r_i}{r}\right)^{2r_i}\nonumber\\
    &\underset{(\text{a})}{\leq}\frac{2^{5r}e^{2r}r^{2r}}{k!} \max_{\sum r_i=r,\sum k_i=k} \prod_{i=1}^t \left(\frac{r_i}{r}\right)^{r_i-k_i+1}.
\end{align}
\endgroup
In the inequality (a), we applied Gibbs' Inequality to the $(k/k_i)^{k_i}$ terms, which is introduced in Lemma~\ref{lem:gibbs}.

To proceed from \eqref{eq:degree}, we again apply Gibbs' inequality. Denote $z_i\coloneqq r_i - k_i +1 \geq 1$ and $Z=\sum_{i=1}^t z_i = r-k+t$. We then have
\begin{align}
    \prod_{i=1}^t \left(\frac{r_i}{r}\right)^{z_i} \leq \prod_{i=1}^t \left(\frac{z_i}{Z}\right)^{z_i} = Z^{-Z}\prod_{i=1}^t z_i^{z_i}\leq Z^{-Z}(Z-t+1)^{Z-t+1}.
\end{align}
Here, the last inequality follows from Lemma~\ref{lem:seq_prod_upper_bound} and that the maximum of $\prod_i z_i^{z_i}$ is attained when $t-1$ of the $z_i$'s are equal to 1, and the remaining one equals $Z-(t-1)$. Therefore, we have
\begingroup
\allowdisplaybreaks
\begin{align}
    \sum_{\Pi\in \mathcal P_{\geq 2}(2r,k,t)} \prod_{\tau\in \Pi}\deg_{G(\Pi)}(\tau)! &\leq \frac{2^{5r}e^{2r}r^{2r}}{k!}\cdot \frac{(r-k+1)^{r-k+1}}{(r-k+t)^{r-k+t}}\\
    &\leq 2^{5r}e^{2r+k} r^{2r} k^{-k} {t^{-(t-1)}}\\
    &\leq 2^{5r}e^{3r} r^{2r-k-t+1} \left(\frac{r}{k}\right)^k \left(\frac{r}{t}\right)^t\\
    &\leq 2^{6r} e^{(3+2/e)r} r^{2r-k-t}.
\end{align}
\endgroup
Here in the last inequality, we use the fact that $f(x) = (r/x)^x$ attains its maximum at $x=r/e$.\\

\noindent {\bf Step 4. (Conclusion).} For the last part of the proof, combining the results from Lemma~\ref{lem:main}, rearranging terms, and applying Young's inequality, we obtain the following moment bound. Here $\sigma^2\coloneqq \sum_{i\neq j}a_{ij}^2p_ip_j$. Continuing from \eqref{eq:quadratic3}, we have
\begingroup
\allowdisplaybreaks
\begin{align}
\mathbb E \left| \sum_{i\neq j} a_{ij} X_i X_j\right|^r &\le  (eK)^{2r}\sum_{k=2}^{r} \sum_{t=1}^{k/2} \sum_{\Pi\in \mathcal P_{\geq 2}(2r,k,t) } \Big( \mathcal S_{G(\Pi)} \cdot\prod_{\tau\in \Pi}\deg_{G(\Pi)}(\tau)! \Big)\\
&\leq (eK)^{2r}\sum_{k=2}^{r} \sum_{t=1}^{k/2} \|A\|_{\max}^{r-k} \gamma_{1,\infty}^{k-2t} \sigma^{2t} \sum_{\Pi\in \mathcal P_{\geq 2}(2r,k,t)} \prod_{\tau\in \Pi}\deg_{G(\Pi)}(\tau)! \\
&\leq (64 e^{(5+2/e)r}K^2)^{r}\sum_{k=2}^{r} \sum_{t=1}^{k/2} \|A\|_{\max}^{r-k} \gamma_{1,\infty}^{k-2t} \sigma^{2t} \cdot r^{2r-k-t}\\
&\leq (64 e^{(5+2/e)r}K^2)^{r} \sum_{k=2}^{r}\sum_{t=1}^{k/2} (r^2 \|A\|_{\max})^{r-k}  (r\gamma_{1,\infty})^{k-2t} (\sqrt{r} \sigma)^{2t}\\
&\leq (64 e^{(5+2/e)r}K^2)^{r} \sum_{k=2}^{r}\sum_{t=1}^{k/2} 3^r\max\left\{\frac{(r^2 \|A\|_{\max})^{r}}{r/(r-k)}, \frac{(r\gamma_{1,\infty})^{r} }{r/(k-2t)}, \frac{(\sqrt{r} \sigma)^{r}}{r/2t}\right\} \\
&\leq (192 e^{(5+2/e)}K^2)^{r} \cdot \frac{r^2}{2} \max\left\{{(r^2 \|A\|_{\max})^{r}},{(r\gamma_{1,\infty})^{r} }, {(\sqrt{r} \sigma)^{r}}\right\}\\
&\leq (C_1K^2)^r \max\left\{{(r^2 \|A\|_{\max})^{r}},{(r\gamma_{1,\infty})^{r} }, {(\sqrt{r}  \sigma )^{r}}\right\},
\end{align}
where we can take absolute constant $C_1=384 e^{(5+2/e)}$. Here in the deduction we use the fact that $r^2\leq 2^r$ if $r\in\{2\}\cup [4,+\infty)$.
\endgroup


\subsubsection{Proof of Lemma \ref{lem:main}} \label{subsec:lem}
For each fixed $\Pi\in \mathcal P_{\geq 2}(2r,k,t)$ and its corresponding graph $G:=G(\Pi)$, we aim to bound 
\begin{align}
\mathcal S_{G}:=\sum_{s\in [n]^{V(G)}}\left(\prod_{\{\pi_1,\pi_2\}\in E(G)} a_{s_{\pi_1}s_{\pi_2}}\cdot \prod_{\pi \in V(G)} p_{s_\pi}\right)
\end{align}
using the matrix norms of $A$. Recall that the graph $G$ contains $r$ edges, and each vertex has a degree of at least 2. For diagonal-free $A$, we recall that
\[\gamma_{1,\infty}:=\max_i\left\{\sum_{j\neq i} |a_{ij}|p_j, \sum_{j\neq i} |a_{ji}|p_j\right\}, \quad \sigma^2\coloneqq \sum_{i\neq j} a_{ij}^2p_ip_j.\]

The main idea for the estimates is to delete the edges in $G = G(\Pi)$ as follows: 
\begin{enumerate}
    \item For each connected component $\mathcal C$ of $G$, 
    we consider any fixed spanning subgraph $H_{\mathcal C}$ in $\mathcal C$ that contains exactly one cycle (or equivalently, spanning subgraph with $|E(H_{\mathcal C})|=|V(H_{\mathcal C})|$; such structure can be obtained by first picking any fixed spanning tree of $\mathcal C$ and then adding one extra edge to it).
    We then delete all other edges in $\mathcal C$. Let $G'$ be the remaining subgraph (the union of all $H_{\mathcal C}$) after the deletion of all connected components, and hence $|E(G')|= |V(G')|=k$. 
    
    For each deleted edge, say $\{u,v\}$, we bound the corresponding $a_{uv} \le \|A\|_{\max}$, and we will delete $r-k$ many edges this way in total, which implies that we delete a total of $r-k$ edges from $G$. These deleted edges collectively contribute a factor of $\|A\|_{\max}^{r-k}$ to bound $\mathcal S_{G}$, i.e.,
    \[\mathcal S_{G}\leq \|A\|_{\max}^{r-k} \cdot\mathcal S_{G'}\]

    \item Now $G'$ is a graph consisting of $t$ connected components, each containing exactly one cycle and possibly some disjoint branches (or sub-trees) connected to the cycles. Let $G''\subseteq G'$ denote the smaller subgraph containing all $t$ cycles, denoted by $C_j$ for $j=1,\ldots,t$. From $G'$ to $G''$, we can repeatedly remove vertex $v$ and its edge $\{u,v\}$, where the degree of $v$ is 1. We fix $u$ first and bound $\sum_{v=1}^n a_{uv}p_v \le \gamma_{1,\infty}$ or $\sum_{v=1}^n a_{vu}p_v \le \gamma_{1,\infty}.$ Suppose $\ell$ edges are removed in this step, we have $|E(G'')| = k-\ell$, and more importantly,  
    \[\mathcal S_{G'}\leq \gamma_{1,\infty}^{\ell} \cdot\mathcal S_{G''}.\]
Moreover, $\mathcal S_{G''}=\prod_{j=1}^t \mathcal S_{C_j}.$

    \item Finally, for each remaining cycle $C_j$ denoted by $(v_1,v_2,\cdots, v_m)$ with $m$ edges, we have
    \begin{align}\label{eq:trbd}
    &\sum_{\substack{s\in [n]^{V(C_j)}}} a_{s(v_1)s(v_2)} a_{s(v_2) s(v_3)} \cdots a_{s(v_m) s(v_1)} \cdot p_{s(v_1)} p_{s(v_2)} \cdots p_{s(v_m)} \\
    & \quad = \sum_{\substack{i_1,\dots,i_m =1}}^n a_{i_1i_2} a_{i_2i_3} \cdots a_{i_m i_1} \cdot p_{i_1} p_{i_2} \cdots p_{i_m}\\
    & \quad= \tr(B^{m})\leq \|B\|_F^2 \|B\|^{m-2},
    \end{align}
    where $B:=\sqrt{P}A\sqrt{P}$ and  $P:=\mathrm{diag}(p_1,\dots, p_n)$. Next, we show that $\|B\| \le \gamma_{1,\infty}$. To see this, assume that $\|B\|$ is attained at some vector $x=(x_1,\dots,x_n)^T\in \mathbb S^{n-1}$, i.e., $\|B\|=\|Bx\|$. By the Cauchy-Schwarz inequality, we have 
    \begin{align}
        \|B\|^2 &= \|Bx\|^2=  \sum_{i}\left(\sum_{j}a_{ij}\sqrt{p_i p_j}x_j\right) ^2 \leq \sum_{i}\left(\sum_{j}a_{ij}{p_j}\right)\left(\sum_{j}a_{ij}{p_i}x_j^2\right)\\
        &\leq \gamma_{1,\infty} \sum_{i}\left(\sum_{j}a_{ij}{p_i}x_j^2\right) = \gamma_{1,\infty} \sum_{j}\left(\sum_{i}a_{ij}{p_i}x_j^2\right)\\
        &= \gamma_{1,\infty}^2 \sum_{j}x_j^2 = \gamma_{1,\infty}^2.
    \end{align}
    It follows that $\tr(B^{m})\leq \|B\|_F^2 \|B\|^{m-2}\leq \sigma^2 \gamma_{1,\infty}^{m-2}$, where $\|B\|_F^2 = \sum_{i,j}a_{ij}^2p_ip_j = \sigma^2$ is the variance of the quadratic form. Hence,
    \[\mathcal S_{G''} = \prod_{j=1}^t \mathcal S_{C_j} \leq \sigma^{2t} \gamma_{1,\infty}^{k-\ell - 2t}.\]
\end{enumerate}

In the process above, we have $r-k$ edges deleted in the first step and $\ell$ edges deleted in the second step. Then
\begin{align}
    \mathcal S_{G} &\leq \|A\|_{\max}^{r-k} \cdot\mathcal S_{G'}
    \leq \|A\|_{\max}^{r-k} \gamma_{1,\infty}^{\ell} \cdot\mathcal S_{G''}
    \leq \|A\|_{\max}^{r-k} \gamma_{1,\infty}^{k-2t} \sigma^{2t}.
\end{align}


\subsection{Proof of Lemma \ref{lem:momentHW} for any $\alpha>0$}\label{sec:lemma-alpha}
We now present the essential modifications to the case $\alpha=1$ required to extend the proof of Lemma \ref{lem:momentHW} to arbitrary $\alpha \in (0,\infty]$, highlighting the differences. 

First of all, the moments of sparse $\alpha$-subexponential random variables satisfy
\[\E |X_i|^l \leq p_i K^l l^{\frac{l}{\alpha}}.\]
For $\alpha \ge 1$, we simply bound $\E |X_i|^l\leq p_i K^l l^l$, which leads to the same moment bounds as in the case where $\alpha=1$:
\begin{align}\label{eq:alpha>1}
\|X^T A X\|_{L_r} \le C_1 K^2\max\left\{r^{2}\|A\|_{\max}, r\gamma_{1,\infty},  \sqrt{r}\sigma\right\}.
\end{align}

Next, we will focus on the case where $\alpha \in (0,1)$. From the moment bounds 
\[\E |X_i|^l \leq p_i K^l l^{\frac{l}{\alpha}} \leq p_i (eK)^l l! \cdot l^{(\frac{1}{\alpha}-1)l},\]
analogous to the derivation of \eqref{eq:quadratic2}, we now have
\begin{align}
\label{eq: moment_alpha_with_deg}
    &\E|X^T A X|^r \nonumber\\
    &\leq (eK)^{2r}\sum_{k=2}^{r} \sum_{t=1}^{k/2} \sum_{\Pi\in \mathcal P_{\geq 2}(2r,k,t)}
    \!\left(\prod_{\tau\in \Pi}\deg_{G(\Pi)}(\tau)! \deg_{G(\Pi)}(\tau)^{(\frac{1}{\alpha}-1)\deg_{G(\Pi)}(\tau)} \mathcal{S}_{G(\Pi)}\!\right),
\end{align}
where $\mathcal P_{\geq 2}(2r,k,t)$ and the definitions of $G(\Pi), \mathcal{S}_{G(\Pi)}$ follow those in \eqref{eq:quadratic2}. 

Note that compared to \eqref{eq:quadratic2}, we only have an extra multiplicative factor 
\[\left(\prod_{\tau\in \Pi}\deg_{G(\Pi)}(\tau)^{\deg_{G(\Pi)}(\tau)}\right)^{\frac{1}{\alpha}-1}:=\left(\mathcal D(\Pi)\right)^{\frac{1}{\alpha}-1}.\]
We claim that $\mathcal D(\Pi) \le 2^{2r} r^{2(r-k+1)}$ uniformly for all $\Pi\in \mathcal P_{\geq 2}(2r,k,t)$ due to Lemma~\ref{lem:seq_prod_upper_bound}. Consider any fixed graph $G(\Pi)$ with $r$ edges and $k$ vertices. Let its degree sequence be denoted by $d_1,\dots, d_k$. We have $d_i \geq 2$ and $d_1+\dots+d_k = 2r$. Therefore, by Lemma~\ref{lem:seq_prod_upper_bound},
\begin{align}
    \prod_{i=1}^k d_i^{d_i} \leq 2^2 \times \dots\times 2^2 \times (2r-2(k-1))^{2r-2(k-1)} = 2^{2r} r^{2(r-k+1)}.
\end{align}
Using the estimates on this additional factor $\left(\mathcal D(\Pi)\right)^{\frac{1}{\alpha}-1}$ and then following a similar estimation procedure as in the proof of Lemma~\ref{lem:momentHW}, we obtain
\begingroup
\allowdisplaybreaks
\begin{align}
    \E|X^T A X|^r &\leq (64 e^{5+2/e}K^2)^{r}\sum_{k=2}^r \sum_{t=1}^{k/2}\|A\|_{\max}^{r-k} \gamma_{1,\infty}^{k-2t}\sigma^{2t}  r^{2r-k-t} \cdot \big(2^{2r} r^{2(r-k+1)}\big)^{\frac{1}{\alpha}-1}\\
    &\leq (2^{2(\frac{1}{\alpha}-1)}\cdot 64 e^{(5+2/e)}K^2)^{r} \sum_{k=2}^r \sum_{t=1}^{k/2}\|A\|_{\max}^{r-k} \gamma_{1,\infty}^{k-2t}\sigma^{2t}  r^{2r-k-t} \cdot r^{2(\frac{1}{\alpha}-1)(r-k+1)}\\
    &\leq (2^{\frac{2}{\alpha}}\cdot 16 e^{5+2/e}K^2)^{r} r^{2(\frac{1}{\alpha}-1)}\sum_{k=2}^r \sum_{t=1}^{k/2}\|A\|_{\max}^{r-k} \gamma_{1,\infty}^{k-2t}\sigma^{2t}  \cdot r^{\frac{2}{\alpha}r-(\frac{2}{\alpha}-1)k-t} \\
    & = (2^{\frac{2}{\alpha}}\cdot 16 e^{5+2/e}K^2)^{r} r^{2(\frac{1}{\alpha}-1)}  \sum_{k=2}^r \sum_{t=1}^{k/2} (r^{\frac{2}{\alpha}}\|A\|_{\max})^{r-k} (r\gamma_{1,\infty})^{k-2t}(\sqrt{r}\sigma)^{2t} \\
    &\leq (2^{\frac{2}{\alpha}}\cdot 16 e^{5+2/e}K^2)^{r} r^{2(\frac{1}{\alpha}-1)} \cdot \frac{r^2}{2}\cdot 3^r\max\left\{r^{\frac{2}{\alpha}}\|A\|_{\max}, r\gamma_{1,\infty},  \sqrt{r}\sigma\right\}^r\\
    &\leq (2^{\frac{3}{\alpha}}\cdot 48 e^{5+2/e}K^2)^{r} \max\left\{r^{\frac{2}{\alpha}}\|A\|_{\max}, r\gamma_{1,\infty},  \sqrt{r}\sigma\right\}^r.
\end{align}
\endgroup
Here we again use Young's inequality for the products of three terms.

Hence, we have 
$$\|X^T A X\|_{L_r} \le C_\alpha K^2\max\left\{r^{\frac{2}{\alpha}}\|A\|_{\max}, r\gamma_{1,\infty},  \sqrt{r}\sigma\right\}$$ 
for any positive even integer $r$ by choosing $C_\alpha = 48 e^{5+2/e}\cdot 8^{1/\alpha}$ (in the case $\alpha\in(0,1]$). Combining this bound with $C_1=384 e^{(5+2/e)}$ in the case of \eqref{eq:alpha>1}, by taking $C_\alpha=48 e^{5+2/e}\cdot 8^{\max\{1,1/\alpha\}}$ in general completes the proof of Lemma \ref{lem:momentHW}. 

\subsection{Proof of Theorem~\ref{thm:sparse_alpha}}\label{sec:proofmainthm}

    Combining the moment bound from Lemma \ref{lem:momentHW} and Markov's inequality, we derive
    \begin{align}
        \Pr{\left| \sum_{i\not=j} a_{ij} X_i X_j\right|\geq t } 
        & \leq \frac{\E \left| \sum_{i\not=j} a_{ij} X_i X_j\right|^r }{t^r}\\
        &\leq \left(\frac{1}t C_\alpha K^2\max\left\{r^{\max\{2,\frac{2}{\alpha}\}}\|A\|_{\max}, r\gamma_{1,\infty},  \sqrt{r}\sigma\right\}\right)^r.
    \end{align}

    For any $t>0$, if 
    \begin{equation}
    \label{eq:<2}
        \min\left\{\frac{t^2}{(C_\alpha e)^2 K^4 \sigma^2 }, \frac{t}{C_\alpha e K^2\gamma_{1,\infty}},\left(\frac{t}{C_\alpha e K^2\|A\|_{\max} }\right)^{\min\{\frac{\alpha}{2}, \frac{1}{2}\}}\right\} < 2,
    \end{equation}
    then in the inequality in Theorem~\ref{thm:sparse_alpha}, there is
    \[\exp\left(-\min\left\{\frac{t^2}{(C_\alpha e)^2 K^4 \sigma^2}, \frac{t}{C_\alpha e K^2\gamma_{1,\infty}},\left(\frac{t}{C_\alpha e K^2\|A\|_{\max} }\right)^{\min\{\frac{\alpha}{2}, \frac{1}{2}\}}\right\}\right)\geq e^{-2}.\]
   The conclusion of Theorem~\ref{thm:sparse_alpha} holds since the right-hand side of the bound is at least 1. 
    
    Next we assume the \eqref{eq:<2} fails.
    By choosing $r$ to be the largest positive even integer with
    $$2\leq r\leq \min\left\{\frac{t^2}{(C_\alpha e)^2 K^4 \sigma^2}, \frac{t}{C_\alpha e K^2\gamma_{1,\infty}},\left(\frac{t}{C_\alpha e K^2\|A\|_{\max} }\right)^{\min\{\frac{\alpha}{2}, \frac{1}{2}\}}\right\},$$ we have that 
    \[\frac{1}t C_\alpha K^2\max\left\{r^{\max\{2,\frac{2}{\alpha}\}}\|A\|_{\max}, r\gamma_{1,\infty},  \sqrt{r}\sigma\right\} \leq \frac{1}{e},\]
    and further 
    \begin{align*}
        &\Pr{\left| \sum_{i\not=j} a_{ij} X_i X_j\right|\geq t } \\
        &\leq e^{-r}\\
        &\leq \exp\left(-\left(\min\left\{\frac{t^2}{ (C_\alpha e)^2K^4 \sigma^2 }, \frac{t}{C_\alpha eK^2\gamma_{1,\infty}},\left(\frac{t}{C_\alpha e K^2\|A\|_{\max} }\right)^{\min\{\frac{\alpha}{2}, \frac{1}{2}\}}\right\} - 2\right)\right)\\
        &= e^2 \exp\left(-c_\alpha\min\left\{\frac{t^2}{ K^4 \sigma^2 }, \frac{t}{K^2\gamma_{1,\infty}},\left(\frac{t}{K^2\|A\|_{\max} }\right)^{\min\{\frac{\alpha}{2}, \frac{1}{2}\}}\right\} \right),
    \end{align*}
by choosing $c_\alpha = \min\{{(C_\alpha e)^{-2}},{(C_\alpha e)^{-1}}, {(C_\alpha e)^{-\min\{\frac{\alpha}{2}, \frac{1}{2}\}}}\} = (C_\alpha e)^{-2}$.
The proof of Theorem \ref{thm:sparse_alpha} is complete.


\subsection{Proof of Theorem~\ref{thm:linear}}\label{sec:prooflinear}
Theorem~\ref{thm:linear} is established through the following moment estimates of the linear form, which we shall prove first.

\begin{lemma}
    \label{lem:linear_moment}
   Under the assumptions of Theorem~\ref{thm:linear}, let $r$ be an even integer. We have 
    \begin{align*}
         \left\|\sum_{i=1}^n a_i X_i\right\|_{L_r}
        &\leq C_\alpha' K \max\left\{r^{\max\{1, \frac{1}{\alpha} \}}\|a\|_\infty,\sqrt{r}\lambda\right\}.
    \end{align*}
    Here, $C_1'=8^{1+1/\alpha}e^2$.
\end{lemma}

\begin{proof}[Proof of Lemma \ref{lem:linear_moment}]
Note that for a single variable we have $\E |X_i|^l \leq p_i K^l l^{\frac{l}{\alpha}}$. Let $\mathcal P(r,k)$ denote the set of all the partitions of $[r]$ into $k$ subsets and $\mathcal P_{\geq2}(r,k)$ denote the subset of those partitions where all blocks have size at least 2. Recall that $\lambda \coloneqq \sqrt{\sum_{i=1}^n a_i^2p_i}$. We obtain the following estimates:
    \begingroup
    \allowdisplaybreaks
    \begin{align}
            \E \left|\sum_{i=1}^n a_iX_i\right|^r &\leq  \sum_{i_1,\dots, i_r} a_{i_1}\dots a_{i_r} \E [X_{i_1}\dots X_{i_r}]\\
        &\leq \sum_{k=1}^{r/2}\sum_{\Pi \in \mathcal P_{\geq 2}(r,k)} \sum_{s\in [n]^\Pi} \prod_{\tau \in \Pi} a_{s(\tau)}^{|\tau|} \E |X_{s(\tau)}|^{|\tau|} \\
        &\leq \sum_{k=1}^{r/2}\sum_{\Pi \in \mathcal P_{\geq 2}(r,k)} \prod_{\tau \in \Pi}  \left(\sum_{s=1}^n a_{s}^{|\tau|} \E |X_{s}|^{|\tau|}\right) \\  
        &\underset{(\text{a})}{\leq} \sum_{k=1}^{r/2}\sum_{\Pi \in \mathcal P_{\geq 2}(r,k)} \prod_{\tau \in \Pi} \left( K^{|\tau|} |\tau|^{\frac{|\tau|}{\alpha}} \sum_{s=1}^n \|a\|_{\infty}^{|\tau|-2} a_{s}^{2}p_s\right)\\
        & =  K^{r} \sum_{k=1}^{r/2} \|a\|_\infty^{r-2k} \lambda^{2k}\sum_{\Pi \in \mathcal P_{\geq 2}(r,k)} \left(\prod_{\tau \in \Pi} |\tau|^{|\tau|}\right)^{\frac{1}{\alpha}}\\
        &\underset{(\text{b})}{\leq} K^{r} \sum_{k=1}^{r/2} \|a\|_\infty^{r-2k} \lambda^{2k} \frac{1}{k!} \sum_{\Pi \in \mathcal P_{\geq 2}(r,k)}  \sum_{\phi \in \Phi(\Pi)} \left(\prod_{i=1}^k |\phi^{-1}(i)|^{|\phi^{-1}(i)|}\right)^{\frac{1}{\alpha}}\\
        &\underset{(\text{c})}{=} K^{r} \sum_{k=1}^{r/2} \|a\|_\infty^{r-2k} \lambda^{2k} \frac{1}{k!}\sum_{\begin{subarray}c
            \phi:[r]\to [k],\\
            |\phi^{-1}(i)|\geq 2
        \end{subarray}} \left(\prod_{i=1}^k |\phi^{-1}(i)|^{|\phi^{-1}(i)|}\right)^{\frac{1}{\alpha}}\\
        & = K^{r} \sum_{k=1}^{r/2} \|a\|_\infty^{r-2k} \lambda^{2k} \frac{1}{k!}\sum_{r_1+\dots+ r_k=r;r_i\geq 2} \binom{r}{r_1,\dots,r_k}\left(\prod_{i=1}^k r_i^{r_i}\right)^{\frac{1}{\alpha}}\\
        &\leq e^{2r}K^r\sum_{k=1}^{r/2} \|a\|_\infty^{r-2k} \lambda^{2k} \frac{r^r}{k^k}\sum_{r_1+\dots+ r_k=r;r_i\geq 2}\left(\prod_{i=1}^k r_i^{r_i}\right)^{\frac{1}{\alpha}-1}.
    \end{align}
    \endgroup
In the inequality (a) above, we use the the fact that $|\tau|\geq 2$ for $\tau\in \Pi$ so $a_s^{|\tau|-2}\leq \|a\|_{\infty}^{|\tau|-2}$. In the inequality (b), similar to the proof of Lemma~\ref{lem:main}, we use $\Phi(\Pi)$ to denote the set of surjections $\phi:[r]\to [k]$ corresponding to the partition $\Pi$, and further in (c) we rewrite the sum over $\Pi$ to surjections $\phi$. Depending on the value of $\alpha$, we have the following two cases:
\begin{itemize}
    \item When $\frac{1}{\alpha}\geq 1$, applying Lemma~\ref{lem:seq_prod_upper_bound}, 
    for any fixed series $r_1,\dots, r_k$, we have
    \[\prod_{i=1}^k r_i^{r_i} \leq 2^{2(k-1)}(r-2(k-1))^{r-2(k-1)}\leq 4^k r^{r-2k+2}.\]
    Therefore, we further obtain
    \begingroup
    \allowdisplaybreaks
    \begin{align}
        \E \left|\sum_{i=1}^n a_iX_i\right|^r 
        &\leq e^{2r}K^r \sum_{k=1}^{r/2} \|a\|_\infty^{r-2k} \lambda^{2k} \cdot\frac{r^r}{k^k}\sum_{r_1+\dots+ r_k=r;r_i\geq 2}\left(\prod_{i=1}^k r_i^{r_i}\right)^{\frac{1}{\alpha}-1}\\
        &\leq e^{2r}K^r \sum_{k=1}^{r/2} \|a\|_\infty^{r-2k} \lambda^{2k} \cdot\frac{r^r}{k^k}\binom{r-1}{k-1} \left(4^kr^{r-2k+2}\right)^{\frac{1}{\alpha}-1}\\
        &\leq (2e^2K)^r\sum_{k=1}^{r/2} \|a\|_\infty^{r-2k} \lambda^{2k} \cdot\frac{r^r}{k^k}\left(4^kr^{r-2k+2}\right)^{\frac{1}{\alpha}-1}\\
        &\underset{(\text{d})}\leq (2e^2 K)^{r} (4r^2)^{(\frac{1}{\alpha}-1)r} e^{\frac re} \sum_{k=1}^{r/2} \|a\|_\infty^{r-2k} \lambda^{2k} \cdot r^{\frac{1}{\alpha}r - (\frac{2}{\alpha}+1)k}\\
        &\leq (2e^{2+\frac{1}{e}}8^{\frac{1}{\alpha}-1} K)^{r} \sum_{k=1}^{r/2} (r^{\frac{1}{\alpha}}\|a\|_\infty)^{r-2k} (\sqrt{r} \lambda)^{2k}\\
        &\underset{(\text{e})}{\leq}  (2e^{2+\frac{1}{e}}8^{\frac{1}{\alpha}-1} K)^{r} \frac{r}{2}\cdot\max_{1\le k \le r/2} 2^r\max\left\{\frac{(r^{\frac{1}{\alpha}}\|a\|_\infty)^r}{r/(r-2k)},\frac{(\sqrt{r} \lambda)^r}{r/(2k)}\right\}\\
        &\leq (8^{1/\alpha}e^2 K)^{r} \max\left\{r^{\frac{1}{\alpha}}\|a\|_\infty,\sqrt{r} \lambda\right\}^r,
    \end{align}
    \endgroup 
    where the inequality (d) is by rearranging terms and apply $f(x)=(r/x)^x\leq e^{r/e}$ to the $k^k$ term and the inequality (e) is due to Young's inequality.

    \item When $\frac{1}{\alpha}<1$, simply using the lower bound $(\prod_{i=1}^k r_i^{r_i})^{\frac{1}{\alpha}-1} \leq 1$ and the Young's inequality, we have    
    \begin{align}\label{eq:momentlinear}
        \E \left|\sum_{i=1}^n a_i X_i\right|^r 
        &\leq (e^2 K)^{r} \sum_{k=1}^{r/2} \|a\|_\infty^{r-2k} \lambda^{2k} \cdot\frac{r^r}{k^k}\sum_{r_1+\dots+ r_k=r;r_i\geq 2}\left(\prod_{i=1}^k r_i^{r_i}\right)^{\frac{1}{\alpha}-1}\nonumber\\  
        &\leq (e^2 K)^{r}\sum_{k=1}^{r/2} \|a\|_\infty^{r-2k} \lambda^{2k} \cdot\frac{r^r}{k^k} \binom{r-1}{k-1}\nonumber\\ 
        &\leq (2e^2 K)^{r} e^{r/e}\sum_{k=1}^{r/2} \|a\|_\infty^{r-2k} \lambda^{2k} \cdot r^{r-k}\nonumber\\ 
        &\leq (8e^2 K)^{r} \max\left\{r\|a\|_\infty,\sqrt{r} \lambda\right\}^r
    \end{align}
\end{itemize} 
Therefore, combining the two cases, taking $C_1'=8^{\max\{1,1/\alpha\}}e^2$ completes the proof.
\end{proof}

We are now positioned to prove Theorem \ref{thm:linear}. The proof follows a similar approach to that of Theorem~\ref{thm:sparse_alpha} in Section \ref{sec:proofmainthm}. In essence, we define $r$ as the largest positive even integer satisfying 
$$2\leq r\leq \min\left\{ \left(\frac{t}{C_\alpha' e K \|a\|_\infty}\right)^{\min\{\alpha,1\}}, \frac{t^2}{(C_\alpha')^2 e^2 K^2 \lambda^2}\right\}.$$ By employing analogous arguments to those given in the proof of Theorem~\ref{thm:sparse_alpha}, we conclude the proof of Theorem~\ref{thm:linear} with $c_{\alpha}' = (C_\alpha' e)^{-2}$.

\section{Improved precise tail bounds for the linear and quadratic forms}\label{sec:improved}

Before presenting our improved, precise tail bounds for the linear and quadratic forms for $\alpha$-subexponential random variables with $\alpha\in (1,\infty]$, we introduce several essential notations. 

For real numbers $x$ and $y$, when $x\geq -e^{-1}$, the \emph{Lambert W function} $W(x)$ is defined as the solution to the equation $ye^y=x$. When $x\in (-e^{-1},0)$, the equation admits two real solutions, corresponding to two branches of the Lambert W function: the principal branch $W_0(x)$ and the lower branch $W_{-1}(x)$. On this interval, $W_0(x)$ is an increasing function satisfying $W_0(x)\geq -1$, while $W_{-1}(x)$ is a decreasing function satisfying $W_{-1}(x)\leq -1$ (see Figure \ref{fig:lambert}). We compile the basic properties of this special function in Appendix \ref{app:lambert}; for a detailed discussion of this function, we refer to \cite{corless1996lambert}. 

Throughout the context, we denote 
\begin{align}\label{eq:beta}
    \beta:=1-\frac{1}{\alpha}\quad \text{for } \alpha\in (1,\infty].
\end{align}

\begin{theorem}[Tail bound for the linear form: $\alpha>1$]\label{thm:linearbetter}Under the assumptions in Theorem \ref{thm:linear} for $\alpha\in(1,\infty]$ where ${\sum_{i=1}^n a_i^2p_i}\le \lambda^2$ and $\|a\|_{\infty} \le M$, we have for any $t>0$,
\begin{align}\label{eq:bettertb}
    \Prob\left(\left|\sum_{i=1}^n a_iX_i \right|>t \right)
    \le e^2 \exp\left(-C_\alpha\min\left\{ \frac{t^2}{ K^2 \lambda^2}, \frac{t}{KM}\max\left\{ 1,  H_{\alpha}(T) \right\} \right\} \right),
\end{align}
where we denote 
$$H_{\alpha}(T):= \left[-\beta W_{-1}\Big( -\frac{1}{\beta}T^{-\frac{1}{\beta}}\Big) \right]^{\beta}\mathbf{1}(T>(e/\beta)^{\beta})\quad \text{with}\quad 
T:=\frac{tM}{c_\alpha K\lambda^2}.
$$
\end{theorem}

\begin{remark}
    Compared to Theorem \ref{thm:linear}, the large deviation regime exhibits a new tail bound of the form \[ \exp\left(-C_\alpha\frac{ t}{K M}\cdot H_{\alpha}(T) \right). \] The function $H_\alpha(t)$ is increasing in $t$ and grows to infinity as $t\to\infty$. Using the asymptotic expansion of the Lambert W function $W_{-1}(x)\approx \log(-x)-\log(-\log(-x))$ as $x\to 0^{-}$ (see \cite{corless1996lambert}), we obtain that for large $t$,
    \[ H_{\alpha}(T)\approx \left[ \log(T) + \beta\log\left( \beta^{-1}\log(T)-\log(\beta^{-1}) \right)-\beta \log(\beta^{-1}) \right]^\beta. \]
    In the special case where $\alpha=\infty$ and $\beta=1$, this simplifies to
    \[H_{\infty}(T) \approx \log(T) + \log(\log(T)) = \log\left(\frac{tM}{C K\lambda^2} \right) +\log \left(\log\left(\frac{tM}{C K\lambda^2} \right) \right).\]
\end{remark}

Next, we present the improved tail bound for the quadratic forms. 
\begin{theorem}[Tail bound for the quadratic form: $\alpha>1$]\label{thm:quadbetter}
Under the assumptions in Theorem \ref{thm:sparse_alpha} for $\alpha\in(1,\infty]$ where $\sum_{i\neq j} a_{ij}^2 p_ip_j \le \sigma^2$, $\gamma_{1,\infty} \le \Gamma_{1,\infty}$, and $\|A\|_{\max} \le M$, we have for any $t>0$,
\begin{align}\label{eq:quad-newtail}
&\Prob\left(\left|\sum_{i\neq j} a_{ij} X_i X_j \right|>t \right)\\
&\le e^2\exp\left(-C_{\alpha}\min\left\{ \frac{t^2}{ K^4 \sigma^2}, \frac{t}{ K^2 \Gamma_{1,\infty}}, \sqrt{\frac{t}{ K^2 M}} \max\left\{ 1,\min_{i=1,2}H_{\alpha}(T_i) \right\}\right\} \right),
\end{align}
where we denote \[H_{\alpha}(T):= \left[-\beta W_{-1}\Big( -\frac{1}{\beta}T^{-\frac{1}{\beta}}\Big) \right]^{\beta}\mathbf{1}(T>(e/\beta)^{\beta})\]
with
\[T_1:=\frac{M}{\Gamma_{1,\infty}}\left(\frac{t}{c_\alpha K^2 M }\right)^{1/2} \quad\text{and}\quad T_2:=\left(\frac{M}{\sigma}\right)^{2/3}\left(\frac{t}{c_\alpha' K^2 M}\right)^{1/2}.\]
\end{theorem}

\begin{remark}
Due to the implicit nature of the Lambert W function, for practical applications of Theorems \ref{thm:linearbetter} and \ref{thm:quadbetter}, we utilize the bounds established in \cite{chatzigeorgiou2013bounds}:
\begin{align}\label{eq:Wbound}
-1-\sqrt{2u} -u < W_{-1}(-e^{-u-1}) < -1-\sqrt{2u} - \frac{2}{3}u
\end{align}
for $u>0$. One can readily verify that Theorem \ref{thm:simple} follows directly from applying these bounds \eqref{eq:Wbound} to Theorems \ref{thm:linearbetter} and \ref{thm:quadbetter}, observing that $W_{-1}\big( -\frac{1}{\beta}T^{-\frac{1}{\beta}}\big) = W_{-1}\big(-e^{-\log\big(\frac{\beta}{e} T^{\frac{1}{\beta}}\big)-1}\big)$.
\end{remark}

\subsection{Proof of Theorem \ref{thm:linearbetter}}
We assume $$\sqrt{\sum_{i=1}^n a_i^2p_i}\le \lambda \quad \text{and}\quad \|a\|_{\infty}\le M.$$ We begin by refining the moment estimates in Lemma \ref{lem:linear_moment} for $\alpha>1$. The existing moment bounds in Lemma \ref{lem:linear_moment} for $\alpha>1$ can be expressed as
\begin{align}\label{eq:oldlinearm}
\left\|\sum_{i=1}^n a_iX_i\right\|_{L_r}
        &\leq
        \begin{cases}
           C_\alpha K (\sqrt{r}\lambda), &\text{if } r\leq \left(\frac{\lambda}{M}\right)^2;\\
           C_\alpha K (r M), &\text{if } r> \left(\frac{\lambda}{M}\right)^2.
        \end{cases}
\end{align}
Our primary input is an improved moment estimate in the regime where $r> \left({\lambda}/{M}\right)^2$: 
\begin{lemma}\label{lem:moment-linear}
Under the assumptions in Theorem \ref{thm:linearbetter}, let $r$ be an even integer. We have
\begin{align}
\left\|\sum_{i=1}^n a_iX_i\right\|_{L_r}
        &\leq
        \begin{cases}
           C_\alpha K (\sqrt{r}\lambda), &\text{if } r\leq \left(\frac{\lambda}{M}\right)^2;\\
           C_\alpha K \frac{r M}{\left[\log r + 2\log(M/\lambda) \right]^{1-\frac{1}{\alpha}}}, &\text{if } r>\left(\frac{\lambda}{M}\right)^2.
        \end{cases}
\end{align}
\end{lemma}
In comparison to \eqref{eq:oldlinearm}, when $r> e \cdot \left(\frac{\lambda}{M}\right)^2$, the new upper bound in Lemma \ref{lem:moment-linear} is sharper than $C_\alpha K (r M)$ with an extra logarithmic term. When $\left(\frac{\lambda}{M}\right)^2 < r \le e \cdot \left(\frac{\lambda}{M}\right)^2$, the old upper bound $C_\alpha K (r M)$ in \eqref{eq:oldlinearm} is preferred. 

The following moment estimates follow immediately by combining  Lemma \ref{lem:moment-linear} with \eqref{eq:oldlinearm}. Define $$\Delta:=\left(\frac{\lambda}{M}\right)^2.$$

\begin{lemma}\label{coro:linearbetter}Under the assumptions in Theorem \ref{thm:linearbetter}, let $r$ be an even integer. We have
\begin{align}\label{eq:linearcom}
\left\|\sum_{i=1}^n a_iX_i\right\|_{L_r} \le C_\alpha K \max\left\{\sqrt{r}\lambda, rM\min\left\{ 1, \frac{1}{\left[\log (r/\Delta)  \right]_{+}^{1-\frac{1}{\alpha}}} \right\} \right\}.
\end{align}
\end{lemma}

We defer the proof of Lemma \ref{lem:moment-linear} to the end of this section. First, we show that Theorem \ref{thm:linearbetter} is a consequence of Lemma \ref{lem:moment-linear} and \ref{coro:linearbetter}.

\begin{proof}[Proof of Theorem \ref{thm:linearbetter}]
 Using Lemma \ref{coro:linearbetter} together with Markov's inequality, we  establish the tail probability estimates for any positive even integer $r$:
\begin{align}
\Prob\left(\left|\sum_{i=1}^n a_iX_i \right|>t \right) &\le \frac{\E|\sum_{i=1}^n a_iX_i |^r}{t^r} \\
&\le \left(\frac{ 1}{t} C_\alpha K \max\left\{\sqrt{r}\lambda, \min\left\{ r M, \frac{r M}{\left[\log (r/\Delta)  \right]^{1-\frac{1}{\alpha}}}\mathbf{1}(r>\Delta)\right\} \right\} \right)^r .
\end{align}
Similar to the proof of Theorem~\ref{thm:sparse_alpha} in Section \ref{sec:proofmainthm}, we choose $r$ to be the largest even integer such that 
\begin{align}\label{eq:rineq}
\frac{ 1}{t}  C_\alpha K \max\left\{\sqrt{r}\lambda, \min\left\{ rM, \frac{r M}{\left[\log (r/\Delta)  \right]^{1-\frac{1}{\alpha}}}\mathbf{1}(r>\Delta)\right\} \right\}  \le e^{-1}. 
\end{align}
This choice leads to
\[\Prob\left(\left|\sum_{i=1}^n a_iX_i \right|>t \right) \le e^{-r}.\]
Note that when $t$ is very small, the inequality \eqref{eq:rineq} cannot be satisfied since $r\geq 2$. However, this case is handled by having $e^2$ on the right-hand side of \eqref{eq:bettertb}, making \eqref{eq:bettertb} trivially true. Therefore, we may assume without loss of generality that $t$ is sufficiently large and $r$ satisfying \eqref{eq:rineq} exists. 

Next, we analyze \eqref{eq:rineq} to estimate $r$. First, we observe that \eqref{eq:rineq} is equivalent to the following two constraints:
\begin{align}\label{eq:1st-r}
    r\le \frac{t^2}{(eC_{\alpha}K \lambda)^2}
\end{align}
and 
\begin{align}\label{eq:2nd-r}
\min\left\{ rM, \frac{r M}{\left[\log (r/\Delta)  \right]^{1-\frac{1}{\alpha}}}\mathbf{1}(r>\Delta)\right\} \le \frac{t}{e C_\alpha K}.
\end{align}

For the latter inequality, we consider two cases depending on $\log(r/\Delta)\le 1$ or $\log(r/\Delta) >1$. 
\begin{itemize}
    \item If $\log(r/\Delta)\le 1$ or $r \le e \Delta$, \eqref{eq:2nd-r} reduces to $rM \le \frac{t}{e C_\alpha K}$ and hence 
    \begin{align}\label{eq:3rd-r}
    r \le \frac{t}{eC_{\alpha} K M}.
    \end{align}

    \item If $\log(r/\Delta)> 1$ or $r > e \Delta$, \eqref{eq:2nd-r} reduces to 
    \[ \frac{r M}{\left[\log (r/\Delta)  \right]^{1-\frac{1}{\alpha}}} \le \frac{t}{e C_\alpha K}.\]
    Recall that $\beta=1-\frac{1}{\alpha}>0$. Rewriting this inequality, we further obtain
    \[ (r/\Delta)^{\frac{1}{\beta}} \le \beta\left( \frac{t}{eC_{\alpha} K \Delta M }\right)^{\frac{1}{\beta}} \log\big((r/\Delta)^{\frac{1}{\beta}}\big). \]
    For simplicity, denote 
    $$L:= \beta\left( \frac{t}{eC_{\alpha} K \Delta M}\right)^{\frac{1}{\beta}}= \beta T^{\frac{1}{\beta}}$$ 
    and $u=\log[(r/\Delta)^{\frac{1}{\beta}}]>\frac1\beta>0$. The inequality becomes  $e^u \le L u.$
    By Proposition \ref{prop:lambert}, the feasibility of this inequality implies that $L= \beta T^{\frac{1}{\beta}}> e$ 
    and its solution is 
    \[ -u \ge W_{-1}(-L^{-1}),\]
    which leads to 
    \[r \le \Delta \left[ e^{-W_{-1}(-L^{-1}) } \right]^{\beta} = \Delta  e^{-\beta W_{-1}(-\beta^{-1} T^{-\frac{1}{\beta}}) }.  \]
    We simplify the expression $e^{-\beta W_{-1}(-\beta^{-1} T^{-\frac{1}{\beta}}) }= e^{-\beta w_0}$ by setting $w_0 =W_{-1}(-\beta^{-1} T^{-\frac{1}{\beta}})$. Note that $w_0<0$ and $w_0 e^{w_0} = -\beta^{-1} T^{-\frac{1}{\beta}}$. Hence,  
    \[ e^{-\beta w_0} = [-\beta W_{-1}(-\frac{1}{\beta}T^{-\frac{1}{\beta}})]^\beta \cdot T\]
    and we obtain \[ r \le \Delta T [-\beta W_{-1}(-\frac{1}{\beta}T^{-\frac{1}{\beta}})]^\beta\]
    as long as  $L= \beta T^{\frac{1}{\beta}}> e$.
\end{itemize}

Combining the two cases, we derive that \eqref{eq:2nd-r} is equivalent to
\begin{align*}
    r\le \max \left\{\frac{t}{eC_{\alpha} K M}, \Delta T \big[-\beta W_{-1}(-\frac{1}{\beta}T^{-\frac{1}{\beta}})\big]^\beta \mathbf{1}(\beta T^{\frac{1}{\beta}}> e) \right\}.
\end{align*}

In light of \eqref{eq:1st-r}, we conclude that 
\begin{align}\label{eq:rineq-2}
r\le \min\left\{\frac{t^2}{(eC_{\alpha}K \lambda)^2}, \max \left\{\frac{t}{eC_{\alpha} K M}, \Delta T \big[-\beta W_{-1}(-\frac{1}{\beta}T^{-\frac{1}{\beta}})\big]^\beta \mathbf{1}(\beta T^{\frac{1}{\beta}}> e) \right\} \right\}.
\end{align}
Note that $\Delta T = \frac{t}{eC_{\alpha} K M}$. Since we choose $r$ as the largest integer that satisfies the above inequality, \eqref{eq:bettertb} follows. 
\end{proof}

In the end of this section, we present the proof of Lemma \ref{lem:moment-linear}.
\begin{proof}[Proof of Lemma \ref{lem:moment-linear}]
We refine the estimate \eqref{eq:momentlinear} for $\alpha>1$ in the proof of Theorem \ref{thm:linear}. Starting from \eqref{eq:momentlinear}
\[\E \left|\sum_{i=1}^n a_i X_i\right|^r 
        \leq (C_\alpha K)^{r} \sum_{k=1}^{r/2} M^{r-2k} \lambda^{2k} \cdot\frac{r^r}{k^k}\sum_{r_1+\dots+ r_k=r;r_i\geq 2}\left(\prod_{i=1}^k r_i^{r_i}\right)^{\frac{1}{\alpha}-1},\]
we utilize the lower bound        
\[\prod_{i=1}^k r_i^{r_i} \geq \left(\frac{r}{k}\right)^r,\]
which is derived by applying Jensen's inequality to the function $x\log x$ after taking logarithms. 

It follows that 
\begin{align}
\E \left|\sum_{i=1}^n a_iX_i\right|^r &\le (C_\alpha K)^{r}\sum_{k=1}^{r/2} M^{r-2k} \lambda^{2k} \cdot{r^{r-k}} \binom{r-1}{k-1} \left( \frac{k^r}{r^r}\right)^{1-\frac{1}{\alpha}}\\
&\le (C_\alpha K)^r \max_{1\le k \le r/2} M^{r-2k} \lambda^{2k}\cdot 2^r r^{\frac{r}{\alpha}-k} (k^{r})^{1-\frac{1}{\alpha}},\\
&\le  (C_\alpha K)^r M^r(r^r)^{\frac{1}{\alpha}}\max_{1\le k \le r/2}\left( \frac{\lambda^2}{M^2}\right)^{k} k^{r(1-\frac{1}{\alpha})} r^{-k}\\
&= (C_\alpha K)^r M^r(r^r)^{\frac{1}{\alpha}}\max_{1\le k \le r/2} h_{\alpha}(k),
\end{align}
where recall that $\Delta:=\frac{\lambda^2}{M^2}$ and  
$$h_{\alpha}(k):= \Delta^k k^{r(1-\frac{1}{\alpha})} r^{-k}= k^{r(1-\frac{1}{\alpha})} \left(\frac{\Delta}{r}\right)^k.$$ 
Our next step is to optimize the function $h_\alpha$. We split the discussion into two cases.\\
\medskip

\begin{itemize}
    \item \noindent{Case 1.} When $r\le \Delta$ 
    the function $h_\alpha(k)$ is increasing. Therefore,  
    \[h_\alpha(k)\leq h_\alpha(r/2) = \frac{\lambda^{r}}{2^{r(1-\frac{1}{\alpha})} M^{r}}r^{r(\frac{1}{2}-\frac{1}{\alpha})}.\]
This yields
    \[\E \left|\sum_{i=1}^n a_iX_i\right|^r \leq (C_\alpha K)^r (\sqrt{r}\lambda)^r.\]
    \item \noindent{Case 2.}  When $r>\Delta$ 
    we optimize $h_\alpha(k)$ by setting $(\log h_\alpha(k))' =0$:
\[ \log \Delta + \left( 1-\frac{1}{\alpha}\right) \frac{r}{k} - \log r = 0.\] This yields the critical point
\[ k_* = \left( 1-\frac{1}{\alpha}\right)\frac{r}{\log (r/\Delta)}.\]
Since $\log h_{\alpha}(k)$ is concave, its maximum occurs at $k_*$. Using the bound
\[h_{\alpha}(k)\le h_\alpha(k_*) \le k_*^{r(1-\frac{1}{\alpha})},  \]
and substituting $k_*$, we obtain
    \begin{align}
        \E \left|\sum_{i=1}^n a_iX_i\right|^r &\leq (C_\alpha K)^r M^r(r^r)^{\frac{1}{\alpha}}k_*^{r(1-\frac{1}{\alpha})} \leq (C_{\alpha}K)^r \left(M \frac{r}{[\log(r/\Delta)]^{1-\frac{1}{\alpha}}}\right)^r.
    \end{align}
\end{itemize}

\end{proof}


\subsection{Proof of Theorem \ref{thm:quadbetter}}

Previously, for $\alpha \ge 1$, we treated the moment estimates of the quadratic form $\sum_{i\neq j} a_{ij} X_i X_j$ identically to the case $\alpha = 1$. As given in Lemma \ref{lem:momentHW}, we obtained 
\begin{align}\label{eq:oldquadm}
\left\| \sum_{i\neq j} a_{ij} X_i X_j \right\|_{L_r} \le C_1 K^2 \max\{ \sqrt{r} \sigma, r \Gamma_{1,\infty}, r^2 M\},
\end{align}
where 
\[ \sum_{i\neq j} a_{ij}^2 p_i p_j \le \sigma^2, \quad \gamma_{1,\infty}\le \Gamma_{1,\infty}, \quad \|A\|_{\max} \le M.\]
Analogous to the linear case, we seek to improve this bound in the large deviation regime. Specifically, we provide refined moment estimates when the maximum on the RHS of \eqref{eq:oldquadm} is achieved at $r^2 M$, that is, when
$$r> \frac{\Gamma_{1,\infty}}{M}:=\Delta_1 \quad \text{and}\quad r^{3/2} > \frac{\sigma}{M}:=\Delta_2.$$

\begin{lemma}\label{lem:momentquad}
    Under the assumptions of Theorem \ref{thm:quadbetter}, for any integer $r$ satisfying $r> {\Gamma_{1,\infty}}/{M}:=\Delta_1$ and $ r^{3/2} > {\sigma}/{M}:=\Delta_2$, we have
\begin{align}\label{eq:momentquad}
\left\| \sum_{i\neq j} a_{ij} X_i X_j \right\|_{L_r} \le C_{\alpha} K^2 \frac{r^2 M}{[\min\{\log(r/\Delta_1), \log(r^{3/2}/\Delta_2) \} ]^{2-\frac{2}{\alpha}}}.
\end{align}
\end{lemma}

Note that $\beta=1-\frac{1}{\alpha}$. Combining Lemma \ref{lem:momentquad} with \eqref{eq:oldquadm}, we have
\begin{lemma}\label{lem:momentquad-1}
    Under the assumptions of Theorem \ref{thm:quadbetter}, for any integer $r\ge 1$, we have
   \begin{align}\label{eq:quad-newmoment}
&\left\| \sum_{i\neq j} a_{ij} X_i X_j \right\|_{L_r} \nonumber\\
&\le C_{\alpha} K^2\max\left\{ \sqrt{r}\sigma, r \Gamma_{1,\infty},r^2 M\min\left(1,\frac{1}{[\min\{\log(r/\Delta_1)_+, \log(r^{3/2}/\Delta_2)_+ \} ]^{2\beta}}\right) \right\}.
\end{align} 
\end{lemma}

We first demonstrate that Theorem \ref{thm:quadbetter} follows as a consequence of Lemma \ref{lem:momentquad-1}. The complete proof of Lemma \ref{lem:momentquad} is presented at the end of this section. 

\begin{proof}[Proof of Theorem \ref{thm:quadbetter}] 
With \eqref{eq:quad-newmoment}, the tail bound is derived analogously to the linear case. Starting with Markov's inequality and applying \eqref{eq:quad-newmoment}, we obtain
\begin{align*}
&\Prob\left(\left|\sum_{i\neq j}^n a_{ij}X_iX_j \right|>t \right) \le \frac{\E|\sum_{i\neq j}^n a_{ij}X_iX_j |^r}{t^r} \\
&\le \frac{1}{t^r}\left(C_{\alpha} K^2\max\left\{ \sqrt{r}\sigma, r \Gamma_{1,\infty},r^2 M\min\left(1,\frac{1}{[\min\{\log(r/\Delta_1)_+, \log(r^{3/2}/\Delta_2)_+ \} ]^{2\beta}}\right) \right\} \right)^r.
\end{align*}
We choose $r$ to be the largest even integer satisfying
\[ \frac{1}{t} C_{\alpha} K^2\max\left\{ \sqrt{r}\sigma, r \Gamma_{1,\infty},r^2 M\min\left(1,\frac{1}{[\min\{\log(r/\Delta_1)_+, \log(r^{3/2}/\Delta_2)_+ \} ]^{2\beta}}\right) \right\} \le e^{-1}.\]
To get the desired tail bound, we solve this inequality to express $r$ in terms of $t$. Note that the above inequality is equivalent to
\begin{align}
&\frac{1}{t} C_{\alpha} K^2 \sqrt{r}\sigma \le e^{-1},\quad \frac{1}{t} C_{\alpha} K^2 r \Gamma_{1,\infty} \le e^{-1},\label{eq:quad-case12}
\end{align}
and
\begin{align}\label{eq:quad-case3}
\frac{1}{t} C_{\alpha} K^2 r^2 M\min\left(1,\frac{1}{[\min\{\log(r/\Delta_1)_+, \log(r^{3/2}/\Delta_2)_+ \} ]^{2\beta}}\right) \le e^{-1}.
\end{align}
The first two inequalities \eqref{eq:quad-case12} are equivalent to 
\begin{align}\label{eq:quad-tail12}
r\le \frac{t^2}{e^2 C_{\alpha}^2 K^4 \sigma^2} \quad\text{and} \quad r\le \frac{t}{eC_{\alpha} K^2 \Gamma_{1,\infty}}.
\end{align}
It remains to solve \eqref{eq:quad-case3}. We consider two cases. 

\medskip

\noindent{Case 1}. If $1 \le \frac{1}{[\min\{\log(r/\Delta_1)_+, \log(r^{3/2}/\Delta_2)_+ \} ]^{2\beta}}$, then \eqref{eq:quad-case3} is equivalent to $\frac{1}{t} C_{\alpha} K^2 r^2 M \le e^{-1}$ and thus
\begin{align}\label{eq:quad-tail31}
    r\le \sqrt{\frac{t}{eC_{\alpha} K^2 M}}.
\end{align}

\noindent{Case 2}. If $1 > \frac{1}{[\min\{\log(r/\Delta_1)_+, \log(r^{3/2}/\Delta_2)_+ \} ]^{2\beta}}$, then this implies that
\[r/\Delta_1> e  \quad\text{and}\quad r^{3/2}/\Delta_2>e.  \] Furthermore, \eqref{eq:quad-case3} is equivalent to
\[ \frac{r^2 C_{\alpha} K^2 M}{[\min\{\log(r/\Delta_1), \log(r^{3/2}/\Delta_2)\} ]^{2\beta}} \le t e^{-1}.\]
Rearranging the terms, we obtain
\[r \le [\min\{\log(r/\Delta_1), \log(r^{3/2}/\Delta_2)\} ]^{\beta}\sqrt{\frac{t}{eC_{\alpha} K^2 M}},\]
which is equivalent to two inequalities
\begin{align}\label{eq:ineq-r1}
    r\le \sqrt{\frac{t}{eC_{\alpha} K^2 M}} [\log(r/\Delta_1)]^{\beta}
\end{align}
and
\begin{align}\label{eq:ineq-r2}
    r\le \sqrt{\frac{t}{eC_{\alpha} K^2 M}} [\log(r^{3/2}/\Delta_2)]^{\beta}.
\end{align}

\begin{itemize}
    \item Let us solve \eqref{eq:ineq-r1} first. Rewriting \eqref{eq:ineq-r1} to get 
    \[ \left(r/\Delta_1\right)^{1/\beta} \le \beta \left(\frac{t}{eC_{\alpha} K^2 M\Delta_1^2} \right)^{\frac{1}{2\beta}} \log[(r/\Delta_1)^{1/\beta}].\]
    Denote $u=\log[\left(r/\Delta_1\right)^{1/\beta}]>\frac1\beta>0$, $T_1 = \sqrt{\frac{t}{eC_{\alpha} K^2 M\Delta_1^2}}$, and $L_1 = \beta T_1^{1/\beta}$. The inequality is now written as 
    \begin{align}\label{eq:ineq-L1}
     e^u \le L_1 u.
    \end{align} 
    Note that by the supposition $r/\Delta_1 >e$ and $\beta \in (0,1]$, $u>1$. By Proposition \ref{prop:lambert}, when $L_1 > e$, \eqref{eq:ineq-L1} is solvable and its solution is 
    \begin{align}\label{eq:quad-tail321}
    r\le \Delta_1 \left[ e^{-W_{-1}(-L_1^{-1})} \right]^{\beta} = \Delta_1 T_1 \left[ -\beta W_{-1}\Big(-\frac{1}{\beta} T_1^{-1/\beta}\Big)\right]^{\beta}.
    \end{align}
    For the identity on the right-hand side of \eqref{eq:quad-tail321}, we used that if $ye^y=x$, then $y=W(x)$ and $e^{-W(x)}=\frac{W(x)}{x}$.

    \item The second inequality \eqref{eq:ineq-r2} can be solved analogously. First, \eqref{eq:ineq-r2} can be written as 
    \[\left(r/\Delta_2^{2/3}\right)^{1/\beta} \le \frac{3\beta}{2} \left(\frac{t}{eC_{\alpha} K^2 M\Delta_2^{4/3}} \right)^{\frac{1}{2\beta}} \log[(r/\Delta_2^{2/3})^{1/\beta}]. \]
    Set $u=\log[(r/\Delta_2^{2/3})^{1/\beta}] > \frac{2}{3\beta}>0$, $T_2 = \left(\frac{3}{2}\right)^\beta\cdot \sqrt{\frac{t}{eC_{\alpha} K^2 M\Delta_2^{4/3}}} $, and $L_2= \beta T_2^{1/\beta}.$ 
    The inequality becomes 
    \[e^u \le L_2 u, \]
    which, by Proposition \ref{prop:lambert}, is solvable when $L_2 > e$ and we have 
    \begin{align}\label{eq:quad-tail322}
    r\leq \Delta_2^{2/3} \left[ e^{-W_{-1}(-L_2^{-1})} \right]^{\beta}= \Delta_2^{2/3} T_2 \left[ -\beta W_{-1}\Big(-\frac{1}{\beta} T_2^{-1/\beta}\Big)\right]^{\beta}.
    \end{align}
\end{itemize}

Combining \eqref{eq:quad-tail321} and \eqref{eq:quad-tail322}, we obtain 
\[ r\le \sqrt{\frac{t}{eC_{\alpha}K^2M}}\cdot\min_{i=1,2} \left[ -\beta W_{-1}\Big(-\frac{1}{\beta} T_i^{-1/\beta}\Big)\right]^{\beta}\]
when $\min_{i=1,2} \beta T_i^{1/\beta} > e$. 
Together with \eqref{eq:quad-tail31}, we see
\begin{align*}
r&\le \sqrt{\frac{t}{eC_{\alpha} K^2 M}}\max\left\{1, \min_{i=1,2} \left[ -\beta W_{-1}\Big(-\frac{1}{\beta} T_i^{-1/\beta}\Big)\right]^{\beta}\mathbf{1}\left(\min_{i=1,2} \beta T_i^{1/\beta} > e\right)  \right\}.
\end{align*}
The conclusion follows by combining the preceding inequality with \eqref{eq:quad-tail12} and our choice of $r$.
\end{proof}

We finish this section with the proof of Lemma \ref{lem:momentquad}.
\begin{proof}[Proof of Lemma \ref{lem:momentquad}]
We will prove Lemma~\ref{lem:momentquad} by a more detailed variation of the proof of Theorem~\ref{thm:sparse_alpha}. Recall that from \eqref{eq: moment_alpha_with_deg} we have
\[ \E\left| \sum_{i\neq j} a_{ij} X_i X_j \right|^r \leq (eK)^{2r}\sum_{k=1}^r\sum_{t=1}^{k/2}\sum_{\Pi\in \mathcal P_{\geq 2}(2r,k,t)} \prod_{\tau\in \Pi}\deg_{G(\Pi)}(\tau)! \cdot \big(\mathcal D(\Pi)\big)^{\frac{1}{\alpha}-1} \mathcal S(\Pi), \]
where $\mathcal S(\Pi)$ is bounded in Lemma~\ref{lem:main} and $\mathcal D(\Pi)$ is defined as 
\[\mathcal D(\Pi) \coloneqq \prod_{\tau\in \Pi}\deg_{G(\Pi)}(\tau)^{\deg_{G(\Pi)}(\tau)}.\]

Recall that we simply ignored the extra $\left(\mathcal D(\Pi)\right)^{\frac{1}{\alpha}-1}$ term when $\alpha \geq 1$. However, let $d_1,\dots,d_k\geq 2$ to be the degree sequence in the graph $G(\Pi)$ with $d_1+\dots+d_k=2r$. By Jensen's Inequality, we have 
\[\mathcal D(\Pi) = \prod_{i=1}^k d_i ^{d_i} \geq \prod_{i=1}^k\left(\frac{2r}{k}\right)^{\frac{2r}{k}} = \frac{4^rr^{2r}}{k^{2r}}.\]

Then we can combine the bound of $\mathcal{D}(\Pi), \mathcal{S}(\Pi)$ with estimates in Step 3 of the proof of Lemma \ref{lem:momentHW}, which proved
\[\sum_{\Pi\in \mathcal P_{\geq 2}(2r,k,t)} \prod_{\tau\in \Pi}\deg_{G(\Pi)}(\tau)! \leq C^r r^{2r}k^{-k}t^{-t}.\]
Therefore, 
\begin{align}
    \E\left| \sum_{i\neq j} a_{ij} X_i X_j \right|^r &\leq (C_\alpha K^2)^r \sum_{k=1}^r \sum_{t=1}^{k/2} 
    \|A\|_{\max}^{r-k} \gamma_{1,\infty}^{k-2t} \sigma^{2t} \cdot
    r^{2r}k^{-k}t^{-t} \big(r^{-2r}k^{2r}\big)^{1-\frac{1}{\alpha}}\\
    &\le (C_\alpha K^2)^r \sum_{k=1}^r \sum_{t=1}^{k/2} 
    \|A\|_{\max}^{r-k} \gamma_{1,\infty}^{k-2t} \sigma^{2t}  \cdot
    r^{\frac{2}{\alpha} r} k^{(2-\frac{2}{\alpha})r-k} t^{-t}\\
    &\le (C_\alpha K^2)^r \sum_{k=1}^r \sum_{t=1}^{k/2} 
    M^{r-k} \Gamma_{1,\infty}^{k-2t} \sigma^{2t}  \cdot
    r^{\frac{2}{\alpha} r} k^{(2-\frac{2}{\alpha})r-k} t^{-t}.
\end{align}
Recall that $\Delta_1=\frac{\Gamma_{1,\infty}}{M}$ and $\Delta_2 = \frac{\sigma}{M}$. Applying the bounds $(r/k)^k \le c^r$ and $(r/t)^t \le c^r$ as in the proof of Lemma \ref{lem:momentHW} yields
\begin{align}
    \E\left| \sum_{i\neq j} a_{ij} X_i X_j \right|^r &\le (C_\alpha K^2)^r M^r r^{\frac{2}{\alpha}r} \sum_{k=1}^r \sum_{t=1}^{k/2} \Delta_1^{k-2t} \Delta_2^{2t} r^{-k-t} k^{(2-\frac{2}{\alpha})r}\\
    &\le (C_\alpha K^2)^r M^r r^{\frac{2}{\alpha}r} \max_{\substack{1\le k \le r\\ 1\le t \le k/2} } \left( \frac{\Delta_1}{r} \right)^{k-2t} \left(\frac{\Delta_2}{r^{3/2}} \right)^{2t} k^{(2-\frac{2}{\alpha})r}.
\end{align}
Denote
\[ g_{\alpha}(k,t):= \left( \frac{\Delta_1}{r} \right)^{k-2t} \left(\frac{\Delta_2}{r^{3/2}} \right)^{2t} k^{(2-\frac{2}{\alpha})r} = \left(\frac{\Delta_2^2}{r \Delta_1^2}\right)^t \left(\frac{\Delta_1}{r} \right)^{k} k^{(2-\frac{2}{\alpha})r} .\]
We optimize $g_{\alpha}(k,t)$ for $1\le t \le k/2$ and $1\le k \le r$ by considering two cases. 
\medskip

\noindent{Case 1.} When $\frac{\Delta_2^2}{r \Delta_1^2 } \le 1$, we bound $$g_{\alpha}(k,t) \le \left(\frac{\Delta_1}{r} \right)^{k} k^{(2-\frac{2}{\alpha})r}.$$
Since $r/\Delta_1>1$, following the Case 2 in the proof of Lemma \ref{lem:moment-linear}, we further have
\[ g_{\alpha}(k,t)  \le k_*^{(2-\frac{2}{\alpha})r} \quad \text{with} \quad k_* =  \left(2-\frac{2}{\alpha} \right) \frac{r}{\log(r/\Delta_1)}\]
and consequently, 
\[ \E\left| \sum_{i\neq j} a_{ij} X_i X_j \right|^r \le (C_\alpha K^2)^r M^r \frac{r^{2r}}{[\log(r/\Delta_1)]^{(2-\frac{2}{\alpha})r}}. \]

\noindent{Case 2.} When $\frac{\Delta_2^2}{r \Delta_1^2 } > 1$, we have $g_{\alpha}(k,t) \le g_{\alpha}(k,k/2)$. It suffices to optimize $g_{\alpha}(k,k/2)$ or equivalently, $$\log g_{\alpha}(k,k/2) = k\log\left( \frac{\Delta_2}{r^{3/2}} \right) + \left(2-\frac{2}{\alpha}\right) r \log(k).$$ Since $r^{3/2}/\Delta_2>1$, using the same argument as in Case 2 of the proof of Lemma \ref{lem:moment-linear}, we optimize $\log g_{\alpha}(k,k/2)$ and obtain 
\[ g_{\alpha}(k,t)  \le k_*^{(2-\frac{2}{\alpha})r} \quad \text{with} \quad k_* =  \left(2-\frac{2}{\alpha} \right) \frac{r}{\log(r^{3/2}/ \Delta_2)}\]
It follows that
\[ \E\left| \sum_{i\neq j} a_{ij} X_i X_j \right|^r \le (C_\alpha K^2)^r M^r \frac{r^{2r}}{[\log(r^{3/2}/\Delta_2)]^{(2-\frac{2}{\alpha})r}}. \]

Combining both cases, we obtain 
\[\E\left| \sum_{i\neq j} a_{ij} X_i X_j \right|^r \le (C_{\alpha}K^2)^r M^r\frac{r^2}{[\min\{ \log(r/\Delta_1),\log(r^{3/2}/\Delta_2) \}]^{(2-\frac{2}{\alpha})r}}. \]

\end{proof}


\section{Proofs of applications}\label{sec:proof-apply}

\subsection{Proof of Theorem \ref{thm:local}}
In this section, we present the proof of Theorem \ref{thm:local}. While our proof strategy closely follows the framework developed in \cite{HKM19} (and references therein), we focus on highlighting the key differences. 

The new technical ingredients in our proof are the large deviation inequalities for sparse $\alpha$-subexponential random variables. We first note that the large deviation inequalities in Theorem \ref{thm:sparse_alpha} and Theorem \ref{thm:linear} can be extended to complex coefficients or complex matrices by treating the real and imaginary parts separately. For the convenience of our subsequent analysis, we summarize the necessary results below. 

Let $a=(a_1,\cdots,a_n)^T\in \mathbb{C}^n$ and $A = (a_{ij})_{1\le i,j \le n} \in \mathbb{C}^{n\times n}$, where $A$ is diagonal-free and $A=A^*$. Consider a random vector $X=(X_1,\cdots, X_n)^T \in \mathbb{R}^n$ with independent centered $\alpha$-subexponential random components. Each $X_i = x_i \xi_i $, where $x_i, \xi_i$ are independent, $x_i\sim \mathrm{Ber}(p)$ and $\xi_i$ has mean 0, variance 1, and satisfies $\|\xi_i\|_{L_r} \le Kr^{\frac{1}{\alpha}}$ for all $r\ge 1$. Then by Theorem \ref{thm:sparse_alpha} and Theorem \ref{thm:linear}, there is a constant $C_\alpha >0$ such that for any $r\ge 2$,
\begin{align}
& \left\| \sum_{i=1}^n a_i X_i \right\|_{L_r} \le C_\alpha K \max\left\{ r^{\max\{\frac{1}{\alpha},1\}} \|a\|_\infty, \sqrt{r} \sqrt{p} \|a\|_2 \right\},\label{eq:linear}\\
&\left\| \sum_{i=1}^n a_i (X_i^2- p) \right\|_{L_r} \le C_\alpha K^2 \max \left\{ r^{\max\{\frac{2}{\alpha},1\}} \|a\|_\infty, \sqrt{r} \sqrt{p} \|a\|_2 \right\},\label{eq:linearsq}\\
&\| X^T A X \|_{L_r} \le C_\alpha K^2 \max \left\{ r^{\max\{\frac{2}{\alpha}, 2\} } \|A\|_{\max}, r p \|A\|_{1,\infty} , \sqrt{r} p \|A\|_F\right\}\label{eq:quad}.
\end{align}
To derive \eqref{eq:linearsq}, we first observe that $\|X_i^2-\E X_i^2\|_{L_r} \le 2\|X_i^2\|_{L_r}\le (p_i 2^r K^{2r} r^{\frac{2}{\alpha}})^{1/r}$. Then we apply Theorem \ref{thm:linear} to the centered random variables $X_i^2 - \E X_i^2$ to obtain \eqref{eq:linearsq}.
Without loss of generality, we assume $C_\alpha \ge 1$. Recall that $\widetilde{H}$ is the normalized version of $H$ whose entries satisfy Assumption \ref{assumption:tildeH}.

In order to prove Theorem \ref{thm:local}, it suffices to prove the following result. Define a fundamental error parameter
\begin{align}\label{eq:zeta}
\zeta\equiv \zeta(n,r,p,\eta,\alpha) := \sqrt{12 C_\alpha K^2 \max \left\{ \frac{r^{\max\{\frac{1}{\alpha},1 \}}}{\sqrt{np}}, \frac{r}{\sqrt{n\eta}}\right\}} .
\end{align}
\begin{theorem}\label{thm:localmain}
Let $z\in \mathbf S$, $r\in \mathbb N$ and $p\in (0,1)$. Assume $n\ge 10$, $r\ge 2$ and $r^{\max\{\frac{1}{\alpha},1 \}} \le \sqrt{np}$. For any $t\ge 1$ such that $20 t\zeta \le 1$, we have
\[ \Prob \left( \max_{i,j} |G_{ij}(z) - m(z) \delta_{ij}| > 12 t\zeta \right)\le 3 n^6 t^{-2r}.\]
\end{theorem}
Let us assume Theorem \ref{thm:localmain} holds and prove Theorem \ref{thm:local}. Given $D>0$ and $\delta\in (0,1)$, we take $r=\log n$, $t=\exp(\frac{D+7}{2})$. Let $ \mathtt{f}(\delta):= \max\{1, \frac{(20\delta)^4}{12^4}\}.$ Then the conditions of Theorem \ref{thm:localmain} are satisfied provided that 
\[ p\ge \frac{12^6 e^{2(D+7)} \mathtt{f}(\delta)}{\delta^4} \cdot \frac{C_\alpha^2 K^4 (\log n)^{\max\{\frac{2}{\alpha},2 \}}}{n}\]
and
\[ \eta \ge \frac{12^6 e^{2(D+7)} \mathtt{f}(\delta)}{\delta^4} \cdot \frac{C_\alpha^2 K^4 (\log n)^2}{n}. \]
Under these conditions, we have $20 t\zeta \le 1$, $12 t\zeta \le \delta$ and $3n^6 t^{-2r} \le n^{-D}$. Setting 
\begin{align}\label{eq:L}
L(\delta, D):=\frac{12^6 e^{2(D+7)} \mathtt{f}(\delta)}{\delta^4},
\end{align}
we complete the proof of Theorem \ref{thm:local}.

\medskip

The remainder of this section is devoted to the proof of Theorem \ref{thm:localmain}. We begin by introducing necessary notations and preliminary results.
Following the notations in \cite{HKM19}, for any $k\in [n]$, we denote by $$H^{(k)}:= (H_{ij})_{i,j\in [n]\setminus \{k\} }$$ the $(n-1)\times (n-1)$ matrix obtained by deleting the $k$-th row and column from $H$. The Green function of $H^{(k)}$ is defined as $$G^{(k)}(z) = (H^{(k)}-z)^{-1}. $$ Moreover, we use the shorthand notation $\sum_{i}^{(k)} := \sum_{i: i\neq k}$. The Stieltjes transform of the ESD $\mu_n$ is given by $$s(z):=\frac{1}{n}\sum_{i=1}^n G_{ii}(z).$$
Define the random $z$-dependent error parameter 
$$\Gamma(z) := \max\left\{ \max_{i,j}|G_{ij}(z)|, \max_{i,j\neq k} |G_{ij}^{(k)}(z)|\right\}$$ and the $z$-dependent indicator function 
\begin{align}\label{def:phi}
\phi(z):= \mathbf{1}(\Gamma(z)\le 2).
\end{align}
For brevity, we omit the spectral parameter $z$ when the context is clear. Using Schur's complement (see Lemma 5.2 from \cite{BGK16}), one has 
\begin{align}
\frac{1}{G_{ii}} = -z - s + Y_i,
\end{align}
where 
\begin{align}\label{eq:Yi}
Y_i := H_{ii} + \frac{1}{n}\sum_k \frac{G_{ki} G_{ik}}{G_{ii}} - \sum_{k\neq l}^{(i)} H_{ik} G_{kl}^{(i)} H_{li} - \sum_{k}^{(i)} (H_{ik}^2 - n^{-1}) G_{kk}^{(i)}.
\end{align}
An important identity, known as the Ward identity, is frequently used in our estimates (see (4.4) from \cite{HKM19}):
\begin{align}\label{eq:ward}
\sum_{j} |G_{ij}|^2 = \frac{\Im G_{ii}}{\eta}.
\end{align}

The proof of Theorem \ref{thm:localmain} follows from a standard stochastic continuity argument (see \cite{HKM19, BGK16}) combined with a bootstrapping technique, proceeding from large scales ($\eta =1$) down to small scales ($\eta \approx \log n/n$). The new technical inputs are the following error estimates under the event $\phi=1$, which are proved using \eqref{eq:linear}, \eqref{eq:linearsq} and \eqref{eq:quad}.
\begin{lemma}[Main estimates]\label{lem:newerror}
Assume $r \ge 2$ and $r^{\max\{\frac{1}{\alpha},1 \}} \le \sqrt{np}$. Then we have
\begin{align}
&\max_{i\neq j} \|\phi G_{ij}\|_{L_r} \le 4 C_\alpha K \max\left\{ \frac{r^{\max\{\frac{1}{\alpha},1\}}}{\sqrt{np}}, \frac{\sqrt{r}}{\sqrt{n \eta}} \right\},\label{eq:lemkey1}\\
&\max_{i,j\neq k} \|\phi(G_{ij} - G_{ij}^{(k)}) \|_{L_r} \le 4 C_\alpha K \max\left\{ \frac{r^{\max\{\frac{1}{\alpha},1\}}}{\sqrt{np}}, \frac{\sqrt{r}}{\sqrt{n \eta}} \right\},\label{eq:lemkey2}\\
&\max_{i} \|\phi Y_i G_{ii} \|_{L_r} \le 12 C_\alpha K^2 \max\left\{ \frac{r^{\max\{\frac{1}{\alpha},1 \}}}{\sqrt{np}}, \frac{r}{\sqrt{n\eta}} \right\}.\label{eq:lemkey3}
\end{align}
\end{lemma}
\begin{proof}[Proof of Lemma \ref{lem:newerror}] We define $$\phi^{(i)} :=\mathbf{1}(\max_{k,l\neq i} |G_{kl}^{(i)}| \le 2)$$ and observe that $\phi \le \phi^{(i)}$. Due to the independence of $H^{(i)}$ and the $i$-th row and column of $H$, let $\|\cdot\|_{L_r | H^{(i)}}$ denote the conditional $L_r$ norm with respect to the conditional expectation $\E(\cdot | H^{(i)})$. 

We begin with the proof of \eqref{eq:lemkey1}. By Schur's complement formula (see also Lemma 4.2 from \cite{HKM19}), for $i\neq j$, we have
\[ G_{ij} = -G_{ii} \sum_{k}^{(i)} H_{ik} G_{kj}^{(i)}.\]
Using the nesting property of conditional $L_r$ norms, we obtain
\[\|\phi G_{ij}\|_{L_r}  \le 2 \left\| \phi^{(i)} \Big\| \sum_{k}^{(i)} H_{ik} G_{kj}^{(i)} \Big\|_{L_r | H^{(i)}} \right\|_{L_r}.\]
By \eqref{eq:linear}, we derive
\begin{align}
\phi^{(i)} \Big\| \sum_{k}^{(i)} H_{ik} G_{kj}^{(i)} \Big\|_{L_r | H^{(i)}} &= \frac{1}{\sqrt{np}} \phi^{(i)} \Big\| \sum_{k}^{(i)} \widetilde{H}_{ik} G_{kj}^{(i)} \Big\|_{L_r | H^{(i)}}\\
&\le \frac{C_\alpha K}{\sqrt{np}}\phi^{(i)} \max \left\{ r^{\max\{\frac{1}{\alpha},1\}} \max_{k\neq i}|G_{kj}^{(i)}|, \sqrt{rp} \sqrt{\sum_k^{(i)} |G_{kj}^{(i)}|^2}\right\}\\
&\le 2\frac{C_\alpha K}{\sqrt{np}} \max \left\{r^{\max\{\frac{1}{\alpha},1\}}, \frac{\sqrt{r p}}{\sqrt{\eta}} \right\},
\end{align}
where in the last inequality, we used \eqref{eq:ward} to get $\sum_k^{(i)} |G_{kj}^{(i)}|^2 = {\Im G_{jj}^{(i)}}/{\eta}.$ It follows that
\[\|\phi G_{ij}\|_{L_r}  \le 4 C_\alpha K \max\left\{ \frac{r^{\max\{\frac{1}{\alpha},1\}}}{\sqrt{np}}, \frac{\sqrt{r}}{\sqrt{n\eta}} \right\}\]
and hence \eqref{eq:lemkey1} is proved.

The proof of \eqref{eq:lemkey2} follows analogously by using the identity (see Lemma 4.2 in \cite{HKM19}) 
\[ \|\phi (G_{ij} - G_{ij}^{(k)})\|_{L_r} =\|\phi G_{ik} \sum_{l}^{(k)} H_{kl} G_{lj}^{(k)}\|_{L_r} \] and applying \eqref{eq:linear}. We omit the details.

To prove \eqref{eq:lemkey3}, we start from the decomposition \eqref{eq:Yi}, which gives
\begin{align}\label{eq:YiG}
Y_i G_{ii} =G_{ii} H_{ii}  + \frac{1}{n}\sum_k G_{ki} G_{ik} - G_{ii}\sum_{k\neq l}^{(i)} H_{ik} G_{kl}^{(i)} H_{li} - G_{ii}\sum_{k}^{(i)} (H_{ik}^2 - n^{-1}) G_{kk}^{(i)}.
\end{align}
We proceed by bounding each term in this expression separately. From the definition of $H_{ii}$, we first have
\[ \|\phi G_{ii} H_{ii}\|_{L_r} \le \frac{2}{\sqrt{np}}\|\widetilde{H}_{ii}\|_{L_r} \le \frac{2K r^{\frac{1}{\alpha}}}{\sqrt{np}}.  \]
Next, by Ward identity \eqref{eq:ward}, we obtain
\[ \phi\frac{1}{n}\left|\sum_k G_{ki} G_{ik}\right| \le \phi \frac{1}{n} \sum_k |G_{ik}|^2 = \phi \frac{\Im G_{ii}}{n\eta} \le \frac{2}{n\eta}. \]
For the fourth term on the right-hand side of \eqref{eq:YiG}, we use the nesting property of conditional $L_r$ norms to derive
\begin{align}
\Big\|\phi G_{ii}\sum_{k\neq l}^{(i)} H_{ik} G_{kl}^{(i)} H_{li}\Big\|_{L_r} &\le 2 \left\| \phi^{(i)}\Big\|\sum_{k\neq l}^{(i)} H_{ik} G_{kl}^{(i)} H_{li} \Big\|_{L_r | H^{(i)}}\right\|_{L_r}\\
&=\frac{2}{np} \left\| \phi^{(i)}\Big\|\sum_{k\neq l}^{(i)} \widetilde{H}_{ik} G_{kl}^{(i)} \widetilde{H}_{li} \Big\|_{L_r | H^{(i)}}\right\|_{L_r}.
\end{align}
Applying \eqref{eq:quad}, we further obtain
\begin{align}
&\phi^{(i)}\Big\|\sum_{k\neq l}^{(i)} \widetilde{H}_{ik} G_{kl}^{(i)} \widetilde{H}_{li} \Big\|_{L_r | H^{(i)}}\\
&\qquad\le \phi^{(i)} C_\alpha K^2 \max\left\{ r^{\max\{\frac{2}{\alpha}, 2\} } \|G_0^{(i)}\|_{\max}, r p \|G_0^{(i)}\|_{1,\infty} , \sqrt{r} p \|G_0^{(i)}\|_F \right\},
\end{align}
where the matrix $G_0^{(i)}$ is obtained from $G^{(i)}$ by setting all diagonal entries to zero. Observe that $\phi^{(i)} \|G_0^{(i)}\|_{\max} \le 2$ and by Cauchy-Schwarz inequality and \eqref{eq:ward}, we have 
\begin{align}
\phi^{(i)}\|G_0^{(i)}\|_{1,\infty} &=\phi^{(i)}\max_{k\neq i} \sum_j^{(i)} |G_{kj}^{(i)}| \\
&\le \phi^{(i)} \sqrt{n}\max_{k\neq i} \sqrt{\sum_j^{(i)} |G_{kj}^{(i)}|^2}=\phi^{(i)} \sqrt{n}\max_{k\neq i} \sqrt{\frac{\Im G_{kk}^{(i)}}{\eta}} \le 2\frac{\sqrt{n}}{\sqrt{\eta}}.
\end{align}
It follow from  \eqref{eq:ward} that
\begin{align}
\phi^{(i)}\|G_0^{(i)}\|_F = \phi^{(i)} \sqrt{\sum_{k\neq j}^{(i)} |G_{kj}^{(i)}|^2 } \le \phi^{(i)} \sqrt{\sum_k^{(i)} \frac{\Im G_{kk}^{(i)}}{\eta}} \le 2 \frac{\sqrt n}{\sqrt{\eta}}.
\end{align}
Hence, we arrive at
\begin{align}
\Big\|\phi G_{ii}\sum_{k\neq l}^{(i)} H_{ik} G_{kl}^{(i)} H_{li}\Big\|_{L_r} 
\le  4 C_\alpha K^2 \max\left\{ \frac{r^{\max\{\frac{2}{\alpha}, 2\} } }{np}, \frac{r}{\sqrt{n\eta}} \right\}.
\end{align}
For the third term on the right-hand side of \eqref{eq:YiG}, similar calculation using \eqref{eq:linearsq} yields 
\begin{align}
\Big\|\phi G_{ii}\sum_{k}^{(i)} (H_{ik}^2 - n^{-1}) G_{kk}^{(i)} \Big\|_{L_r} &\le \frac{2}{np} \left\| \Big\| \phi^{(i)} \sum_{k}^{(i)} (\widetilde{H}_{ik}^2 - p ) G_{kk}^{(i)} \Big\|_{L_r | H^{(i)}} \right\|_{L_r}\\
&\le C_\alpha K^2 \frac{4}{np}\max \left\{ r^{\max\{\frac{2}{\alpha},1\}}, \sqrt{r}  \sqrt{np} \right\}\\
&= 4 C_\alpha K^2 \max\left\{\frac{r^{\max\{\frac{2}{\alpha},1\}}}{np}, \frac{\sqrt{r}}{\sqrt{np}} \right\}.
\end{align}
Combining the above estimates, we obtain
\begin{align}
\max_{i} \|\phi Y_i G_{ii} \|_{L_r} \le &\frac{2K r^{\frac{1}{\alpha}}}{\sqrt{np}} + \frac{2}{n\eta} + 4 C_\alpha K^2 \max\left\{ \frac{r^{\max\{\frac{2}{\alpha}, 2\} } }{np}, \frac{r}{\sqrt{n\eta}} \right\}\\
&+ 4 C_\alpha K^2 \max\left\{\frac{r^{\max\{\frac{2}{\alpha},1\}}}{np}, \frac{\sqrt{r}}{\sqrt{np}} \right\}.
\end{align}
By our supposition $r^{\max\{\frac{1}{\alpha},1 \}} \le \sqrt{np}$, we bound $$\frac{r^{\max\{\frac{2}{\alpha}, 2\} } }{np} \le \frac{r^{\max\{\frac{1}{\alpha}, 1\} } }{\sqrt{np}}.$$
Since $r\ge 2$, $C_\alpha\ge 1$, $K\ge 1$ and $n\eta \ge 1$, 
we obtain a simplified bound  
\begin{align}
\max_{i} \|\phi Y_i G_{ii} \|_{L_r} \le 12 C_\alpha K^2 \max\left\{ \frac{r^{\max\{\frac{1}{\alpha}, 1\} } }{\sqrt{np}}, \frac{r}{\sqrt{n\eta}} \right\}
\end{align}
and this completes the proof of \eqref{eq:lemkey3}.
\end{proof}

Now we proceed with the proof of Theorem \ref{thm:localmain}. Since $G_{ij}$, $s$, $m$ and $\zeta$ are $n^2$-Lipschitz on $\mathbf S$, it suffices to prove the following result for the spectral parameter $z = E + \ii \eta$ (with fixed $E$) being in the lattice set $\{z_0, z_1,\ldots, z_{\mathcal K} \}$, where $z_k:= E + \ii \eta_k$ with $\eta_k:= 1- k n^{-3}$. Here $\mathcal K:=\max\{k\in \mathbb N : 1- k n^{-3} \ge n^{-1}\}$ and $\mathcal K\le n^3 -n^2$. Recall $\zeta=\zeta(n,r,p,\eta,\alpha)$ from \eqref{eq:zeta}. We denote $\zeta(n,r,p,\eta_k,\alpha):= \zeta_k$ below. 
\begin{prop}\label{prop:local}
Assume $n\ge 10$, $r\ge 2$ and $r^{\max\{\frac{1}{\alpha},1 \}} \le \sqrt{np}$. For any $t\ge 1$ satisfying $20\zeta_k t \le 1$ for every $k=0,1,\ldots,\mathcal K$, we have
\begin{align}\label{localk}
\Prob\left(\max_{i,j}|G_{ij}(z_k) - m(z_k) \delta_{ij}| > 10 \zeta_k t \right) \le 3 k n^3 t^{-2r}.
\end{align}
\end{prop}
The proof of Proposition \ref{prop:local} follows the standard stochastic continuity argument as seen in \cite{BGK16, HKM19}. For completeness, we include the proof in Appendix \ref{app:local}.

The proof of Theorem \ref{thm:localmain} follows from Proposition \ref{prop:local}, the Lipschitz continuity of $G_{ij}$ and $m$, $\zeta$ is a decreasing function of $\eta$, and $\frac{1}{(n\eta)^2}\le \zeta t$, and a union bound on the lattice set, following the same arguments as in \cite{HKM19}. We omit the details.

\subsection{Proof of Theorem~\ref{thm:vec_norm}}

    Note that $\|BX\|^2 = X^TB^TBX$. To apply Theorem~\ref{cor:main} to matrix $A=B^TB$ and random vector $X$, note that $(B^TB)_{ij} = \langle B_i,B_j\rangle$ and the expectation can be computed as 
    \[\E \|BX\|^2 = \E X^TB^TBX = \sum_{i=1}^n (B^TB)_{ii} \E X_i^2 = p \tr(B^TB) = p\|B\|_F^2.\]
    Moreover, the relevant matrix norms can be expressed as
    \begin{align*}
        p^2\sum_{i\neq j} (B^TB)_{ij}^2 + p\sum_{i=1}^n (B^TB)_{ii}^2
        & \leq  ~  p^2\|B^TB\|_F^2 + p\sum_{i=1}^n \|B_i\|^4,\\
        \gamma_{1,\infty}(B^TB) & = \max_{1\leq i\leq n} \Big\{p\sum_{j\neq i} |\langle B_i,B_j\rangle|, \|B_i\|^2\Big\},\\
        \|B^T B\|_{\max} &= \max_{i,j} |\langle B_i,B_j\rangle|=\|B^T\|_{2,\infty}^2.
    \end{align*}
 Hence, applying the tail bound in Theorem~\ref{cor:main} with 
    \begin{align*}
        t =& \;\sqrt{s} \cdot K^2 \sqrt{p^2\|B^T B\|_F^2 + p\sum_{i}\|B_i\|^4}+ s \cdot K^2 \max_{1\leq i\leq n}\Big\{ p\sum_{j\neq i}|\langle B_i,B_j \rangle|,\|B_i\|^2 \Big\}\\
        & + s^{\frac{2}{\alpha}} \cdot K^2 \|B^T\|_{2,\infty}^2,
    \end{align*}
    we have 
    \[\Prob\left(|\|BX\|^2 -{p}\|B\|_F^2| > t \right) \leq C'\exp(-cs).\]

\subsection*{Acknowledgments}
We thank Fanny Augeri, Antti Knowles, Stanislav Minsker, Roman Vershynin, and Shuheng Zhou for helpful discussions.
K. Wang is partially supported by Hong Kong RGC grant GRF 16304222 and ECS 26304920. Y. Zhu was partially supported by  NSF-Simons Research Collaborations on the Mathematical and Scientific
Foundations of Deep Learning, an AMS-Simons Travel Grant, and the Simons Grant MPS-TSM-00013944. This material is based upon work supported by the Swedish Research Council under grant no. 2021-06594 while Y. Zhu was in residence at Institut Mittag-Leffler in Djursholm, Sweden during the Fall of 2024.

\bibliographystyle{plain}
\bibliography{EJP-Apr17-2025/HWineq}

\appendix
\section{Auxiliary results}\label{sec:appendix}
\subsection{Auxiliary lemmas}
\begin{lemma}[Gibbs' inequality \cite{mackay2003information}]
\label{lem:gibbs}
    Suppose that $(p_1,\dots, p_n), (q_1,\dots, q_n)$ are discrete probability measures, i.e., $0\leq p_i,q_i\leq 1$, $\sum_{i=1}^n p_i=\sum_{i=1}^n q_i=1$. Then 
    \begin{align}
        -\sum_{i=1}^n p_i\log (p_i) \leq -\sum_{i=1}^n p_i\log q_i.
    \end{align}
\end{lemma}

\begin{lemma} 
    \label{lem:seq_prod_upper_bound}
    For positive $1\leq a\leq b$, there is \[a^a b^b < (a-1)^{(a-1)}(b+1)^{(b+1)}.\]
    In general, if we have positive integers $L$ and  $x_1, x_2,\dots, x_k\geq L\geq1$ with $x_1+\dots+x_k = S \geq kL$, then the maximal of $\prod_{i=1}^k x_i^{x_i}$ is taken when all but one $x_i$'s are $L$, i.e.,
    \[\prod_{i=1}^k x_i^{x_i} \leq L^{(k-1)L}(S-(k-1)L)^{(S-(k-1)L)}.\]
\end{lemma}
\begin{proof}
    We first prove the first inequality, which is equivalent to show that 
    \[\frac{a^a}{(a-1)^{a-1}}< \frac{(b+1)^{b+1}}{b^b}.\]
    Note that the function $f(x) = (x+1)(1+\frac{1}{x})^{x}$ is strictly in creasing on $(0,+\infty)$. We know the inequality holds as $f(a-1)< f(b)$.

    In general, suppose $x_1,\dots,x_k$ maximize $\prod_{i=1}^k x_i^{x_i}$ and there are at least two $x_i, x_j, i\neq j$ such that $L< x_i \leq  x_j$. We replace $x_i, x_j$ with $x_i-1, x_j+1$. Then the constraints still hold, and by the first part of the lemma we obtain a larger value of $\prod_{i=1}^k x_i^{x_i}$, which leads to a contradiction. Therefore, the maximum is taken where all but one $x_i$'s are $L$.
\end{proof}

\subsection{The Lambert W function}\label{app:lambert}
For real numbers $x$ and $y$, the equation \[y e^y=x\] has real solutions for $y$ only when $x\ge -e^{-1}$. When $x\ge 0$, there exists a unique solution given by the principal branch of the Lambert W function, denoted as $y=W_0(x)$. For $-e^{-1}\le x<0$, there are exactly two solutions: the principal branch $y=W_0(x)$ and the lower branch $y=W_{-1}(x)$. The two branches of the Lambert W function are illustrated in Figure \ref{fig:lambert}.
We include a result about the Lambert W function that is used in our proof.
\begin{figure}[!ht]
 \begin{center}
   \includegraphics[width=7cm]{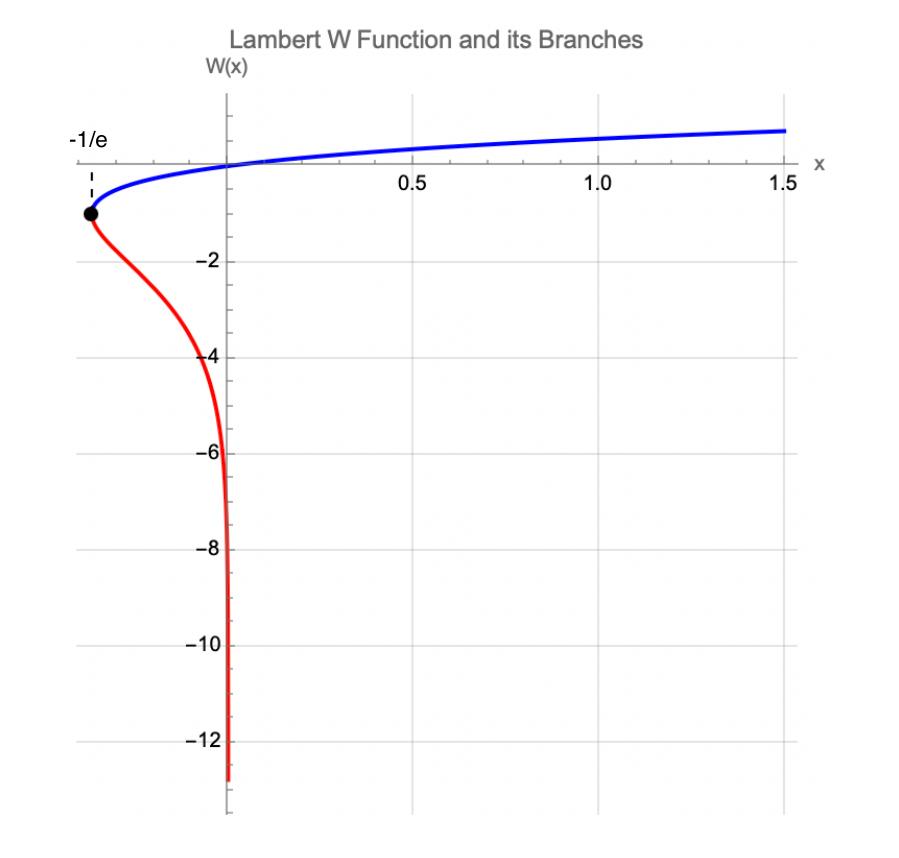}
   \caption{The blue curve is the principal branch $W_0(x)$ and the red curve is the lower branch $W_{-1}(x)$.}
   \label{fig:lambert}
 \end{center}
\end{figure}

\begin{prop}\label{prop:lambert}For positive $u$ and $L$, the inequality $e^u \le L u$ is solvable only when $L \ge e$ and the solution is 
    \[ -W_{0}(-L^{-1}) \le  u \le -W_{-1}(-L^{-1}).\]
\end{prop}
\begin{proof}The inequality $e^u \le L u$ is equivalent to
\begin{align}\label{eq:lambert}
(-u) e^{-u} \le - L^{-1}. 
\end{align}
This inequality \eqref{eq:lambert} is solvable only when $- L^{-1} \ge -e^{-1}$ since the minimum of $x e^x$ is $-e^{-1}$. Specially, if $L=e$, then $u=1$. For $L\ge e$, $-L^{-1} \in [-e^{-1},0)$ and from the definition and behavior of Lambert W function (see Figure \ref{fig:lambert}), we find that $W_{-1}(-L^{-1}) \le - u \le W_{0}(-L^{-1})$. This completes the proof.
\end{proof}

\section{Proofs of Theorem \ref{thm:delo} and Theorem \ref{thm:delo-new}}\label{app:delo-proof}
\subsection{Proof of Theorem \ref{thm:delo}}\label{app:delo}
We begin the proof of Theorem \ref{thm:delo} by showing that, with high probability, the eigenvalues of $H$ lie within the spectral domain $\mathbf S$. This is done by the spectral norm concentration inequality from \cite{latala2018dimension} with a truncation argument.

\begin{lemma}\label{lem:bdH}
Consider the matrix $H$ in \eqref{eq:H}.
  There is an absolute constant $C>0$ such that for any $D>0$, if $p\geq \frac{K^2e^{2(D+2)}(\log n)^{1+\frac{2}{\alpha}}}{n}$, with probability at least $1-2n^{-D}$,
    \begin{align}
        \|\widetilde{ H}\| \leq C(1+\sqrt{D})\sqrt{np}.
    \end{align}
\end{lemma}
\begin{proof}
Since the entries of $\widetilde{H}$ satisfy $\E|\widetilde{H}_{ij}|^r \le K^r r^{\frac{r}{\alpha}}$   for all $r\ge 1$, applying Markov's inequality to $|\widetilde{H}_{ij}|^{r}$ with $r=\log n$ and  a union bound yields
\begin{align}
   & \Prob\left(|\widetilde{H}_{ij}|>  e^{D+2} K (\log n)^{\frac{1}{\alpha}} \right) \le n^{-2-D},\\
&\Prob\left(\|\widetilde{H}\|_{\infty} >  e^{D+2} K (\log n)^{\frac{1}{\alpha}} \right) \le n^{-D}.
\end{align}
Denote $B:= e^{D+2} K (\log n)^{\frac{1}{\alpha}}$. Define $\check{H}_{ij}=\widetilde{H}_{ij} \mathbf{1} \{|\widetilde{H}_{ij}|\leq B \}$. Then $|\check{H}_{ij}-\E \check{H}_{ij}| \leq 2B$.  
Since $\xi_{ij}$ is centered and has variance 1, for all $i,j$,
\[\E(\check{H}_{ij}-\E \check{H}_{ij})^2=\Var(\check{H}_{ij})\leq \Var(\widetilde H_{ij})\leq p. \]   
Since $\check H$ has independent bounded entries, 
\cite[Theorem 4.8]{latala2018dimension} implies 
\[ \|\check{H}-\E\check{H}\|\leq C\left(\sqrt{np}+ B\sqrt{\log n}+t \right) \]
with probability at least $1-\exp(-\frac{t^2}{CB^2})$. Taking $t=\sqrt{CD\log n} B$ implies with probability at least $1-n^{-D}$,
$ \|\check{H}-\E\check{H}\|\leq C(1+\sqrt D) (\sqrt{np}+ B\sqrt{\log n})$. On the other hand, since $\E\check{H}$ is of rank-1, 
\begin{align}
    \| \E\check{H}\|&=n |\E \check{H}_{ij}|=n|\E \check{H}_{ij}-\E \widetilde{H}_{ij}|=n |\E \widetilde{H}_{ij} \mathbf{1} \{|\widetilde{H}_{ij}|> B \}|\\
    &\leq n\sqrt{\E\widetilde{H}_{ij}^2}\sqrt{\Prob(|\widetilde{H}_{ij}|>B) }\leq C n\sqrt{p} n^{-1-D/2}\leq C\sqrt{np}.
\end{align}
Therefore with probability at least $1-2n^{-D}$,
\begin{align}
    \| \widetilde{H}\| =\|\check{H}\|&\leq \| \check{H}-\E \check{H}\| +\| \E \check{H}\|\\
    &\leq C (1+\sqrt{D})(\sqrt{np}+Ke^{D+2} (\log n)^{\frac{1}{2}+\frac{1}{\alpha}})\leq 2C (1+\sqrt{D})\sqrt{np},
\end{align}
where in the last inequality, we use the assumption that $p\geq \frac{K^2e^{2(D+2)}(\log n)^{1+2/\alpha}}{n}$. 
\end{proof}

We work on the conclusion of Lemma \ref{lem:bdH}. This result implies that for each eigenvalue $\lambda_i$ of $H$, $z_i := \lambda_i + \ii \eta$, where $\eta = L C_\alpha^2 K^4 \frac{\log^2 n}{n}$, lies in a spectral domain $$\hat{\mathbf S}\equiv \hat{\mathbf S}(D):=\{ E + \ii \eta \in \mathbb{C} : |E|\le C(1+\sqrt{D}), n^{-1} < \eta \le 1 \}.$$ The domain $\hat{\mathbf S}$ is a slight modification of the spectral domain $\mathbf S$ defined in \eqref{eq:S}. An inspection of the proof of Theorem \ref{thm:local} reveals that there exist constants $L \equiv L(\delta, D)$ and $n_0 \equiv n_0(\delta,D)$ such that for
$p \ge L C_\alpha^2 K^4 \frac{(\log n)^{\max\{2,\frac{2}{\alpha} \}}}{n}$, we have
\begin{align}\label{eq:entryG01}
\Prob\left( \max_{i,j} \left| G_{ij}(z) - m(z) \delta_{ij} \right| \le \delta \right) \ge 1- n^{-D}.
\end{align}
for any $z=E + \ii \eta \in \hat{\mathbf S}$ with $\eta \ge L C_\alpha^2 K^4 \frac{\log^2 n}{n}$, provided $n\ge n_0$. The only modification needed in the proof occurs in the final step, where we take a union bound over the lattice set $\hat{\mathbf S} \cap n^{-3} {\mathbb Z}^2$. The size of this lattice set depends on $D$, and we omit the detailed analysis. 

Therefore, taking $\delta=1/2$ in \eqref{eq:entryG01} and using the bound $|m|\le 1$, we have
\begin{align}
    \frac{3}{2} \ge \Im G_{kk}(z_i) = \sum_{j=1}^n \frac{\eta}{(\lambda_j -\lambda_i)^2 + \eta^2} |u_j(k)|^2 \ge \frac{1}{\eta} |u_i(k)|^2,
\end{align}
which yields $|u_i(k)| \le \sqrt{{2}{\eta}}  = \sqrt{2L} C_\alpha K^2 \frac{\log n}{\sqrt n}$. This completes the proof.

\subsection{Proof of Theorem \ref{thm:delo-new}}\label{app:delo-new}

For $\alpha\geq 2$, we provide a stronger spectral norm concentration bound. The proof is based on a general result from \cite{brailovskaya2024extremal} for the spectral norm of inhomogeneous random matrices, where each entry has a symmetric distribution. We apply a symmetrization argument to obtain the following bound: 

\begin{lemma}\label{lemma:norm_concentration}
     Let $X$ be an $n\times n$ real symmetric matrix with $X_{ij}=b_{ij}\xi_{ij}$  and $X_{ii}=\sqrt{2}b_{ij}\xi_{ii}$. Here $b_{ij}\geq 0$ and $\xi_{ij}$ are independent  real random variables for $i\geq j$ with $\E \xi_{ij}^{2p} \leq \E g^{2p}$ for all $p\in \mathbb N$, where $g\sim N(0,1)$. Let 
    $  \sigma^2=\max_{i} \sum_{j} b_{ij}^2, \sigma_{*}^2=\max_{ij} b_{ij}^2$.  Then \begin{align}\label{eq:symmetric_concentration}
        \mathbb P\left(\|X\| \geq 4\sqrt{\sigma^2+\sigma_*^2}+\sigma_*t \right) \leq 2n e^{-\frac{t^2}{16}}.
        \end{align}
\end{lemma}
\begin{proof}
   Let $X'$ be a symmetric random matrix such that $X'_{ij}=\varepsilon_{ij}X_{ij}$, where $\varepsilon_{ij}$ are i.i.d. Rademacher random variables.  Then $X'_{ij}$ has a symmetric distribution. From \cite[Exercise 6.4.5]{vershynin2018high} for any $t\geq 0$,
$\E e^{t\|X\|} \leq  \E e^{2t\|X'\|}$. 
From \cite[Lemma 3.17]{brailovskaya2024extremal}, assuming $\sigma_*=1$, and let $d=\lceil \sigma^2\rceil,$ we have
$\E \tr e^{tX'}\leq e^{2\sqrt{d}  t+t^2}$.
Since $e^{t\|X'\|} \leq n\tr e^{tX'} +n\tr e^{-tX'}$, we obtain $\E[e^{t\|X'\|}]\leq 2ne^{2\sqrt{d}  t+t^2}$. Then $\E e^{t\|X\|}\leq  2n e^{4\sqrt{d} t+4t^2}$. By Markov's inequality, for any $\varepsilon\geq 0$,
\begin{align}
    \mathbb P \left( \|X\| \geq 4\sqrt{\sigma^2+1}+ \varepsilon  \right)\leq 2n e^{-\frac{\varepsilon^2}{16}}.
\end{align}
Applying the inequality above to $\frac{1}{\sigma_*}X'$, we obtain \eqref{eq:symmetric_concentration}.
\end{proof}

With Lemma~\ref{lemma:norm_concentration}, we obtain the following spectral norm concentration result for $\widetilde{H}$.
\begin{lemma}\label{lemma:new-norm} Consider the matrix $H$ in \eqref{eq:H}. Assume  $\alpha\geq 2$.
For any $c>0$,  assume $p\geq \frac{c\log n}{n}$.  For any $D>0$, there is a constant $C>0$ depending on $c,D$ such that   with probability at least $1-2n^{-D}$, $ \|\widetilde{ H}\| \leq CK\sqrt{np}$.
\end{lemma}
\begin{proof}
 Since $\alpha\geq 2$, from the moment assumption on $\xi_{ij}$ we have \[  \E (\xi_{ij})^{2p}\le K^{2p} (2p)^{\frac{2p}{\alpha}}\leq K^{2p} (2p)^p.\]
Since $\E g^{2p}=(2p-1)!!=\frac{2^p}{\sqrt \pi}\Gamma(p+1/2)$, and $\Gamma(x) \sim \sqrt{2\pi x} (x/e)^x$ as $x\to\infty$,  there is an absolute constant $C>1$ such that 
$\E \left( \frac{1}{CK} \xi_{ij}\right)^{2p}\leq \E g^{2p}$.
Let $X$ be an $n\times n$ symmetric random matrix such that $X_{ij}=x_{ij}$ where $x_{ij}\sim \mathrm{Ber}(p)$. Conditioned on $X$, $\widetilde H$ is a symmetric random matrix with a variance profile given by $X$. Since $p\geq \frac{c\log n}{n}$, by Bernstein's inequality, there is a constant $C_2>0$ depending on $D$ such that with probability $1-n^{-D}$, 
\[\max_{i} \sum_{j} X_{ij} \leq C_2 np.\]
Conditioned on the event $\mathcal E:=\{\max_{i} \sum_{j} X_{ij} \leq C_2np\}$, by choosing $C_2$ sufficiently large, we can apply Lemma~\ref{lemma:norm_concentration} to $\frac{1}{CK} \tilde H$  to get 
\begin{align}
    \mathbb P \left( \frac{1}{CK}\|\tilde H\|  \geq 3C_2 \sqrt{np}  \mid \mathcal E\right) \leq n^{-D}.
\end{align}
Therefore with probability at least $1-2n^{-D}$, $\|\widetilde{H}\| \leq C'K\sqrt{np}$ for some constant $C'$ depending only on $c, D$.
\end{proof}

The proof of Theorem \ref{thm:delo-new} follows a similar approach to that of Theorem \ref{thm:delo}. The proof of Theorem~\ref{thm:delo-new} follows a similar approach to that of Theorem~\ref{thm:delo};
the key difference is that, while Theorem~\ref{thm:delo} uses only the baseline concentration bounds (Theorems~\ref{thm:sparse_alpha} and~\ref{thm:linear}), the proof of Theorem~\ref{thm:delo-new} additionally relies on the Bennett-type refinements (Theorems~\ref{thm:linearbetter} and~\ref{thm:quadbetter}) via the improved moment estimates in Lemmas~\ref{coro:linearbetter} and~\ref{lem:momentquad-1}, which are sharper in the large-deviation regime. Concretely, these improved inputs yield an additional logarithmic gain (visible through the $[\log(\cdot)]^{-\beta}$ factors in Lemma~\ref{lem:newerror-1}), which strengthens the large-deviation
error parameters and reduces the sparsity threshold needed to close the self-consistent estimates
from $p\gtrsim \log^2 n/n$ (the regime of Theorem~\ref{thm:delo}) to $p\gtrsim \log n/n$, enabling
Theorem~\ref{thm:delo-new}. For brevity, we present only the essential estimates, omitting the full derivation.

First, we have the analog of Theorem \ref{thm:local} as the following local law that holds up to the supercritical regime for $p$. Recall that we denote $\beta = 1-\frac{1}{\alpha}$. 
\begin{theorem}[Local law for sparse Wigner matrix with sub-gaussian entries]\label{thm:local-new}
    Consider the matrix $H$ in \eqref{eq:H} and its Green function $G(z)$. Assume $\alpha\ge 2$. For any $D>0$ and $\delta,\varepsilon\in (0,1)$, there exist $L \equiv L(\delta, D,\varepsilon)$ and $n_0\equiv n(\delta,\alpha,K,D,\varepsilon)$ such that the following holds for $n\ge n_0$.  Let $  p\ge L C_\alpha^2 K^4 \frac{\log n}{n}$. For $z=E+\ii \eta \in \mathbf S$ where $\eta \ge  n^{-1+\tau}$ with $\tau \ge \varepsilon(\log n)^{-1+\frac{1}{2\beta}}$, we have
\begin{align}\label{eq:entryG02}
\Prob\left( \max_{i,j} \left| G_{ij}(z) - m(z) \delta_{ij} \right| \le \delta \right) \ge 1- n^{-D}.
\end{align}
Here, $C_\alpha\ge 1$ is an absolute constant depending only on $\alpha$.
\end{theorem}
Recall $\phi$ from \eqref{def:phi}. The proof of Theorem \ref{thm:local-new} relies on the next key estimates analogous to Lemma \ref{lem:newerror}. These are derived using Lemma \ref{lem:linear_moment}, Lemma \ref{coro:linearbetter}, and Lemma \ref{lem:momentquad-1}, following a similar approach to the proof of Lemma \ref{lem:newerror}. 
\begin{lemma}\label{lem:newerror-1}
Assume $r \ge 2$ and $\alpha \ge 2$. Denote $\beta = 1- \frac{1}{\alpha}$. We have
\begingroup
    \allowdisplaybreaks
\begin{align*}
&\max_{i\neq j} \|\phi G_{ij}\|_{L_r} \le 4 C_\alpha K \max\left\{ \frac{\sqrt{r}}{\sqrt{n \eta}}, \frac{r}{\sqrt{np}} \min\left\{1, \frac{1}{[\log(\frac{2r\eta}{p})_+]^{\beta}} \right\} \right\}:=\zeta_1,\\
&\max_{i,j\neq k} \|\phi(G_{ij} - G_{ij}^{(k)}) \|_{L_r} \le 4 C_\alpha K \max\left\{ \frac{\sqrt{r}}{\sqrt{n \eta}}, \frac{r}{\sqrt{np}} \min\left\{1, \frac{1}{[\log(\frac{2r\eta}{p})_+]^{\beta}} \right\} \right\}=\zeta_1,\\
&\max_{i} \|\phi Y_i G_{ii} \|_{L_r} \le 6 C_\alpha K^2 \max\left\{ \frac{r}{\sqrt{n\eta}}, \frac{r^2}{np} \min\left\{ 1, \frac{1}{[\log(\frac{r\sqrt{\eta}}{p\sqrt{n}})_+]^{2\beta}} \right\}\right\}\\
&\qquad\qquad \qquad\qquad+ 6 C_\alpha K^2 \max\left\{ \frac{\sqrt{r}}{\sqrt{n p}}, \frac{r}{np}\right\}:=\zeta_2.
\end{align*}
\endgroup
\end{lemma}
Specifically, $\zeta_1$ and the first term of $\zeta_2$ are obtained using the improved moment bounds from Lemma \ref{coro:linearbetter} and Lemma \ref{lem:momentquad-1}, while the second term of $\zeta_2$ follows from Lemma \ref{lem:linear_moment}.

With Lemma \ref{lem:newerror-1}, we can prove the analogue of Theorem \ref{thm:localmain}, following the same line of its proof. Define the fundamental error parameter $\widetilde{\zeta}:=\max\{\sqrt{\zeta_1},\sqrt{\zeta_2}\}$ where $\zeta_1,\zeta_2$ are given in Lemma \ref{lem:newerror-1}.
\begin{prop}\label{thm:localmain-1}
Under the same assumption as in Theorem~\ref{thm:local-new}.
Let $z\in \mathbf S$, $r\in \mathbb N$ and $p\in (0,1)$. Assume $n\ge 10$ and $r\ge 2$. For any $t\ge 1$ such that $20 t\widetilde{\zeta} \le 1$, we have
\[ \Prob \left( \max_{i,j} |G_{ij}(z) - m(z) \delta_{ij}| > 12 t\tilde\zeta \right)\le 3 n^6 t^{-2r}.\]
\end{prop}
Since Theorem \ref{thm:local-new} is valid for $p\gtrsim \frac{\log^2 n}{n}$ as established in Theorem \ref{thm:local}, it suffices to prove Theorem \ref{thm:local-new} for the regime that $\log n \lesssim np \le \log^3 n$ (say). Now we apply Proposition \ref{thm:localmain-1} with $r=\log n$, $t=\exp(\frac{D+7}{2})$ and $\eta=n^{-1+\tau}$ with $\tau=\varepsilon(\log n)^{-1+\frac{1}{2\beta}}$. Denote $g(D,\delta):= \max\{\frac{12^2 t^2}{\delta^2},{20^2 t^2}\}$. Note that $\frac{2r\eta}{p}>1$ and $\frac{r\sqrt{\eta}}{p\sqrt{n}}>1$ for sufficiently large $n$. It is straightforward to verify that as long as $$p\ge \frac{40C_\alpha^2 K^4}{\varepsilon^{2\beta}}g(D,\delta)^2 \frac{\log n}{n},$$
we have $12 t\widetilde{\zeta} \le \delta$ and $20t\widetilde{\zeta} \le 1$ in Proposition \ref{thm:localmain-1} for $n$ sufficiently large. Theorem \ref{thm:local-new} follows immediately.

Consequently, Theorem \ref{thm:delo-new} can be proved following the same approach used for Theorem \ref{thm:delo}, incorporating Theorem \ref{thm:local-new} and Lemma \ref{lemma:new-norm}. The details are omitted.

\section{Proof of Proposition \ref{prop:local}}\label{app:local}
We now prove Proposition \ref{prop:local}, following closely the proof of Proposition 4.5 from \cite{HKM19}. We outline the main steps below. First, we recall the stability estimate from Lemma 4.4 of \cite{HKM19}: For $z\in \mathbf S$, let $m,\tilde{m}$ be the solutions of the equation $x^2 + zx +1 =0$. If $s$ satisfies $s^2 + zs +1=r$, then
\begin{align}\label{eq:stable}
    \min\{ |s-m|, |s-\tilde{m}| \} \le \sqrt{|r|}.
\end{align}
From Lemma \ref{lem:newerror} and the definition of $\zeta_k$, we have
\begin{align}\label{eq:properr}
    \max\left\{ \max_{i\neq j} \| \phi G_{ij}(z_k)\|_{L_r}, \max_{i,j\neq k} \|\phi(G_{ij} - G_{ij}^{(k)})(z_k) \|_{L_r}, \max_{i} \|\phi Y_i G_{ii}(z_k) \|_{L_r} \right\} \le \zeta_k^2.
\end{align}

Define the events
\[ \Omega_k:= \{|s(z_k) -m(z_k)| \le \zeta_k t\}\quad\text{and}\quad \Xi_k:=\{\Gamma(z_k) \le 3/2\}\]
for $k=0,1,\ldots,\mathcal K$. We proceed with an induction on $k$.

For $k=0$, we have $\eta_0=1$ and $\Gamma(z_0)\le 1$ using the trivial bound $\|G(z_0)\| \le 1$. Hence, $\phi(z_0)=1$ and $\Prob(\Xi_0)=1$. From the $1+z s + s^2 = \frac{1}{n} \sum_{i}G_{ii} Y_i$ and \eqref{eq:properr}, we obtain
\[\| 1+z_0 s(z_0) + s(z_0)^2 \|_{L_r} \le \max_i \|\phi Y_i G_{ii}(z_0)\|_{L_r} \le \zeta_0^2.\]
By Markov's inequality, we further get
\[ \Prob\left(|1+z_0 s(z_0) + s(z_0)^2| > \zeta_0^2 t^2 \right) \le t^{-2r}.\]
It follows from \eqref{eq:stable} that 
\begin{align}\label{eq:minsm}
    \Prob\left( \min\{ |s(z_0)-m(z_0)|, |s(z_0)-\tilde{m}(z_0)| \}> \zeta_0 t \right)\le t^{-2r}. 
\end{align}
By Vieta's formula, $m(z)+\tilde m(z)=-z$.  Therefore, for $z_0=E+i$, 
\[\Im \tilde m(z_0) = -1 - \Im m(z_0) < -1.\]
From $\Im \tilde{m}(z_0)<-1$, $\Im s>0$ and $\zeta_0 t <1/10$, we conclude that
\begin{align}\label{eq:k01}
  \Prob(\Omega_0^c) = \Prob\left( |s(z_0)-m(z_0)|> \zeta_0 t \right)\le t^{-2r}. 
\end{align}
Likewise, \eqref{eq:properr}, together with  Markov's inequality, yields
\begin{align}\label{eq:k02}
    \max_i \Prob(|Y_i(z_0) G_{ii}(z_0)|>\zeta_0 t) \le t^{-2r}\quad\text{and}\quad \max_{i\neq j} \Prob(|G_{ij}(z_0)| > \zeta_0 t) \le t^{-2r}.
\end{align}
These estimates, combined with $|G_{ii} - m| \le |(s-m) G_{ii}| + |Y_i G_{ii}|$ (see (4.29) from \cite{HKM19}), lead to
\begin{align}\label{eq:k03}
    \max_i \Prob(|G_{ii}(z_0) - m(z_0)| > 2 \zeta_0 t) \le 2 t^{-2r}.
\end{align}
Hence, by a union bound, we get
\begin{align}\label{eq:k04}
    \Prob\left( \max_{i,j} |G_{ij}(z_0) - m(z_0)\delta_{ij}|> 2 \zeta_0 t\right) \le (n^2+2n) t^{-2r}< 3n^2 t^{-2r}.
\end{align}

For $k\ge 1$, we show that conditioning on $\Omega_{k-1}\cap \Xi_{k-1}$, the event $\Omega_{k}\cap \Xi_{k}$ holds with high probability. Define a threshold index $\widetilde{\mathcal K}:=\min\{ k \in \mathbb N: k\le \mathcal K, |m(z_k) -\tilde{m}(z_k)| \le 5 \zeta_k t\}.$ For $k< \widetilde{\mathcal K}$, $|m(z_k) -\tilde{m}(z_k)| > 5 \zeta_k t$, and we must be sure that $s(z_k)$ is close to $m(z_k)$ rather than $\tilde{m}(z_k)$. In contrast, for $k\ge \widetilde{\mathcal K}$, the distinction between $m(z_k)$ and $\tilde{m}(z_k)$ becomes insignificant, leading to a simpler argument. The majority of the proof follows steps analogous to those used for the base case $k=0$. We provide a sketch of the key arguments below. 

\emph{Case I. } For $1\le k < \widetilde{\mathcal K}$, since $\Gamma(z)$ is $n^2$-Lipschitz, we have $\phi(z_k)=1$ on $\Xi_{k-1}$. Following the derivation of \eqref{eq:minsm}, we obtain
\[\Prob\left( \phi(z_k) \min\{ |s(z_k)-m(z_k)|, |s(z_k)-\tilde{m}(z_k)| \}> \zeta_k t \right)\le t^{-2r}.  \]
Since $\zeta$ is a decreasing function of $\eta$ and $m,s$ are $n^2$-Lipschitz, on $\Omega_{k-1}$ we have 
\[|s(z_k)-m(z_k)| \le |s(z_{k-1})-m(z_{k-1})| + \frac{2n^{-3}}{\eta^2} \le \zeta_{k-1} + \frac{2}{n^3\eta} < 2 \zeta_k t.\]
This implies that $\min\{ |s(z_k)-m(z_k)|, |s(z_k)-\tilde{m}(z_k)| = |s(z_k)-m(z_k)|$.
Therefore,
\[\Prob(\Omega_{k-1}\cap \Xi_{k-1} \cap \Xi_k^2) \le t^{-2r}. \]
Similarly to \eqref{eq:k02} and \eqref{eq:k03}, the following estimates hold:
\begin{align}
    &\max_{i} \Prob(\Omega_{k-1} \cap \Xi_{k-1} \cap \{|G_{ii}(z_k) - m(z_k)| > 2 \zeta_k t \}) \le 2 t^{-2r},\\
    &\max_{i\neq j} \Prob( \Xi_{k-1} \cap \{|G_{ij}(z_k)| >\zeta_k t \} ) \le t^{-2r},\\
    &\max_{i,j\neq l} \Prob( \Xi_{k-1} \cap \{ |G_{ij}(z_k) - G_{ij}^{(l)}(z_k) | > \zeta_k t \}) \le t^{-2r}. 
\end{align}
A union bound, together with the estimate $|m|\le 1$, implies that
\[\Prob(\Omega_{k-1} \cap \Xi_{k-1} \cap \{ \Gamma(z_k) > 1+ 2\zeta_k t\}) \le (2n + n^2 + n^3) t^{-2r} \le 2n^3 t^{-2r}.\]
Noting that $\zeta_k t \le 1/20$, we obtain
\[\Prob(\Omega_{k-1} \cap \Xi_{k-1} \cap \Xi_k^c ) \le 2n^3 t^{-2r}\] and consequently,
\[ \Prob\left(\Omega_{k-1} \cap \Xi_{k-1} \cap (\Omega_k \cap \Xi_k)^c \right) \le (2n^3+1) t^{-2r}.\]
Since $\Prob((\Omega_0 \cap \Xi_0)^c) \le t^{-2r}$, a standard induction argument (see, for instance, (5.17) from \cite{BGK16}) shows that 
\[\Prob((\Omega_k \cap \Xi_k)^c) \le t^{-2r} + k(2n^3 +1) t^{-2r}.\]
Finally, similar to \eqref{eq:k04}, we arrive at
\[\Prob\left( \max_{i,j} |G_{ij}(z_k) - m(z_k)\delta_{ij}|> 2 \zeta_k t\right) \le t^{-2r} + k(2n^3 +1) t^{-2r} + (2n+n^2) t^{-2r} \le 3 kn^3 t^{-2r}. \]

\emph{Case II. } For $\widetilde{\mathcal K} \le k \le \mathcal K$, given that $|\tilde{m}(z_k)-m(z_k)|< 5 \zeta_k t$, we have the following estimates:
\begin{align}
    &\Prob(\Xi_{k-1} \cap \{|s(z_k)-m(z_k)| > 6\zeta_k t) \le t^{-2r},\\
    &\max_{i} \Prob(\Xi_{k-1}  \cap \{|G_{ii}(z_k) - m(z_k)| > 10 \zeta_k t \}) \le 2 t^{-2r},\\
    &\max_{i\neq j} \Prob( \Xi_{k-1} \cap \{|G_{ij}(z_k)| >\zeta_k t \} ) \le t^{-2r}.
\end{align}
Since $|m|\le 1$ and $20\zeta_k t \le 1$, applying a union bound yields
\[ \Prob(\Xi_{k-1} \cap \Xi_k^c) \le (2n + n^2) t^{-2r}.  \]
Using the bound $\Prob(\Xi_{\widetilde{K}-1}^c) \le 3(\widetilde{K}-1)n^3 t^{-2r}$ from \emph{Case I} and an induction argument, we obtain
\[ \Prob(\Xi_k^c) \le 3(\widetilde{K}-1)n^3 t^{-2r} + (k-\widetilde{K}) (2n + n^3) t^{-2r}. \]
Finally, following an argument similar to \eqref{eq:k04}, we conclude
\[\Prob\left( \max_{i,j} |G_{ij}(z_k) - m(z_k)\delta_{ij}|> 10 \zeta_k t\right) 
\le 3 kn^3 t^{-2r}. \]

\section{Proofs of Corollary~\ref{cor:main} and Corollary \ref{coro:general-HW}}\label{App:B}

\subsection{Proof of Corollary \ref{cor:main}}

To obtain the Hanson-Wright inequality for a general matrix $A$, we bound the concentration of the diagonal sum $S_{\mathrm{diag}}=\sum_{i=1}a_{ii} X_i^2$ in the quadratic form $\sum_{i,j}a_{ij}X_iX_j$ by applying Theorem~\ref{thm:linear} to $X_i^2-EX_i^2$. Note that 
\begin{align}\label{eq:Ymom}
\E |X_i^2 - \E X_i^2|^l \leq 2^l \E |X_i|^{2l} \leq 2^l p_i K^{2l}(2l)^{2l}.
\end{align}
Here in the first inequality, we use the fact that for a non-negative random variable $Z$, one has \[\E(Z-\E Z)^l \leq \E [2^{l-1}(Z^l + (\E Z)^l)] \leq 2^l \E Z^l.\]
Hence, $X_i^2-EX_i^2$'s are sparse $\frac{1}{2}$-subexponential random variables. In particular, the tail bound for $S_{\mathrm{diag}}$ is given by 
    \begin{align}
        \Prob\left( \left| S_{\mathrm{diag}} -\E S_{\mathrm{diag}}  \right| 
        \ge t\right) 
        \le  \, e^2 \exp\left( -c \min \left\{ \frac{t^2}{K^4 \sum_i a_{ii}^2 p_i}, \left(\frac{t}{K^2\max_i |a_{ii}|}  \right)^{\frac{1}{2}}\right\} \right).\label{eq:diag_bernstein}
    \end{align}
Combining the bounds for the diagonal and off-diagonal terms, we prove Corollary \ref{cor:main}.

\subsection{Proof of Corollary \ref{coro:general-HW}}

In this section, we briefly discuss the extension of our results to the concentration of quadratic forms of noncentered random variables. Let $Y_1,\cdots, Y_n$ be independent random variables, and consider a symmetric, diagonal-free matrix $A$ for simplicity.  Denote $X_i=Y_i - \E(Y_i)$. We rewrite 
\begin{align}
Y^T A Y = \sum_{i\neq j} a_{ij} Y_i Y_j = \sum_{i\neq j} a_{ij} (X_i+\E(Y_i)) (X_j + \E(Y_j)).
\end{align}
Through straightforward calculation, we obtain
$$Y^T A Y - \E(Y^T A Y) = X^T A X + 2\sum_{i}\left(\sum_{j} a_{ij} \E(Y_j) \right) X_i:=X^T A X + 2\sum_{i=1}^n b_i X_i. $$
Note that $\Var(Y^T A Y)= \Var(X^T A X) + \Var(2\sum_{i}b_i X_i)$ since $X^T A X$ and $\sum_{i=1}^n b_i X_i$ are uncorrelated. 

The concentration inequality for $Y^T A Y - \E(Y^T A Y)$ can be obtained by considering those of $X^T A X$ and $\sum_i b_i X_i$ separately. Assume the moments of $Y_i$ satisfy $\E|Y_i|^r \le p_i (K r)^r$ for $r\ge 1$. Then from the inequality $\|X_i\|_{L_r} = \|Y_i - \E(Y_i)\|_{L_r} \le 2 \|Y_i\|_{L_r}$, we deduce that the $X_i$'s are independent centered random variables satisfying $\E|X_i|^r \le p_i (2K r)^r$ for $r\ge 2$. Consequently, the conclusion of Theorem \ref{thm:sparse_alpha} holds for $X^T A X$ where $\alpha=1$. Moreover, $\sum_{i=1}^n b_i X_i$ is a sum of independent random variables. Applying Theorem \ref{thm:linear} with $\alpha=1$, we find that the tail bound for $\sum_{i=1}^n b_i X_i$ is a mixture of sub-gaussian and subexponential tails. For the subexponential tail, we can use the estimate that $$\max_{i} |b_i| = \max_i | \sum_{j} a_{ij} \E(Y_j)|\le K\max_i \sum_j |a_{ij}|p_j = K\gamma_{1,\infty}.$$

The conclusion of Corollary \ref{coro:general-HW} follows by combining the results of Theorem \ref{thm:sparse_alpha} and Theorem \ref{thm:linear} in the preceding discussion.

\end{document}